\documentclass{article} 

\usepackage{amsthm} 
\usepackage{amssymb} 
\usepackage{amsmath} 
\usepackage{mathtools} 
\usepackage{graphicx} 

\usepackage{stmaryrd} 
\usepackage{tikz} 

\usepackage{hyperref} 

\hypersetup{
	colorlinks=true, %
	linkcolor=blue, 
	citecolor=blue, 
}

\usepackage{bbm} 
\usepackage{empheq} 

\usepackage[shortlabels]{enumitem} 

\usepackage{comment}

\theoremstyle{plain}
\newtheorem{thm}{Theorem}[section]
\newtheorem{prop}[thm]{Proposition}
\newtheorem{lem}[thm]{Lemma}
\newtheorem{cor}[thm]{Corollary}

\theoremstyle{definition}
\newtheorem{defi}[thm]{Definition}

\theoremstyle{remark}
\newtheorem{remark}[thm]{Remark}

\numberwithin{equation}{section}

\allowdisplaybreaks[3]

\newcommand{\N}{\mathbb{N}}
\newcommand{\Z}{\mathbb{Z}}

\newcommand{\R}{\mathbb{R}}
\newcommand{\C}{\mathbb{C}}

\newcommand{\EE}{\mathbb{E}}

\newcommand{\TT}{\mathbb{T}}

\newcommand{\EX}{\mathbb{X}}

\newcommand{\CB}{\mathcal{B}}
\newcommand{\CC}{\mathcal{C}}
\newcommand{\CD}{\mathcal{D}}
\newcommand{\CE}{\mathcal{E}}
\newcommand{\CF}{\mathcal{F}}

\newcommand{\CI}{\mathcal{I}}

\newcommand{\CL}{\mathcal{L}}
\newcommand{\CM}{\mathcal{M}}

\newcommand{\CR}{\mathcal{R}}
\newcommand{\CS}{\mathcal{S}}

\newcommand{\CX}{\mathcal{X}}

\newcommand{\pt}{\partial_t}
\newcommand{\pxi}{\partial_i}

\newcommand{\rs}{\varodot}
\newcommand{\pl}{\varolessthan}
\newcommand{\pr}{\varogreaterthan}

\newcommand{\be}{\mathbbm{e}}

\newcommand{\CTCC}[1]{C_T\CC^{#1}}

\newcommand{\drawC}[2]{
\draw[thick] #1 -- #2;
\filldraw #1 circle (0.8pt);
\filldraw #2 circle (1.6pt);}

\newcommand{\drawI}[2]{
\draw[thick] #1 -- #2;
\filldraw #1 circle (0.8pt);
\filldraw #2 circle (0.8pt);}

\newcommand{\putrs}[1]{
\filldraw[white] #1 circle (2.5pt);
\filldraw[black] #1 circle (0.4pt);
\draw[draw=black] #1 circle (2.5pt);
}

\newcommand{\drawAC}[2]{
\draw[thick] #1 -- #2;
\draw[fill=black] #1 circle (0.8pt);
\draw[draw=black,fill=white] #2 circle (1.6pt);
}

\newcommand{\IX}{{\,
\begin{tikzpicture}[baseline=0.5pt, scale=6/12]
\coordinate (O1) at (0,0);
\coordinate (X1) at (0pt,12pt);
\drawC{(O1)}{(X1)}
\end{tikzpicture}\,}}

\newcommand{\XX}{{\,
\begin{tikzpicture}[baseline=0.5pt, scale=6/12]
\coordinate (O1) at (0,0);
\coordinate (X1) at (-4pt,12pt);
\coordinate (X2) at (+4pt,12pt);
\drawC{(O1)}{(X1)}
\drawC{(O1)}{(X2)}
\end{tikzpicture}\,}}

\newcommand{\XXX}{{\,
\begin{tikzpicture}[baseline=0.5pt, scale=6/12]
\coordinate (O1) at (0,0);
\coordinate (X1) at (-6pt,11pt);
\coordinate (X2) at (+0pt,12pt);
\coordinate (X3) at (+6pt,11pt);
\drawC{(O1)}{(X1)}
\drawC{(O1)}{(X2)}
\drawC{(O1)}{(X3)}
\end{tikzpicture}\,}}

\newcommand{\ICC}{{\,
\begin{tikzpicture}[baseline=1pt, scale=6/12]
\coordinate (O1) at (0pt,0pt);
\coordinate (O2) at (0pt,7pt);
\coordinate (X1) at (-3pt,15pt);
\coordinate (X3) at (+3pt,15pt);
\drawC{(O2)}{(X1)}
\drawC{(O2)}{(X3)}
\drawI{(O1)}{(O2)}
\end{tikzpicture}\,}}

\newcommand{\ICCC}{{\,
\begin{tikzpicture}[baseline=1pt, scale=6/12]
\coordinate (O1) at (0pt,0pt);
\coordinate (O2) at (0pt,7pt);
\coordinate (X1) at (-5pt,14pt);
\coordinate (X2) at (+0pt,15pt);
\coordinate (X3) at (+5pt,14pt);
\drawC{(O2)}{(X1)}
\drawC{(O2)}{(X2)}
\drawC{(O2)}{(X3)}
\drawI{(O1)}{(O2)}
\end{tikzpicture}\,}}

\newcommand{\CICCC}{{\,
\begin{tikzpicture}[baseline=1pt, scale=6/12]
\coordinate (O1) at (0pt,0pt);
\coordinate (O2) at (0pt,7pt);
\coordinate (X1) at (-5pt,14pt);
\coordinate (X2) at (+0pt,15pt);
\coordinate (X3) at (+5pt,14pt);
\coordinate (Y2) at (+7pt,6pt);
\drawC{(O2)}{(X1)}
\drawC{(O2)}{(X2)}
\drawC{(O2)}{(X3)}
\drawI{(O1)}{(O2)}
\drawC{(O1)}{(Y2)}
\putrs{(O1)}
\end{tikzpicture}\,}}

\newcommand{\cICCc}{{
\begin{tikzpicture}[baseline=1pt, scale=6/12]
\coordinate (O1) at (0pt,0pt);
\coordinate (O2) at (0pt,7pt);
\coordinate (X1) at (-3pt,15pt);
\coordinate (X2) at (3pt,15pt);
\coordinate (X3) at (6pt,4pt);
\coordinate (Y2) at (6pt,4pt);
\drawC{(O2)}{(X1)}
\drawC{(O2)}{(X2)}
\drawAC{(O2)}{(X3)}
\drawI{(O1)}{(O2)}
\drawAC{(O1)}{(Y2)}
\putrs{(O1)}
\end{tikzpicture}}}

\newcommand{\CCICC}{{\,
\begin{tikzpicture}[baseline=1pt, scale=6/12]
\coordinate (O1) at (0pt,0pt);
\coordinate (O2) at (0pt,7pt);
\coordinate (X1) at (-3pt,15pt);
\coordinate (X3) at (+3pt,15pt);
\coordinate (Y1) at (-5pt,6pt);
\coordinate (Y2) at (+5pt,6pt);
\drawC{(O2)}{(X1)}
\drawC{(O2)}{(X3)}
\drawI{(O1)}{(O2)}
\drawC{(O1)}{(Y1)}
\drawC{(O1)}{(Y2)}
\putrs{(O1)}
\end{tikzpicture}\,}}

\newcommand{\CcICc}{{\,
\begin{tikzpicture}[baseline=1pt, scale=6/12]
\coordinate (O1) at (0pt,0pt);
\coordinate (O2) at (0pt,7pt);
\coordinate (X1) at (-3pt,15pt);
\coordinate (X3) at (+6pt,4pt);
\coordinate (Y1) at (-5pt,6pt);
\coordinate (Y2) at (+6pt,4pt);
\drawC{(O2)}{(X1)}
\drawAC{(O2)}{(X3)}
\drawI{(O1)}{(O2)}
\drawC{(O1)}{(Y1)}
\drawAC{(O1)}{(Y2)}
\putrs{(O1)}
\end{tikzpicture}\,}}

\newcommand{\CCICCC}{{\,
\begin{tikzpicture}[baseline=1pt, scale=6/12]
\coordinate (O1) at (0pt,0pt);
\coordinate (O2) at (0pt,7pt);
\coordinate (X1) at (-5pt,14pt);
\coordinate (X2) at (+0pt,15pt);
\coordinate (X3) at (+5pt,14pt);
\coordinate (Y1) at (-7pt,6pt);
\coordinate (Y2) at (+7pt,6pt);
\drawC{(O2)}{(X1)}
\drawC{(O2)}{(X2)}
\drawC{(O2)}{(X3)}
\drawI{(O1)}{(O2)}
\drawC{(O1)}{(Y1)}
\drawC{(O1)}{(Y2)}
\putrs{(O1)}
\end{tikzpicture}\,}}

\newcommand{\CcICCc}{{\,
\begin{tikzpicture}[baseline=1pt, scale=6/12]
\coordinate (O1) at (0pt,0pt);
\coordinate (O2) at (0pt,7pt);
\coordinate (X1) at (-5pt,14pt);
\coordinate (X2) at (+0pt,15pt);
\coordinate (X3) at (+7pt,4pt);
\coordinate (Y1) at (-7pt,6pt);
\coordinate (Y2) at (+7pt,4pt);
\drawC{(O2)}{(X1)}
\drawC{(O2)}{(X2)}
\drawAC{(O2)}{(X3)}
\drawI{(O1)}{(O2)}
\drawC{(O1)}{(Y1)}
\drawAC{(O1)}{(Y2)}
\putrs{(O1)}
\end{tikzpicture}\,}}

\newcommand{\ccICcc}{{\,
\begin{tikzpicture}[baseline=1pt, scale=6/12]
\coordinate (O1) at (0pt,0pt);
\coordinate (O2) at (0pt,7pt);
\coordinate (X1) at (-7pt,4pt);
\coordinate (X2) at (+0pt,15pt);
\coordinate (X3) at (+7pt,4pt);
\coordinate (Y1) at (-7pt,4pt);
\coordinate (Y2) at (+7pt,4pt);
\drawAC{(O2)}{(X1)}
\drawC{(O2)}{(X2)}
\drawAC{(O2)}{(X3)}
\drawI{(O1)}{(O2)}
\drawAC{(O1)}{(Y1)}
\drawAC{(O1)}{(Y2)}
\putrs{(O1)}
\end{tikzpicture}\,}}

\title{The dynamical $\Phi_3^4$ model by an approximation of Laplacian}
\date{}
\author{Reo Adachi \thanks{Graduate School of Engineering Science, Osaka University, 1-3 Machikaneyama, Toyonaka, Osaka 560-8531, Japan, e-mail address: deruf4357@gmail.com
}}

\begin{document}
\maketitle

\begin{abstract}
    The dynamical $\Phi^4_3$ equation is a singular SPDE and has important applications in physics.
    In this paper, we consider the equation by approximating the Laplacian instead
    of the noise or the cubic term as in previous studies.
    By using a good transformation as in \cite{JP21}, we show that the approximating equation is locally well-posed and that its solution converges to the solution of the dynamical $\Phi^4_3$ equation.
    We expect to be able to show the global well-posedness
    and the existence of the invariant measure.
\end{abstract}


\vspace{2mm}


\section{Introduction}
In this paper, we consider the dynamical $\Phi^4_3$ equation, which is formally written as
\begin{equation}
    \label{eq:intro_Phi43}
    (\pt - \Delta +1)\phi = -\phi^3 +\infty \cdot \phi + \xi, \quad (t,x) \in \R_+ \times \TT^3,
\end{equation}
where $\TT^3:=(\R/\Z)^3$ is the three-dimensional torus and $\xi$ is a space-time white noise on $\R_+\times \TT^3$.
Our main results are that the approximating equation
\begin{equation}
    \label{eq:intro_approximation}
    (\pt - \Delta + \epsilon\Delta^2 +1)\phi_\epsilon = -\phi_\epsilon^3 + (3a_\epsilon -9b_\epsilon)\phi_\epsilon + \xi, \quad (t,x) \in \R_+ \times \TT^3
\end{equation}
satisfies the $\epsilon$-uniform local well-posedness and that a sequence of solutions $\{\phi_\epsilon\}_{0< \epsilon \leq 1}$ converges to a solution $\phi$ of (\ref{eq:intro_Phi43}).

The dynamical $\Phi^4_3$ equation has three important applications in physics.
Firstly, it is a model of a scalar field that appears in the constructive quantum field theory.
Its invariant measure called $\Phi^4_3$ measure is expected to be constructed from the equation (\ref{eq:intro_Phi43}) by the theory of stochastic quantization \cite{PW81}.
Secondly, its solution is expected to describe the 3D Ising model with Glauber dynamics and Kac interactions near critical temperature \cite{GLP99}.
\cite{MW17a} showed that two-dimensional dynamical Ising-Kac model converges to the solution of dynamical $\Phi^4_2$ equation.
Finally, three-dimensional continuous phase coexistence model near a bifurcation point is descibed by the dynamical $\Phi^4_3$ models \cite{HX18}.

The dynamical $\Phi^4_3$ equation belongs to the class of singular SPDEs.
A singular SPDE is an equation containing a product of unknown functions which cannot be interpreted in the usual sense.
For the two-dimensional case, the SPDE was solved by Da Prato and Debussche by decomposing the equation into the sum of the OU process and a better remainder \cite{DPD03}.
However, for the three-dimensional case, this approach could not be extended due to the worse regularity, hence the local well-posedness of the dynamical $\Phi^4_3$ equation had been an open problem for a long time until the ground-breaking work by Hairer.
He created a new mathematical theory called the theory of ``Regularity Structure'' \cite{Hai14}, which can treat a large class of semilinear parabolic SPDEs involving the dynamical $\Phi^4_3$ equation.
Furthermore, another approach called the theory of ``paracontrolled distributions'' was created by Gubinelli, Imkeller and Perkowski \cite{GIP15}.
Catellier and Chouk \cite{CC18} showed the local well-posedness in this way.
After that, Mourrat and Weber \cite{MW17b} showed the global well-posedness using the a priori estimate and Gubinelli and Hofmanov\'{a} \cite{GH19} shows the global well-posedness on $\R_+\times\R^3$ by considering weighted Besov space.
Recently, a simpler method using Paracontrolled distributions was proposed in \cite{JP21}, where the authors used something called the inverse Cole-Hopf transform to cancel the ill terms in the equation.
This method is also used in this paper.
For the invariant measure, Albeverio and Kusuoka \cite{AK17} showed its direct construction via the tightness of the approximating measures.

In this paper, we use a different approximation from previous studies.
For example, \cite{CC18} and \cite{MW17b} approximated the noise and \cite{AK17} approximated the cubic terms.
On the other hand, we aprroximate Laplacian and show the local well-posedness and the convergence of the solution.
Although a method for the approximating Laplacian has already been done in \cite{EX21}, we were able to show these in another way by using the method in \cite{JP21}.
We state the main result.
\begin{thm}
    Fix $0<\kappa<1/8$.
    For every initial value $\phi(0)\in \CC^{-1/2-\kappa}$ and $0<\epsilon\leq 1$, there exists a random variable $T>0$ such that the equation (\ref{eq:intro_approximation}) has a unique solution $\phi_\epsilon$ on $C([0,t];\CC^{-1/2-\kappa})$ and the sequence $\{\phi_\epsilon\}_{0 < \epsilon\leq 1}$ of the solution of (\ref{eq:intro_approximation}) converges to the solution $\phi \in C([0,t];\CC^{-1/2-\kappa})$ of the (\ref{eq:intro_Phi43}) in probability for every $0<t<T$.
\end{thm}
We reformulate the above theorem more precisely in Theorems \ref{thm:LWP_of_u_epsilon} and \ref{thm:Convergence} below.
\begin{remark}
    We might be able to take the regularity of the initial value to $-2/3+\kappa$ from previous studies.
    However, we require significant modifications in order to show it.
\end{remark}
In this paper, we can only show up to the local well-posedness, but since the global well-posedness of $\Phi^4_3$ equation has already been confirmed, we expect to be able to show the global well-posedness of this approximating equation and the existence of the invariant measure.
Then we can consider the new problem of whether its invariant measures converge to the $\Phi^4_3$ measure.

The organization of the present paper is as follows.
In Section 2, we formulate the approximating equation (\ref{eq:intro_approximation}) and show that it has $\epsilon$-uniform local well-posedness and a sequence of solutions converges to the solution of the dynamical $\Phi^4_3$ equation.
In Section 3, we show the convergence of stochastic objects.
Although the approximating sequence of such objects differs from the previous ones in \cite{JP21}, the proof of convergence is essentially same.

\subsubsection*{Acknowledgement.}
The author are grateful to Professor Masato Hoshino for helpful discussions and valuable comments.

\section{Local well-posedness and convergence}
\label{sec:2}
This section is devoted to the proof of Theorem \ref{thm:LWP_of_u_epsilon} (local well-posedness) and Theorem \ref{thm:Convergence} (convergence).
In Section \ref{sec:Formulate}, we formulate the approximating equation.
In Section \ref{sec:Preliminary}, we estimates the right-hand side of the approximating equation for main theorems.
In Section \ref{sec:LWP}, we prove Theorem \ref{thm:LWP_of_u_epsilon} and in Section \ref{sec:Convergence}, we prove Theorem \ref{thm:Convergence}.

\subsection*{Notation}
We write $\mathbbm{e}_k(x):=e^{2 \pi i k \cdot x}$ for $k \in \Z^3, x \in \R^3$.
Let $\CS$ denote the space of smooth functions on $\TT^3$ and $\CS'$ denote the space of tempered distributions on $\TT^3$.
The Fourier transform $\CF u$ for $u\in \CS$ is defined by
\[
    \CF u(k):=\int_{\TT^3} \be_{-k}(x)u(x) dx
\]
and its inverse $\CF^{-1}v$ for a rapidly decreasing sequence $\{v(k)\}_{k\in \Z^3}$ is defined by
\[
    \CF^{-1}v(x):=\sum_{k\in \Z^3} v(k)\be_{k}(x).
\]

We denote by $\{\rho_j\}_{j=-1}^{\infty}$ a dyadic partition of unity (\cite{GIP15} Section A.1).
The Littlewood-Paley blocks $\{\Delta_j\}_{j=-1}^{\infty}$ are defined by
\[
    \Delta_j u:= \CF^{-1}(\rho_j \CF u).
\]
We write $\Delta_{<k}:=\sum_{j=-1}^{k-1}\Delta_j$.

For $1\leq p,q \leq \infty$ and $\alpha\in \R$, the Besov space $\CB_{p,q}^\alpha$ is defined by
\begin{align*}
    \CB_{p,q}^\alpha
    := \{ u \in \CS' ; \|u\|_{\CB_{p,q}^\alpha}:=\|(2^{j\alpha}\|\Delta_j u\|_{L^p})_{j\geq -1}\|_{\ell^q} < \infty \}.
\end{align*}
In particular, we write $\CC^\alpha:=\CB_{\infty,\infty}^\alpha$.

Let $\kappa>0$ and $0\leq \epsilon \leq 1$. We write $\CC_\epsilon^\alpha:=\{ u \in \CS'; \|u\|_{\CC_\epsilon^\alpha} <\infty \}$, where
\[
    \|u\|_{\CC_\epsilon^\alpha}
    :=\|u\|_{\CC^\alpha} + \epsilon^{1-\frac{\kappa}{4}} \|u\|_{\CC^{\alpha+2-\kappa}}.
\]
In particular, $\CC^\alpha_0=\CC^\alpha$.

Fix $T>0$.
Let $X$ be a Banach space with norm $\|\cdot \|_X$.
$C_TX$ is the space of all continuous functions from $[0,T]$ to $X$ which is equipped with the supremum norm
\[
    \|u\|_{C_TX} := \sup_{0\leq t\leq T} \|u(t)\|_{X}.
\]
For $0\leq \delta \leq 1$, $C_T^\delta X$ is the space of all $\delta$-H\"{o}lder continuous functions from $[0,T]$ to $X$ which is equipped with the seminorm
\[
    \|u\|_{C_T^\delta X} := \sup_{0\leq s < t \leq T} \frac{\|u(t)-u(s)\|_{X}}{|t-s|^\delta}.
\]
Fix $\kappa > 0$ and $0\leq \epsilon \leq 1$.
For $\alpha \in \R$ and $\beta\leq\alpha$,  we set $\CE_T^\beta \CC_\epsilon^\alpha:=\{ u \in C([0,T], \CC^\beta); \|u\|_{\CE_T^\beta \CC_\epsilon^\alpha} < \infty \}$, where
\[
    \|u\|_{\CE_T^\beta\CC_\epsilon^\alpha}:=\sup_{0\leq t\leq T}\|u(t)\|_{\CC^\beta} + \sup_{0<t\leq T}t^{\frac{\alpha-\beta}{2}}\|u(t)\|_{\CC_\epsilon^\alpha}.
\]
Note that the first term of right-hand side is independent of $\epsilon$.
In particular, $\CE_T^\alpha\CC_\epsilon^\alpha=C_T\CC_\epsilon^\alpha$.
We often represent $\CE_T^\beta\CC^\alpha:=\CE_T^\beta\CC_0^\alpha$ for simplicity.

For $u, v \in \CS'$, we define the Bony's paraproduct
\[
    u \pl v = v \pr v := \sum_{j\geq -1} \sum_{i=-1}^{j-2} \Delta_i u \Delta_j v
\]
and the resonant product
\[
    u \rs v := \sum_{|i-j|\leq 1} \Delta_i u \Delta_j v
\]
if the right-hand side makes sense.
For simplicity, we denote ``inner paraproduct'' and ``inner resonant product'' by
\[
    \nabla f \pl \nabla g:= \sum_{i=1}^3 \pxi f \pl \pxi g, 
    \quad 
    \nabla f \rs \nabla g := \sum_{i=1}^3 \pxi f \rs \pxi g,
\]
respectively.

We use the following differential operator:
\[
    \CL_\epsilon := \pt - \Delta + \epsilon \Delta^2 + 1.
\]
The heat semigroups $e^{t\Delta}$ and $e^{-\epsilon t \Delta^2}$ are given by
\[
    e^{t\Delta} u := \CF^{-1}\left(\phi(t^{\frac{1}{2}}\cdot)\CF u\right), \quad e^{-\epsilon t \Delta^2}u:= \CF^{-1}\left(\widetilde{\phi}((\epsilon t)^{\frac{1}{4}}\cdot)\CF u\right),
\]
where
\[
    \phi(x):=e^{-4\pi^2|x|^2},\quad \widetilde{\phi}(x):=e^{-16\pi^4|x|^4}.
\]
We write $P_t^\epsilon:=e^{t(\Delta-\epsilon\Delta^2-1)}$ and $\CL_\epsilon^{-1}[v](t):=\int_{0}^t P_{t-s}^\epsilon v(s) ds$.

We say that $f\lesssim_a g$ if there is a constant $C(a)>0$ depending only on some parameter $a$ such that $f\leq C(a) g$.

\subsection{Definition of driving vectors and solutions}
\label{sec:Formulate}
\subsubsection{Transformation}
For any $ 0 < \epsilon \leq 1$, we consider approximating the dynamical $\Phi^4_3$ equation on the three-dimensional torus $\TT^3$ of the following form:
\begin{empheq}[left=\empheqlbrace]{align}\label{eq:phi43}
    &\CL_\epsilon \phi_\epsilon = - \phi_\epsilon^3 + (3a_\epsilon - 9 b_\epsilon)\phi_\epsilon + \xi, &&(t,x) \in \R_+\times \TT^3 \\
    &\phi_\epsilon(0,\cdot) = \phi(0), && x \in \TT^{3}
    \label{eq:phi43_init}
\end{empheq}
where $\phi(0) \in \CC^{-1/2-\kappa}$ for some sufficiently small $\kappa > 0$, $a_\epsilon$ and $b_\epsilon$ are deterministic real numbers which diverge as $\epsilon$ tends to $0$ and $\xi$ is a two-sided space-time white noise on $\R\times \TT^3$.
Thanks to the approximation of Laplacian using $\Delta^2$, the solution $\phi_\epsilon$ is smooth enough and (2.1) makes sense.
Since we only expect that $\phi_\epsilon$ converges in $\CC^{-1/2-\kappa}$ as $\epsilon$ tends to $0$ by Schauder estimates (Proposition \ref{prop:Schauder_estimates}), we cannot expect the convergence of the cubic term $\phi_\epsilon^3$.
Therefore, the limiting equation of (\ref{eq:phi43}) is ill-defined.
Hence, by decomposing this equation, we need to derive some well-posed equation.

In what follows, we fix sufficiently small $\kappa>0$.
We consider the stochastic terms built from $\xi$.
\begin{defi}
    We call a family of functions
    \begin{multline*}
        \EX := (\IX, \XX, \ICC(0), \ICCC, \CICCC, \CCICC, \CCICCC, \CD)\\
        \in C_T\CC^{-\frac{1}{2}-\kappa} \times C_T\CC^{-1-\kappa} \times \CC^{1-\kappa} \times (C_T\CC^{\frac{1}{2}-\kappa}\cap C_T^{\frac{1}{4}-\frac{\kappa}{2}} L^\infty) \\
        \times C_T\CC^{-\kappa} \times C_T\CC^{-\kappa} \times C_T\CC^{-\frac{1}{2}-\kappa} \times C_T\CC^{-\kappa}
    \end{multline*}
    a \textit{driving vector}.
    We denote by $\CX_T^{\kappa}$ the set of all driving vectors.
    We define the norm $\|\cdot\|_{\CX_T^{\kappa}}$ by the sum of the norm of each component.
\end{defi}
\begin{remark}
    The third component $\ICC(0)$ is independent of $t$.
\end{remark}
Specifically, we consider the following driving vectors.
\begin{defi}
    \label{defi:EX_epsilon}
    Let $\xi$ be a two-sided space-time white noise on $\R\times \TT^3$ and $\{\EX_\epsilon\}_{0< \epsilon \leq 1}$ be a sequence of stochastic driving vector as defined as follows:
    \begin{enumerate}[(i)]
        \item Let $\IX_\epsilon$ be a stationary solution of
        \[
            \CL_\epsilon \IX_\epsilon = \xi.
        \]
        \item Let $a_\epsilon:=\EE[(\IX_\epsilon)^2]$, $\XX_\epsilon := (\IX_\epsilon)^2 - a_\epsilon$ and $\XXX_\epsilon := (\IX_\epsilon)^3 - 3a_\epsilon\IX_\epsilon$.
        \item Let $\ICC_\epsilon (0):=\int_{-\infty}^0 P_{-s}^\epsilon \XX_\epsilon(s) ds$.
        \item Let $\ICCC_\epsilon$ be a stationary solution of
        \[
            \CL_\epsilon \ICCC_\epsilon = \XXX_\epsilon.
        \]
        \item  Let $b_\epsilon:=\EE[\ICC_\epsilon \rs \XX_\epsilon]$ and
        \begin{align*}
            \CICCC_\epsilon &:= \ICCC_\epsilon \rs \IX_\epsilon,&\CCICC_\epsilon &:= \ICC_\epsilon\rs \XX_\epsilon - b_\epsilon, \\
            \CCICCC_\epsilon &:= \ICCC_\epsilon \rs \XX_\epsilon - 3b_\epsilon \IX_\epsilon, & \CD_\epsilon &:= \nabla \ICC_\epsilon \rs \nabla \ICC_\epsilon + \epsilon \Delta \ICC_\epsilon \rs \Delta \ICC_\epsilon - b_\epsilon.
        \end{align*}
    \end{enumerate}
\end{defi}
\begin{remark}
    We need to modify $\nabla\ICC_\epsilon \rs \nabla \ICC_\epsilon - b_\epsilon$ in \cite{JP21} to $\nabla \ICC_\epsilon \rs \nabla \ICC_\epsilon + \epsilon \Delta \ICC_\epsilon \rs \Delta \ICC_\epsilon - b_\epsilon$ due to the change of $\Delta$ to $\Delta -\epsilon \Delta^2$.
\end{remark}

We can prove that a sequence of driving vectors $\{\EX_\epsilon\}_{0 < \epsilon \leq 1}$ in Definition \ref{defi:EX_epsilon} converges in probability in $\CX_T^\kappa$ as $\epsilon$ tends to $0$. (See Theorem \ref{thm:Convergence_of_DV}.)
If this fact holds, the equation for $\phi_\epsilon$ can be decomposed in the following form.
\begin{prop}\label{prop:Formulate}
    Let $\phi(0) \in \CC^{-1/2-\kappa}$ and $\phi_\epsilon$ be a solution of (\ref{eq:phi43}) and (\ref{eq:phi43_init}).
    Let $\{\EX_\epsilon\}_{0 < \epsilon \leq 1}$ be a sequence of driving vectors in Definition \ref{defi:EX_epsilon}.
    Let $\ICC_\epsilon$ be a solution of 
    \[
        \CL_\epsilon \ICC_\epsilon = \XX_\epsilon
    \]
    with initial value $\ICC_\epsilon(0)$ and $y_\epsilon$ be a solution of
    \[
        \CL_\epsilon y_\epsilon = 3e^{3\ICC_\epsilon}\pl(\ICCC_\epsilon \pl \XX_\epsilon),\quad y_\epsilon(0)=0.
    \]
    We set
    \[
        u_\epsilon := e^{3\ICC_\epsilon} (\phi_\epsilon - \IX_\epsilon + \ICCC_\epsilon) -y_\epsilon.
    \]
    Then, $u_\epsilon$ satisfies the equation
    \begin{empheq}[left=\empheqlbrace]{align}\label{eq:L_epsilon_u}
        &\CL_\epsilon u_\epsilon = \Phi_\epsilon(u_\epsilon) - 2\widetilde{G}_\epsilon(3\ICC_\epsilon,u_\epsilon), \\
        &u_\epsilon(0)=e^{3\ICC_\epsilon(0)}(\phi(0)-\IX_\epsilon(0)+\ICCC_\epsilon(0)),
    \end{empheq}
    where we set
    \begin{align}
        \Phi_\epsilon(u_\epsilon) &:= e^{-6\ICC_\epsilon} u_\epsilon^3 + Z^{(2)}_\epsilon u_\epsilon^2 + Z^{(1)}_\epsilon u_\epsilon + Z^{(0)}_\epsilon -6\nabla \ICC_\epsilon \cdot \nabla u_\epsilon, \label{eq:defi_of_Phi}\\
        \widetilde{G}_\epsilon(h,v) &:= \epsilon \{\Delta h\Delta v -\widetilde{B}(h,v) + 2T(v,h,h) + T(h,h,v) - 2 B(h,h)B(h,v) \}. \label{eq:defi_of_G}
    \end{align}
    for some $Z^{(0)}_\epsilon$, $Z^{(1)}_\epsilon$ and $Z^{(2)}_\epsilon$ which are continuous functions from $\EX_\epsilon$ to $C_T\CC^{-1/2-\kappa}$
    and the bilinear terms $B$, $\widetilde{B}$ and the trilinear term $T$, which will appear frequently, are defined as follows.
    \begin{align*}
        B(f,g) &:= \nabla f \cdot \nabla g,\\
        \widetilde{B}(f,g) &:= \nabla\cdot(\Delta f \nabla g) + \nabla\cdot(\nabla f \Delta g) + \Delta B(f,g),\\
        T(f,g,h)&:= \nabla\cdot(B(f,g) \nabla h).
    \end{align*}
\end{prop}
\begin{proof}
Proof will be done in the following three steps:
(1) Transform from $\phi_\epsilon$ to $g_\epsilon=\phi_\epsilon-\IX_\epsilon+\ICCC_\epsilon$,
(2) Transform from $g_\epsilon$ to $v_\epsilon=e^{3\ICC_\epsilon}g_\epsilon$,
(3) Transform from $v_\epsilon$ to $u_\epsilon = v_\epsilon - y_\epsilon$.

\textbf{(1) Transform from $\phi_\epsilon$ to $g_\epsilon=\phi_\epsilon-\IX_\epsilon+\ICCC_\epsilon$.}
Since we cannot expect the convergence of $\phi_\epsilon^3$, we decompose $\phi_\epsilon = \psi_\epsilon + \IX_\epsilon$ and have the equation
\begin{align}
    \CL_\epsilon \psi_\epsilon 
    &= -(\psi_\epsilon + \IX_\epsilon)^3 +3(a_\epsilon - 3b_\epsilon)(\psi_\epsilon + \IX_\epsilon) \notag \\
    &= -\psi_\epsilon^3 -3\psi_\epsilon^2\IX_\epsilon -3\psi_\epsilon(\IX_\epsilon^2-a_\epsilon) - (\IX_\epsilon^3-3a_\epsilon\IX_\epsilon)- 9b_\epsilon(\psi_\epsilon + \IX_\epsilon) \notag \\
    &= -\psi_\epsilon^3 - 3\psi_\epsilon^2 \IX_\epsilon - 3\psi_\epsilon \XX_\epsilon -\XXX_\epsilon - 9b_\epsilon(\psi_\epsilon + \IX_\epsilon). \label{eq:L_epsilon_psi}
\end{align}
By Schauder estimates, we expect that $\psi_\epsilon$ converges in $\CC^{1/2-\kappa}$.
Hence, we still cannot expect the convergence of the product $\psi_\epsilon \XX_\epsilon$ since the sum of the regularities is negative.
We decompose $\psi_\epsilon=g_\epsilon-\ICCC_\epsilon$, then from (\ref{eq:L_epsilon_psi}), we have
\begin{align}
    \CL_\epsilon g_\epsilon
    &= -(g_\epsilon - \ICCC_\epsilon)^3 - 3(g_\epsilon - \ICCC_\epsilon)^2\IX_\epsilon -3g_\epsilon\XX_\epsilon \notag \\
    &\quad + 3(\ICCC_\epsilon\XX_\epsilon-3b_\epsilon \IX_\epsilon)  - 9b_\epsilon(g_\epsilon - \ICCC_\epsilon ) \notag \\
    &= -3g_\epsilon\XX_\epsilon + Q(g_\epsilon),\label{eq:L_epsilon_g}
\end{align}
where
\begin{align*}
    Q(g_\epsilon) &:= Q_0(g_\epsilon)+ 3(\ICCC_\epsilon\XX_\epsilon-3b_\epsilon \IX_\epsilon)  - 9b_\epsilon(g_\epsilon - \ICCC_\epsilon ), \\
    Q_0(g_\epsilon) &:= -(g_\epsilon - \ICCC_\epsilon)^3 - 3(g_\epsilon - \ICCC_\epsilon)^2\IX_\epsilon .
\end{align*}
Because of the regularities, the products $\ICCC_\epsilon\IX_\epsilon$,  $(\ICCC_\epsilon)^2\IX_\epsilon$ and $\ICCC_\epsilon\XX_\epsilon$ are not expected to converge.
But, by the Bony's decomposition, we have
\begin{align}
    \ICCC_\epsilon \IX_\epsilon
    &= \ICCC_\epsilon(\pl + \pr)\IX_\epsilon + \CICCC_\epsilon, \notag \\
    \ICCC_\epsilon\XX_\epsilon - 3b_\epsilon\IX_\epsilon
    &= \ICCC_\epsilon(\pl+\pr)\XX_\epsilon + \CCICCC_\epsilon \label{eq:CCICCC}
\end{align}
and by using commutator estimates (Proposition \ref{prop:commutator_estimates}), we have
\begin{align*}
    (\ICCC_\epsilon)^2 \IX_\epsilon
    &= \ICCC_\epsilon\{\ICCC_\epsilon(\pl + \pr)\IX_\epsilon + \CICCC_\epsilon\}\\
    &= \ICCC_\epsilon \rs (\ICCC_\epsilon \pl \IX_\epsilon ) +  \ICCC_\epsilon \rs (\ICCC_\epsilon \pr \IX_\epsilon)\\
    &\quad +\ICCC_\epsilon(\pl + \pr)\{\ICCC_\epsilon (\pl + \pr) \IX_\epsilon\} + \ICCC_\epsilon\CICCC_\epsilon\\
    &= C(\ICCC_\epsilon,\IX_\epsilon,\ICCC_\epsilon) +  \ICCC_\epsilon \rs (\ICCC_\epsilon \pr \IX_\epsilon)\\
    &\quad +\ICCC_\epsilon(\pl + \pr)\{\ICCC_\epsilon (\pl + \pr) \IX_\epsilon\} + 2\ICCC_\epsilon\CICCC_\epsilon.
\end{align*}
Hence, we can show that $\ICCC_\epsilon \IX_\epsilon$ and $(\ICCC_\epsilon)^2 \IX_\epsilon$ converge in $\CC^{-1/2-\kappa}$ and $\ICCC_\epsilon(\pl+\pr)\XX_\epsilon + \CCICCC_\epsilon$ converges in $\CC^{-1-\kappa}$ from the convergence of $\EX_\epsilon$.

\textbf{(2) Transform from $g_\epsilon$ to $v_\epsilon=e^{3\ICC_\epsilon}g_\epsilon$.} 
By Schauder estimates, we expect that $g_\epsilon$ converges in $\CC^{1-\kappa}$, therefore we still cannot expect the convergence of the product $g_\epsilon\XX_\epsilon$ in (\ref{eq:L_epsilon_g}) since the sum of the regularities is still negative.
To remove it, we set 
\[
    v_\epsilon=e^{3\ICC_\epsilon}g_\epsilon
\]
as in \cite{JP21}.
Here, we state the following property of $\ICC_\epsilon$ as a corollary from Propositions \ref{prop:Effects_of_heat_semigroup}, \ref{prop:Effects_of_heat_semigroup_2} and \ref{prop:Schauder_estimates}. 
\begin{lem}
    \label{lem:LICC}
    Define $\ICC_\epsilon$ as in Proposition \ref{prop:Formulate}.
    Then, $\ICC_\epsilon$ is a continuous function from $\EX_\epsilon$ to $C_T\CC_\epsilon^{1-\kappa}$ and it holds
    \begin{equation}
        \|\ICC_\epsilon\|_{C_T^{1/2-\kappa/2}L^\infty} \lesssim_{\EX_\epsilon} 1.
    \end{equation}
\end{lem}
We consider the equation satisfied by $v_\epsilon$.
Note that the following relationship holds for $\CL_\epsilon v_\epsilon$.
\begin{align}
    \CL_\epsilon v_\epsilon = (3\CL_\epsilon \ICC_\epsilon)v_\epsilon -3\ICC_\epsilon v_\epsilon  + e^{3\ICC_\epsilon}(\CL_\epsilon g_\epsilon) + F_\epsilon(3\ICC_\epsilon)v_\epsilon - 2G_\epsilon(3\ICC_\epsilon,v_\epsilon) \label{eq:L_epsilon_v},
\end{align}
where
\begin{align*}
    F_\epsilon(h) &:= \{|\nabla h|^2+\epsilon(\Delta h)^2\} + \epsilon \{- \widetilde{B}(h,h) + 2T(h,h,h) - B(h,h)^2\},\\
    G_\epsilon(h,v) &:= \nabla h \cdot \nabla v +\epsilon\Delta h\Delta v\\
    &\quad + \epsilon \{-\widetilde{B}(h,v) + 2T(v,h,h) + T(h,h,v) - 2 B(h,h)B(h,v) \}.
\end{align*}
The derivations are given in Appendix (Proposition \ref{prop:derivation_of_equation}).
Now, we show that $v_\epsilon$ satisfies the equation
\begin{align}
    \CL_\epsilon v_\epsilon
    &= \{e^{-6\ICC_\epsilon}v_\epsilon^3 + \widetilde{Z}^{(2)}_\epsilon v_\epsilon^2 + \widetilde{Z}^{(1)}_\epsilon v_\epsilon + \widetilde{Z}^{(0)}_\epsilon\} \notag \\
    &\quad + 3e^{3\ICC_\epsilon}\pl(\ICCC_\epsilon \pl \XX_\epsilon)
    - 2\{G_\epsilon(3\ICC_\epsilon,v_\epsilon) - 9b_\epsilon e^{3\ICC_\epsilon} \ICCC_\epsilon\},
    \label{eq:CL_epsilon_v}
\end{align}
for some $\widetilde{Z}^{(i)}_\epsilon \in C_T\CC^{-1/2-\kappa}$ to be defined later.
From (\ref{eq:L_epsilon_g}) and (\ref{eq:L_epsilon_v}), we have
\begin{align}
    \CL_\epsilon v_\epsilon
    &= 3\XX_\epsilon v_\epsilon - 3\ICC_\epsilon v_\epsilon + e^{3\ICC_\epsilon}\{-3g_\epsilon \XX_\epsilon + Q(g_\epsilon)\}  + F_\epsilon(3\ICC_\epsilon)v_\epsilon - 2G_\epsilon(3\ICC_\epsilon,v_\epsilon) \notag\\
    &= - 3\ICC_\epsilon v_\epsilon + e^{3\ICC_\epsilon}Q(e^{-3\ICC_\epsilon} v_\epsilon) +F_\epsilon(3\ICC_\epsilon)v_\epsilon - 2G_\epsilon(3\ICC_\epsilon,v_\epsilon)\notag\\
    &= \{e^{3\ICC_\epsilon}Q_0(e^{-3\ICC_\epsilon}v_\epsilon) - 3\ICC_\epsilon v_\epsilon\} + 3e^{3\ICC_\epsilon}\{(\ICCC_\epsilon\XX_\epsilon - 3b_\epsilon \IX_\epsilon) - 3b_\epsilon\ICCC_\epsilon\}\notag\\
    &\quad + \{ F_\epsilon(3\ICC_\epsilon)- 9b_\epsilon\}v_\epsilon
     - 2\{G_\epsilon(3\ICC_\epsilon,v_\epsilon) - 9b_\epsilon e^{3\ICC_\epsilon} \ICCC_\epsilon\}.
     \label{eq:L_epsilon_v_1}
\end{align}
The first term has no problems.
Next, we consider the second term.
By the decomposition (\ref{eq:CCICCC}), we have only to check the convergence of the product 
\[
    3e^{3\ICC_\epsilon} \rs (\ICCC_\epsilon \pl \XX_\epsilon) - 9b_\epsilon e^{3\ICC_\epsilon}\ICCC_\epsilon
\]
since we have $\ICCC_\epsilon \pl \XX_\epsilon \in C_T\CC^{-1-\kappa}$ and $\ICCC_\epsilon\pr \XX_\epsilon, \CCICCC_\epsilon \in C_T\CC^{-1/2-\kappa}$.
By the paralinearization theorem (Proposition \ref{prop:paralinearization}), we have 
\[
    3e^{3\ICC_\epsilon} - 9 e^{3\ICC_\epsilon}\pl \ICC_\epsilon \in C_T \CC^{2-2\kappa}, 
\]
therefore we can show that the resonant product 
\[
    (3e^{3\ICC_\epsilon} - 9 e^{3\ICC_\epsilon}\pl \ICC_\epsilon)\rs (\ICCC_\epsilon \pl \XX_\epsilon) 
\]
converges in $\CC^{1-3\kappa}$.
By using commutator estimates (Proposition \ref{prop:commutator_estimates}) twice, we have
\begin{align*}
    &(9 e^{3\ICC_\epsilon}\pl \ICC_\epsilon) \rs (\ICCC_\epsilon \pl \XX_\epsilon) - 9b_\epsilon e^{3\ICC_\epsilon} \ICCC_\epsilon\\
    &= 9C(e^{3\ICC_\epsilon},\ICC_\epsilon,\ICCC_\epsilon\pl\XX_\epsilon) + 9 e^{3\ICC_\epsilon}(\ICC_\epsilon \rs (\ICCC_\epsilon \pl \XX_\epsilon)) - 9b_\epsilon e^{3\ICC_\epsilon} \ICCC_\epsilon\\
    &=9C(e^{3\ICC_\epsilon},\ICC_\epsilon,\ICCC_\epsilon\pl\XX_\epsilon) + 9 e^{3\ICC_\epsilon} C(\ICCC_\epsilon,\XX_\epsilon,\ICC_\epsilon) + 9e^{3\ICC_\epsilon} \ICCC_\epsilon(\ICC_\epsilon \rs \XX_\epsilon-b_\epsilon)\\
    &= 9C(e^{3\ICC_\epsilon},\ICC_\epsilon,\ICCC_\epsilon\pl\XX_\epsilon) + 9 e^{3\ICC_\epsilon} C(\ICCC_\epsilon,\XX_\epsilon,\ICC_\epsilon) + 9 e^{3\ICC_\epsilon} \ICCC_\epsilon \CCICC_\epsilon.
\end{align*}
Thus, we have
\begin{align*}
    3e^{3\ICC_\epsilon}\{(\ICCC_\epsilon\XX_\epsilon - 3b_\epsilon \IX_\epsilon) - 3b_\epsilon\ICCC_\epsilon\} 
    = 3e^{3\ICC_\epsilon}\pl(\ICCC_\epsilon \pl \XX_\epsilon) + Z_\epsilon,
\end{align*}
where $Z_\epsilon$ is a continuous function from $\EX_\epsilon$ to $\CC^{-1/2-\kappa}$.
More precisely, $Z_\epsilon$ can be expressed as
\begin{align}
    Z_\epsilon &:= 3e^{3\ICC_\epsilon}(\ICCC_\epsilon \pr \XX_\epsilon + \CCICCC_\epsilon) + 3e^{3\ICC_\epsilon}\pr(\ICCC_\epsilon \pl \XX_\epsilon) \notag \\
    &\quad + (3e^{3\ICC_\epsilon} - 9e^{3\ICC_\epsilon} \pl \ICC_\epsilon) \rs (\ICCC_\epsilon \pl \XX_\epsilon) + 9C(e^{3\ICC_\epsilon},\ICC_\epsilon, \ICCC_\epsilon \pl \XX_\epsilon) \notag \\
    &\quad + 9e^{3\ICC_\epsilon} C(\ICCC_\epsilon,\XX_\epsilon,\ICC_\epsilon) + 9e^{3\ICC_\epsilon}\ICCC_\epsilon \CCICC_\epsilon. \label{eq:defi_of_Z}
\end{align}

At last, we consider the third term of (\ref{eq:L_epsilon_v_1}).
By using
\begin{align*}
    |\nabla \ICC_\epsilon|^2 &= 2\nabla \ICC_\epsilon \pl \nabla \ICC_\epsilon + \nabla \ICC_\epsilon \rs \nabla \ICC_\epsilon, \\
    (\Delta \ICC_\epsilon)^2 &= 2 \Delta \ICC_\epsilon \pl \Delta \ICC_\epsilon + \Delta \ICC_\epsilon \rs \Delta \ICC_\epsilon,
\end{align*}
we have
\begin{align}
    &F_\epsilon(3\ICC_\epsilon) - 9 b_\epsilon\notag\\
    &= 9\{ |\nabla \ICC_\epsilon|^2 + \epsilon (\Delta \ICC_\epsilon)^2 -b_\epsilon \} \notag\\
    &\quad + \epsilon \{- \widetilde{B}(3\ICC_\epsilon,3\ICC_\epsilon) + 2T(3\ICC_\epsilon,3\ICC_\epsilon,3\ICC_\epsilon) - B(3\ICC_\epsilon,3\ICC_\epsilon)^2\} \notag\\
    &= 9(2\nabla \ICC_\epsilon \pl \nabla \ICC_\epsilon + \CD_\epsilon) \notag\\
    &\quad + \epsilon \{18 \Delta \ICC_\epsilon \pl \Delta \ICC_\epsilon - \widetilde{B}(3\ICC_\epsilon,3\ICC_\epsilon) + 2T(3\ICC_\epsilon,3\ICC_\epsilon,3\ICC_\epsilon) - B(3\ICC_\epsilon,3\ICC_\epsilon)^2\} \notag\\
    &=: \overline{F}_\epsilon. \label{eq:defi_of_bar_F}
\end{align}
Thanks to $\epsilon$-dependent regularity of $\ICC_\epsilon$, we can show that $\overline{F}_\epsilon$ converges in $\CC^{-1/3-4\kappa/3}$ as $\epsilon$ tends to $0$. (See Lemma \ref{lem:bar_F_estimates}.)

Therefore, we have
\begin{align*}
    \CL_\epsilon v_\epsilon 
    &= \{ e^{3\ICC_\epsilon}Q_0(e^{-3\ICC_\epsilon}v_\epsilon) - 3\ICC_\epsilon v_\epsilon + Z_\epsilon + \overline{F}_\epsilon v_\epsilon \}\\
    &\quad + 3e^{3\ICC_\epsilon}\pl(\ICCC_\epsilon \pl \XX_\epsilon) - 2\{G_\epsilon(3\ICC_\epsilon,v_\epsilon) - 9b_\epsilon e^{3\ICC_\epsilon} \ICCC_\epsilon\}. 
\end{align*}
Here, by expanding the first term, we have
\begin{align*}
    \{ e^{3\ICC_\epsilon}Q_0(e^{-3\ICC_\epsilon}v_\epsilon) - 3\ICC_\epsilon v_\epsilon + Z_\epsilon + \overline{F}_\epsilon v_\epsilon \}
    &=e^{-6\ICC_\epsilon}v_\epsilon^3 + \widetilde{Z}^{(2)}_\epsilon v_\epsilon^2 + \widetilde{Z}^{(1)}_\epsilon v_\epsilon + \widetilde{Z}^{(0)}_\epsilon,
\end{align*}
where $\widetilde{Z}^{(0)}_\epsilon$, $\widetilde{Z}^{(1)}_\epsilon$ and $\widetilde{Z}^{(2)}_\epsilon$ are continuous functions from $\EX_\epsilon$ to $C_T\CC^{-1/2-\kappa}$ for $i=0,1,2$. 
More precisely, $\widetilde{Z}^{(i)}_\epsilon$ can be expressed as
\begin{align}
    \widetilde{Z}^{(2)}_\epsilon&:=3e^{-3\ICC_\epsilon}(\ICCC_\epsilon-\IX_\epsilon), \label{eq:defi_of_tilde_Z_2}\\
    \widetilde{Z}^{(1)}_\epsilon&:=3\ICCC_\epsilon(-\ICCC_\epsilon+2\IX_\epsilon)-3\ICC_\epsilon+\overline{F}_\epsilon, \label{eq:defi_of_tilde_Z_1} \\
    \widetilde{Z}^{(0)}_\epsilon&:=e^{3\ICC_\epsilon}\ICCC_\epsilon^2(\ICCC_\epsilon-3\IX_\epsilon) + Z_\epsilon.\label{eq:defi_of_tilde_Z_0}
\end{align}
Therefore, we have (\ref{eq:CL_epsilon_v}).

\textbf{(3) Transform from $v_\epsilon$ to $u_\epsilon = v_\epsilon - y_\epsilon$.}
We can expect that $v_\epsilon$ converges in $\CC^{1-\kappa}$, hence we cannot expect the convergence of the product $\nabla \ICC_\epsilon \cdot \nabla v_\epsilon$ in $G_\epsilon$ since the sum of the regularities is negative.
To remove it, we decompose $v_\epsilon = u_\epsilon + y_\epsilon$.
Here, we state the following property of $y_\epsilon$ as a corollary from Proposition \ref{prop:Schauder_estimates}. 
\begin{lem}
    \label{lem:Ly}
    Define $y_\epsilon$ as in Proposition \ref{prop:Formulate}.
    Then, $y_\epsilon$ is a continuous function from $\EX_\epsilon$ to $C_T\CC_\epsilon^{1-\kappa}$.
\end{lem}
We show that $u_\epsilon$ satisfies (\ref{eq:L_epsilon_u}).
From (\ref{eq:CL_epsilon_v}) and the linearity of $G_\epsilon$ with respect to second variable, we have
\begin{align}
    \CL_\epsilon u_\epsilon
    &= \{e^{-6\ICC_\epsilon}(u_\epsilon+y_\epsilon)^3 + \widetilde{Z}^{(2)}_\epsilon (u_\epsilon+y_\epsilon)^2 + \widetilde{Z}^{(1)}_\epsilon (u_\epsilon+y_\epsilon) + \widetilde{Z}^{(0)}_\epsilon\} \notag\\
    &\quad  -2\{G_\epsilon(3\ICC_\epsilon, y_\epsilon) - 9b_\epsilon e^{3\ICC_\epsilon}\ICCC_\epsilon \} - 2G_\epsilon(3\ICC_\epsilon,u_\epsilon) \notag \\
    &= \{e^{-6\ICC_\epsilon}(u_\epsilon+y_\epsilon)^3 + \widetilde{Z}^{(2)}_\epsilon (u_\epsilon+y_\epsilon)^2 + \widetilde{Z}^{(1)}_\epsilon (u_\epsilon+y_\epsilon) + \widetilde{Z}^{(0)}_\epsilon -6\nabla \ICC_\epsilon\cdot \nabla u_\epsilon\} \notag \\
    &\quad  -2\{G_\epsilon(3\ICC_\epsilon, y_\epsilon) - 9b_\epsilon e^{3\ICC_\epsilon}\ICCC_\epsilon \} - 2\widetilde{G}_\epsilon(3\ICC_\epsilon,u_\epsilon).
    \label{eq:CL_epsilon_u_1}
\end{align}
We can show that the last term $\widetilde{G}_\epsilon(3\ICC_\epsilon,u_\epsilon)$ converges in $\CC^{-1/3-4\kappa/3}$ if the regularity of $u_\epsilon$ is more than $1+2\kappa$. (See Lemma \ref{lem:G_epsilon_u_estimates}.)

It is sufficient that we consider the second term.
It is divided into two terms as follows:
\begin{align*}
    G_\epsilon(3\ICC_\epsilon, y_\epsilon) - 9b_\epsilon e^{3\ICC_\epsilon} \ICCC_\epsilon
    = 3(\nabla \ICC_\epsilon \cdot \nabla y_\epsilon +\epsilon\Delta \ICC_\epsilon\Delta y_\epsilon - 3b_\epsilon e^{3\ICC_\epsilon} \ICCC_\epsilon) + \overline{G}_\epsilon^{(1)}, 
\end{align*}
where we set
\begin{multline}
    \overline{G}_\epsilon^{(1)}
    :=\epsilon \{-\widetilde{B}(3\ICC_\epsilon,y_\epsilon) + 2T(y_\epsilon,3\ICC_\epsilon,3\ICC_\epsilon)\\
    \quad \quad + T(3\ICC_\epsilon,3\ICC_\epsilon,y_\epsilon)
    - 2 B(3\ICC_\epsilon,3\ICC_\epsilon)B(3\ICC_\epsilon,y_\epsilon) \}
    \label{eq:defi_of_bar_G_1}
\end{multline}
and we can show that $\overline{G}_\epsilon^{(1)}$ converges in $\CC^{-1/3-4\kappa/3}$. (See Lemma \ref{lem:widehat_G_epsilon_u_estimates}.)
For the remaining term, we can decompose
\begin{align}
    &3(\nabla \ICC_\epsilon \cdot \nabla y_\epsilon +\epsilon\Delta \ICC_\epsilon\Delta y_\epsilon - 3b_\epsilon e^{3\ICC_\epsilon} \ICCC_\epsilon)\notag \\
    &= \nabla \ICC_\epsilon \rs \nabla y_\epsilon +\epsilon\Delta \ICC_\epsilon \rs \Delta y_\epsilon - 3b_\epsilon e^{3\ICC_\epsilon} \ICCC_\epsilon + \overline{G}_\epsilon^{(2)},\label{eq:R_term_1}
\end{align}
where $\overline{G}_\epsilon^{(2)}$ is defined by the paraproduct term as follows:
\begin{align}
    \label{eq:defi_of_bar_G_2}
    \overline{G}_\epsilon^{(2)}:=3\{\nabla \ICC_\epsilon(\pl+\pr)\nabla y_\epsilon + \epsilon \Delta\ICC_\epsilon(\pl+\pr)\Delta y_\epsilon\}.
\end{align}
We can show that $\overline{G}_\epsilon^{(2)}$ also converges in $\CC^{-3\kappa}$ as $\epsilon$ tends to $0$.
(See Lemma \ref{lem:widehat_G_epsilon_u_estimates}.)
It remains to check the convergence of the resonant products.
We will show it in same way as Lemma A.2 of \cite{JP21}.
By Bony's paramultiplication bound (Proposition \ref{prop:Bonys_paramultiplication}), we have
\begin{multline}
    \|3e^{3\ICC_\epsilon}\pl(\ICCC_\epsilon\pl\XX_\epsilon) - (3e^{3\ICC_\epsilon}\ICCC_\epsilon)\pl\XX_\epsilon\|_{C_T\CC^{-1/2-2\kappa}}\\
    \lesssim \|3e^{3\ICC_\epsilon}\|_{C_T\CC^{1/2-\kappa}}\|\ICCC_\epsilon\|_{C_T\CC^{1/2-\kappa}}\|\XX_\epsilon\|_{C_T\CC^{-1-\kappa}},
    \label{eq:y_Bony}
\end{multline}
then it holds
\[
    \CL_\epsilon y_\epsilon = C_T\CC^{-1/2-2\kappa} + (3e^{3\ICC_\epsilon}\ICCC_\epsilon)\pl\XX_\epsilon.
\]
By using Schauder estimates (Proposition \ref{prop:Schauder_estimates}), we obtain
\[
    y_\epsilon = \CTCC{3/2-2\kappa} + \CL^{-1}_{\epsilon}[(3e^{3\ICC_\epsilon}\ICCC_\epsilon)\pl\XX_\epsilon].
\]
Moreover, by using Proposition \ref{prop:commutation_L_inverse_paraproduct}, $\CL_\epsilon^{-1}$ approximately commutes with paraproduct, that is
\begin{multline}
    \|\CL_\epsilon^{-1}[(3e^{3\ICC_\epsilon}\ICCC_\epsilon)\pl\XX_\epsilon] - 3(e^{3\ICC_\epsilon}\ICCC_\epsilon)\pl \{\ICC_\epsilon-P_\cdot^\epsilon \ICC_\epsilon(0)\}\|_{C_T\CC^{3/2 - 3\kappa
    }} \\
    \lesssim (\|3e^{3\ICC_\epsilon}\ICCC_\epsilon\|_{C_T\CC^{1/2-\kappa}}+\|3e^{3\ICC_\epsilon}\ICCC_\epsilon\|_{C_T^{1/4-\kappa/2}L^\infty}) \|\XX_\epsilon\|_{C_T\CC^{-1-\kappa}}.
    \label{eq:y_commutation}
\end{multline}
Here, note that
\[
    \CL_\epsilon^{-1}[\XX_\epsilon](t)
    = \ICC_\epsilon(t) - P_t^\epsilon \ICC_\epsilon(0)
\]
and $\|\ICC_\epsilon\|_{C_T^{1/2-\kappa/2}L^\infty} \lesssim 1$ from Lemma \ref{lem:LICC}.
Thus, we have
\begin{align*}
    y_\epsilon = Y_\epsilon -3(e^{3\ICC_\epsilon}\ICCC_\epsilon)\pl P_\cdot^\epsilon \ICC_\epsilon(0) + 3(e^{3\ICC_\epsilon}\ICCC_\epsilon)\pl \ICC_\epsilon, 
\end{align*}
where we define
\begin{align*}
    Y_\epsilon&:= \CL_\epsilon^{-1} [3e^{3\ICC_\epsilon}\pl(\ICCC_\epsilon\pl\XX_\epsilon)]  - 3(e^{3\ICC_\epsilon}\ICCC_\epsilon)\pl \{\ICC_\epsilon-P_\cdot^\epsilon \ICC_\epsilon(0)\}
\end{align*}
and we can show that $Y_\epsilon$ converges in $\CC^{3/2-3\kappa}$ as $\epsilon$ tends to $0$ by using (\ref{eq:y_Bony}) and (\ref{eq:y_commutation}). 
In particular, we have
\begin{align}\label{eq:Y_epsilon_estimates}
    \|Y_\epsilon\|_{C_T\CC^{3/2-3\kappa}} \lesssim_{\EX_\epsilon} 1.
\end{align}
Therefore, we can decompose the first term of (\ref{eq:R_term_1}) as follows:
\begin{align}
    &\nabla \ICC_\epsilon \rs \nabla y_\epsilon +\epsilon\Delta \ICC_\epsilon \rs \Delta y_\epsilon - 3b_\epsilon e^{3\ICC_\epsilon} \ICCC_\epsilon \notag \\
    &= 9[ \nabla \ICC_\epsilon \rs \nabla \{(e^{3\ICC_\epsilon}\ICCC_\epsilon)\pl \ICC_\epsilon\} \notag\\
    &\quad + \epsilon \Delta \ICC_\epsilon \rs \Delta \{(e^{3\ICC_\epsilon}\ICCC_\epsilon)\pl \ICC_\epsilon\} -b_\epsilon e^{3\ICC_\epsilon}\ICCC_\epsilon ] + \overline{G}_\epsilon^{(3)} + \overline{G}_\epsilon^{(4)}, \label{eq:R_term_2}
\end{align}
where we define
\begin{align}
    \overline{G}_\epsilon^{(3)}
    &:=3\{\nabla \ICC_\epsilon \rs \nabla Y_\epsilon + \epsilon \Delta \ICC_\epsilon \rs \Delta Y_\epsilon\}, \label{eq:defi_of_bar_G_3} \\
    \overline{G}_\epsilon^{(4)}
    &:=-9[\nabla \ICC_\epsilon \rs \nabla \{(e^{3\ICC_\epsilon}\ICCC_\epsilon) \pl P_\cdot^\epsilon \ICC_\epsilon(0)\} \notag \\
    &\quad\quad\quad  + \epsilon \Delta \ICC_\epsilon \rs \Delta \{(e^{3\ICC_\epsilon}\ICCC_\epsilon) \pl P_\cdot^\epsilon \ICC_\epsilon(0)\}]. \label{eq:defi_of_bar_G_4}
\end{align}
They are the inner product terms including $Y_\epsilon$ and $P_\cdot^\epsilon \ICC_\epsilon(0)$, respectively, and we can show that they converge in $\CC^{1/2-5\kappa}$ and in $\CC^{\kappa}$, respectively. 
(See Lemma \ref{lem:widehat_G_epsilon_u_estimates}.)
For the remaining terms, by using Leibniz rule and commutator estimates (Proposition \ref{prop:commutator_estimates}), we decompose
\begin{align*}
    &\nabla \ICC_\epsilon \rs \nabla \{(e^{3\ICC_\epsilon}\ICCC_\epsilon)\pl \ICC_\epsilon\}\\
    & = \nabla \ICC_\epsilon \rs \{ \nabla (e^{3\ICC_\epsilon}\ICCC_\epsilon)\pl \ICC_\epsilon\} + \nabla \ICC_\epsilon \rs \{(e^{3\ICC_\epsilon}\ICCC_\epsilon)\pl \nabla \ICC_\epsilon\} \\
    & = \nabla \ICC_\epsilon \rs \{ \nabla (e^{3\ICC_\epsilon}\ICCC_\epsilon)\pl \ICC_\epsilon\} +  C(e^{3\ICC_\epsilon}\ICCC_\epsilon, \nabla\ICC_\epsilon,\nabla \ICC_\epsilon) + e^{3\ICC_\epsilon}\ICCC_\epsilon(\nabla \ICC_\epsilon \rs \nabla \ICC_\epsilon)
\end{align*}
and
\begin{align*}
    &\epsilon \Delta \ICC_\epsilon \rs \Delta \{(e^{3\ICC_\epsilon}\ICCC_\epsilon)\pl \ICC_\epsilon\} \\
    &= \epsilon [\Delta \ICC_\epsilon \rs\{ \Delta (e^{3\ICC_\epsilon}\ICCC_\epsilon)\pl \ICC_\epsilon\} + 2 \Delta \ICC_\epsilon \rs  \{\nabla (e^{3\ICC_\epsilon}\ICCC_\epsilon)\pl \nabla \ICC_\epsilon\} \\
    &\quad \quad + \Delta \ICC_\epsilon \rs  \{(e^{3\ICC_\epsilon}\ICCC_\epsilon)\pl \Delta \ICC_\epsilon\}] \\
    &= \epsilon [\Delta \ICC_\epsilon \rs\{ \Delta (e^{3\ICC_\epsilon}\ICCC_\epsilon)\pl \ICC_\epsilon\} + 2 \Delta \ICC_\epsilon \rs  \{\nabla (e^{3\ICC_\epsilon}\ICCC_\epsilon)\pl \nabla \ICC_\epsilon\} \\
    &\quad\quad +  C(e^{3\ICC_\epsilon}\ICCC_\epsilon, \Delta\ICC_\epsilon,\Delta \ICC_\epsilon) + e^{3\ICC_\epsilon}\ICCC_\epsilon(\Delta \ICC_\epsilon \rs \Delta \ICC_\epsilon)].
\end{align*}
Note that $\CD_\epsilon = \nabla \ICC_\epsilon \rs \nabla \ICC_\epsilon + \epsilon \Delta \ICC_\epsilon \rs \Delta \ICC_\epsilon - b_\epsilon$.
Let
\begin{align}
    \overline{G}_\epsilon^{(5)}
    &:= 9\{ \nabla \ICC_\epsilon \rs  (\nabla(e^{3\ICC_\epsilon}\ICCC_\epsilon)\pl \ICC_\epsilon) \notag \\
    &\quad+ \epsilon \Delta \ICC_\epsilon \rs  (\Delta(e^{3\ICC_\epsilon}\ICCC_\epsilon)\pl \ICC_\epsilon) + 2\epsilon \Delta \ICC_\epsilon \rs  (\nabla(e^{3\ICC_\epsilon}\ICCC_\epsilon)\pl \nabla \ICC_\epsilon) \}, \label{eq:defi_of_bar_G_5} \\
    \overline{G}_\epsilon^{(6)}
    &:=9\{ C(e^{3\ICC_\epsilon}\ICCC_\epsilon, \nabla\ICC_\epsilon,\nabla \ICC_\epsilon ) + \epsilon C(e^{3\ICC_\epsilon}\ICCC_\epsilon, \Delta \ICC_\epsilon,\Delta \ICC_\epsilon ) + e^{3\ICC_\epsilon}\ICCC_\epsilon \CD_\epsilon\}.\label{eq:defi_of_bar_G_6}
\end{align}
Then, the first term of (\ref{eq:R_term_2}) is equal to the sum of $\overline{G}_\epsilon^{(5)}$ and $\overline{G}_\epsilon^{(6)}$.
We can show that they converge in $\CC^{1/2-4\kappa}$ and $\CC^{-\kappa}$, respectively. (See Lemma \ref{lem:widehat_G_epsilon_u_estimates}.)
Therefore, for the second term of (\ref{eq:CL_epsilon_u_1}), we have
\[
    G_\epsilon(3\ICC_\epsilon, y_\epsilon) - 9b_\epsilon e^{3\ICC_\epsilon} \ICCC_\epsilon=\sum_{j=1}^6 \overline{G}_\epsilon^{(j)}=:\overline{G}_\epsilon
\]
and $\overline{G}_\epsilon$ converges in $\CC^{-1/3-4\kappa/3}$.

We define
\[
    \Phi_\epsilon(u_\epsilon):=e^{-6\ICC_\epsilon}(u_\epsilon+y_\epsilon)^3+\widetilde{Z}_\epsilon^{(2)} (u_\epsilon+y_\epsilon)^2 +\widetilde{Z}_\epsilon^{(1)} (u_\epsilon + y_\epsilon) + \widetilde{Z}_\epsilon^{(0)} - 2\overline{G}_\epsilon-6\nabla \ICC_\epsilon \cdot \nabla u_\epsilon,
\]
which is the sum of the first and second terms of (\ref{eq:CL_epsilon_u_1}).
By expanding it, we have
\[
    \Phi_\epsilon(u_\epsilon)=e^{-6\ICC_\epsilon}u_\epsilon^3+Z_\epsilon^{(2)} u_\epsilon^2 +Z_\epsilon^{(1)} u_\epsilon + Z_\epsilon^{(0)}-6\nabla \ICC_\epsilon \cdot \nabla u_\epsilon,
\]
where $Z^{(0)}_\epsilon$, $Z^{(1)}_\epsilon$ and $Z^{(2)}_\epsilon$ are continuous functions from $\EX_\epsilon$ to $C_T\CC^{-1/2-\kappa}$ for $i=0,1,2$. 
More precisely, $Z^{(i)}_\epsilon$ can be expressed as
\begin{align}
    Z^{(2)}_\epsilon &:= \widetilde{Z}^{(2)}_\epsilon + 3e^{-6\ICC_\epsilon}y_\epsilon \label{eq:defi_of_Z_2},\\
    Z^{(1)}_\epsilon &:= \widetilde{Z}^{(1)}_\epsilon + 2\widetilde{Z}^{(2)}_\epsilon y_\epsilon + 3e^{-6\ICC_\epsilon}y_\epsilon^2 \label{eq:defi_of_Z_1},\\
    Z^{(0)}_\epsilon &:= \widetilde{Z}^{(0)}_\epsilon + \widetilde{Z}^{(1)}_\epsilon y_\epsilon + \widetilde{Z}^{(2)}_\epsilon y_\epsilon^2 +e^{-6\ICC_\epsilon}y_\epsilon^3 -2 \overline{G}_\epsilon\label{eq:defi_of_Z_0},
\end{align}
where $\widetilde{Z}^{0}_\epsilon$, $\widetilde{Z}^{1}_\epsilon$ and $\widetilde{Z}^{2}_\epsilon$ are defined by (\ref{eq:defi_of_tilde_Z_2}) to (\ref{eq:defi_of_tilde_Z_0}).
Therefore, we have (\ref{eq:L_epsilon_u}).
\end{proof}

\subsubsection{Formulation of problems}
Let $\CR_\epsilon(\EX_\epsilon, u_\epsilon)$ be the right-hand side of (\ref{eq:L_epsilon_u}), that is
\[
    \CR_\epsilon(\EX_\epsilon,u_\epsilon)
    :=\Phi_\epsilon(u_\epsilon) - 2\widetilde{G}_\epsilon(3\ICC_\epsilon,u_\epsilon).
\]
Moreover, we extend the definition of $\Phi_\epsilon$, $Z_\epsilon^{(i)}$, $\overline{F}_\epsilon$, $\widetilde{G}_\epsilon$, $\overline{G}_\epsilon$ and $\CR_\epsilon$ to the case $\epsilon = 0$ by omitting the term containing $\epsilon$.
For example, we define
\[
    \overline{F}_0(\EX):= 9(2\nabla \ICC \pl \nabla \ICC + \CD).
\]
In particular, $\widetilde{G}_0=0$.
We consider the equation
\begin{align}\label{eq:L_epsilon_u_epsilon}
    \CL_\epsilon u_\epsilon = \CR_\epsilon(\EX_\epsilon,u_\epsilon)
\end{align}
for $ 0\leq \epsilon \leq 1$, the initial condition $u_\epsilon(0)$ and the driving vector $\EX_\epsilon$.
The main problem of this section is the following:
\begin{enumerate}[(i)]
    \item The $\epsilon$-uniform local well-posedness of (\ref{eq:L_epsilon_u_epsilon}).
    \item The convergence of $\{u_\epsilon\}_{0 < \epsilon \leq 1}$ to $u_0$.
\end{enumerate}
We will show (i) in Section \ref{sec:LWP} and (ii) in Section \ref{sec:Convergence}.
In the next section, we will estimates $\CR_\epsilon(\EX_\epsilon,u_\epsilon)$ in oreder to prove (i) and (ii).

\subsection{Preliminary}
\label{sec:Preliminary}

Fix $\kappa>0$ and $0<T\leq 1$.
Let $\EX$, $\EX^{(1)}$ and $\EX^{(2)}$ be elements of $\CX_T^\kappa$ and $\{\EX_\epsilon\}_{0< \epsilon \leq 1}$, $\{\EX_\epsilon^{(1)}\}_{0< \epsilon \leq 1}$ and $\{\EX_\epsilon^{(2)}\}_{0< \epsilon \leq 1}$ be sequences of $\CX_T^\kappa$.
\begin{prop}
    \label{prop:Right-hand_side_estimates}
    For any $0<t\leq T$ and $0\leq\epsilon\leq 1$, we have the following:
    \begin{enumerate}[(1)]
        \item 
            For $u_\epsilon \in \CE_T^{-1/2-\kappa}\CC_\epsilon^{3/2-2\kappa}$,
            \begin{align}
                \|\CR_\epsilon(\EX_\epsilon, u_\epsilon)(t)\|_{\CC^{-1/2-\kappa}} \leq C t^{-1+\frac{\kappa}{2}} (1+\|u_\epsilon\|_{\CE_T^{-1/2-\kappa}\CC_\epsilon^{3/2-2\kappa}}^3),
            \end{align}
            where, $C$ is a positive constant depending only on $\kappa$, $\kappa'$ and $\|\EX_\epsilon\|_{\CX_T^\kappa}$.
        \item 
            For any $u_\epsilon^{(1)}, u_\epsilon^{(2)} \in \CE_T^{-1/2-\kappa}\CC_\epsilon^{3/2-2\kappa}$,
            \begin{multline*}
                \|\CR_\epsilon(\EX_\epsilon^{(1)},u_\epsilon^{(1)})(t)- \CR_\epsilon(\EX_\epsilon^{(2)},u_\epsilon^{(2)})(t)\|_{\CC^{-1/2-\kappa}}\\
                \leq t^{-1+\frac{\kappa}{2}}(C'_1\|\EX_\epsilon^{(1)}-\EX_\epsilon^{(2)}\|_{\CX_T^\kappa} + C_2' \|u_\epsilon^{(1)}-u_\epsilon^{(2)}\|_{\CE_T^{-1/2-\kappa}\CC_\epsilon^{3/2-2\kappa}}).
            \end{multline*}
            Here, $C'_1$ and $C'_2$ are positive constants depending only on $\kappa$, $\kappa'$, $\|\EX_\epsilon^{(i)}\|_{\CX_T^\kappa}$ and $\|u_\epsilon^{(i)}\|_{\CE_T^{-1/2-\kappa}\CC_\epsilon^{3/2-2\kappa}}$.
            In particular, $C'_1$ is given by at most third-order polynomial and $C'_2$ is given by at most second-order polynomial in $\|u_\epsilon^{(i)}\|_{\CE_T^{-1/2-\kappa}\CC_\epsilon^{3/2-2\kappa}}$.
    \end{enumerate}
    Moreover, for any $0<t\leq T$, $0<\epsilon \leq 1$, $u\in \CE_T^{-1/2-\kappa}\CC^{3/2-2\kappa}$ and $u_\epsilon\in \CE_T^{-1/2-\kappa}\CC_\epsilon^{3/2-2\kappa}$,  we have
    \begin{multline*}
        \|\CR(\EX,u)(t)- \CR_\epsilon(\EX_\epsilon,u_\epsilon)(t)\|_{\CC^{-1/2-\kappa}}\\
        \leq t^{-1+\frac{\kappa}{2}} \{C''_1(\epsilon^{\frac{\kappa}{4}}+\|\EX-\EX_\epsilon\|_{\CX_T^\kappa} )+ C_2'' \|u-u_\epsilon\|_{\CE_T^{-1/2-\kappa}\CC^{3/2-2\kappa}}\}.
    \end{multline*}
    Here, $C''_1$ and $C''_2$ are positive constants depending only on $\kappa$, $\kappa'$, $\|\EX\|_{\CX_T^\kappa}$, $\|\EX_\epsilon\|_{\CX_T^\kappa}$, $\|u\|_{\CE_T^{-1/2-\kappa}\CC^{3/2-2\kappa}}$ and $\|u_\epsilon\|_{\CE_T^{-1/2-\kappa}\CC_\epsilon^{3/2-2\kappa}}$. 
    In particular, $C''_1$ is given by at most third-order polynomial and $C''_2$ is given by at most second-order polynomial in $\|u\|_{\CE_T^{-1/2-\kappa}\CC^{3/2-2\kappa}}$ and $\|u_\epsilon\|_{\CE_T^{-1/2-\kappa}\CC_\epsilon^{3/2-2\kappa}}$.
\end{prop}
We prove Proposition \ref{prop:Right-hand_side_estimates} as a consequence of Lemmas \ref{lem:Phi_estimates} and \ref{lem:G_epsilon_u_estimates}.

\subsubsection{Basic estimates}
First, we introduce two interpolation estimates which appear frequently in this section.
\begin{lem}
    \label{lem:interpolation_estimates_epsilon}
    Let $\alpha \in \R$ and $0<\epsilon \leq 1$.
    For any $u\in \CC_\epsilon^\alpha$ and $ 0\leq \delta \leq 1$, we have
    \[
        \|u\|_{\CC^{\alpha+\delta(2-\kappa)}}
        \lesssim \epsilon^{-\delta(1-\frac{\kappa}{4})} \|u\|_{\CC_\epsilon^\alpha}.
    \]
\end{lem}
\begin{proof}
    By the interpolation estimate (\cite{BCD11} Lemma 2.80), we have
    \begin{align*}
        \|u\|_{\CC^{\alpha+\delta(2-\kappa)}}
        & \lesssim \|u\|_{\CC^{\alpha+2-\kappa}}^{\delta}\|u\|_{\CC^{\alpha}}^{1-\delta} \\
        & \lesssim (\epsilon^{-1+\frac{\kappa}{4}} \|u\|_{\CC_\epsilon^{\alpha}})^{\delta}\|u\|_{\CC_\epsilon^{\alpha}}^{1-\delta} \\
        &= \epsilon^{-\delta(1-\frac{\kappa}{4})} \|u\|_{\CC_\epsilon^\alpha}.
    \end{align*}
\end{proof}
In the following, we use
\begin{align*}
    \|u\|_{\CC^{\alpha+2-\kappa}}&\lesssim \epsilon^{-1+\frac{\kappa}{4}}\|u\|_{\CC_\epsilon^\alpha}, && (\delta=1), \\
    \|u\|_{\CC^{\alpha+1-\frac{\kappa}{2}}} &\lesssim \epsilon^{-\frac{1}{2}(1-\frac{\kappa}{4})} \|u\|_{\CC_\epsilon^\alpha}, &&(\delta = 1/2), \\
    \|u\|_{\CC^{\alpha+\frac{2}{3}-\frac{\kappa}{3}}} &\lesssim \epsilon^{-\frac{1}{3}(1-\frac{\kappa}{4}) }\|u\|_{\CC_\epsilon^\alpha}, && (\delta = 1/3).
\end{align*}
\begin{lem}
    \label{lem:interpolation_estimates}
    Let $\beta<\alpha$ and $0<t\leq T$.
    For any $u \in \CE_T^\beta\CC^\alpha$ and $\gamma \in [\beta, \alpha]$, we have
    \[
        \|u(t)\|_{\CC^\gamma}
        \lesssim t^{-\frac{\gamma-\beta}{2}} \|u\|_{\CE_T^\beta\CC^\alpha}.
    \]
    Moreover, for any $\beta<0<\alpha$ and $u \in \CE_T^\beta\CC^\alpha$, we have
    \[
        \|u(t)\|_{L^\infty}\lesssim t^{\frac{\beta}{2}}\|u\|_{\CE_T^\beta\CC^\alpha}.
    \]
\end{lem}
\begin{proof}
    By the interpolation estimate (\cite{BCD11} Lemma 2.80), we have
    \begin{align*}
        \|u(t)\|_{\CC^\gamma}
        &\lesssim \|u(t)\|_{\CC^\alpha}^{\frac{\gamma-\beta}{\alpha-\beta}} \|u(t)\|_{\CC^\beta}^{\frac{\alpha-\gamma}{\alpha-\beta}} \\
        &\lesssim (t^{-\frac{\alpha-\beta}{2}}\|u\|_{\CE_T^\beta\CC^\alpha})^{\frac{\gamma-\beta}{\alpha-\beta}} (\|u\|_{\CE_T^\beta\CC^\alpha})^{\frac{\alpha-\gamma}{\alpha-\beta}} \\
        &= t^{-\frac{\gamma-\beta}{2}}\|u\|_{\CE_T^\beta\CC^\alpha}.
    \end{align*}
    Moreover, for $\beta < 0 < \alpha$, we have
    \begin{align*}
        \|u(t)\|_{L^\infty}
        &\leq \sum_{j\geq -1}\|\Delta_j u(t)\|_{L^\infty}=\|u(t)\|_{\CB_{\infty,1}^{0}} \\
        &\lesssim \|u(t)\|_{\CC^\alpha}^{\frac{-\beta}{\alpha-\beta}} \|u(t)\|_{\CC^{\beta}}^{\frac{\alpha}{\alpha-\beta}} \\
        &\lesssim (t^{-\frac{\alpha-\beta}{2}}\|u\|_{\CE_T^\beta\CC^\alpha})^{\frac{-\beta}{\alpha-\beta}} (\|u\|_{\CE_T^\beta\CC^\alpha})^{\frac{\alpha}{\alpha-\beta}}\\
        &\lesssim t^{\frac{\beta}{2}}\|u\|_{\CE_T^\beta\CC^\alpha}.
    \end{align*}
\end{proof}

Next, we show the estimates of $B$, $T$ and $\widetilde{B}$.
Their estimates are obtained from the following lemma.
\begin{lem}\label{lem:product_estimates}
    Let $ 0 \leq \delta \leq 1$ and $\alpha, \beta<0$ satisfy $\alpha+\beta+\delta(2-\kappa)>0$.
    For $0<\epsilon\leq 1$, $f\in \CC_\epsilon^\alpha$ and $g\in \CC_\epsilon^\beta$, we have
    \begin{equation}
        \|fg\|_{\CC^{\alpha+\beta+\delta(2-\kappa)}}\lesssim \epsilon^{-\delta(1-\frac{\kappa}{4})}\|f\|_{\CC_\epsilon^\alpha}\|g\|_{\CC_\epsilon^\beta} \label{eq:product_estimates}.
    \end{equation} 
\end{lem}
\begin{proof}
    By Bony's decomposition, we have
    \[
        fg=f\pl g + f\pr g+ f\rs g.
    \]
    By using paraproduct estimates (Proposition \ref{prop:paraproduct_estimates}) and Lemma \ref{lem:interpolation_estimates_epsilon} for each term, we will show (\ref{eq:product_estimates}).
    For the first term, we have
    \begin{align*}
        \|f\pl g\|_{\CC^{\alpha+\beta+\delta(2-\kappa)}}
        \lesssim \|f\|_{\CC^\alpha}\|g\|_{\CC^{\beta+\delta(2-\kappa)}} 
        \lesssim \epsilon^{-\delta(1-\frac{\kappa}{4})}\|f\|_{\CC_\epsilon^\alpha}\|g\|_{\CC_\epsilon^\beta}
    \end{align*}
    since $\alpha<0$.
    For the second term, we have
    \begin{align*}
        \|f\pr g\|_{\CC^{\alpha+\beta+\delta(2-\kappa)}}
        \lesssim \|f\|_{\CC^{\alpha+\delta(2-\kappa)}}\|g\|_{\CC^{\beta}}
        \lesssim \epsilon^{-\delta(1-\frac{\kappa}{4})}\|f\|_{\CC_\epsilon^\alpha}\|g\|_{\CC_\epsilon^\beta}
    \end{align*}
    since $\beta<0$.
    For the last term, we have
    \begin{align*}
        \|f\rs g\|_{\CC^{\alpha+\beta+\delta(2-\kappa)}}
        \lesssim \|f\|_{\CC^{\alpha+\delta(2-\kappa)}}\|g\|_{\CC^{\beta}} 
        \lesssim \epsilon^{-\delta(1-\frac{\kappa}{4})}\|f\|_{\CC_\epsilon^\alpha}\|g\|_{\CC_\epsilon^\beta}
    \end{align*}
    since $\alpha+\beta+\delta(2-\kappa) > 0$.
\end{proof}

\begin{cor}\label{cor:B_estimates}
    Let $ 0\leq \delta \leq 1$ and $\alpha, \beta<1$ satisfy $\alpha+\beta-2+\delta(2-\kappa)>0$.
    For $0<\epsilon\leq 1$, $f\in \CC_\epsilon^\alpha$ and $g\in \CC_\epsilon^\beta$, we have
    \begin{align*}
        \|B(f,g)\|_{\CC^{\alpha+\beta-2+\delta(2-\kappa)}}
        \lesssim \epsilon^{-\delta(1-\frac{\kappa}{4})}\|f\|_{\CC_\epsilon^\alpha}\|g\|_{\CC_\epsilon^\beta}.
    \end{align*}
\end{cor}
\begin{proof}
    Note that $\|\partial^k \phi\|_{\CC^{\alpha-|k|}} \lesssim \|\phi\|_{\CC^\alpha}$ for any multiindex $k \in \N^d$ (\cite{BCD11} Lemma 2.1).
    From Lemma \ref{lem:product_estimates}, we have
    \[
        \|B(f,g)\|_{\CC^{\alpha+\beta-2+\delta(2-\kappa)}}\lesssim\epsilon^{-\delta(1-\frac{\kappa}{4})}\|\nabla f\|_{\CC_\epsilon^{\alpha-1}}\|\nabla g\|_{\CC_\epsilon^{\beta-1}} \lesssim \epsilon^{-\delta(1-\frac{\kappa}{4})}\|f\|_{\CC_\epsilon^\alpha}\|g\|_{\CC_\epsilon^\beta}
    \]
    since $\alpha, \beta <1$ and $\alpha+\beta-2+\delta(2-\kappa)>0$.
\end{proof}

\begin{cor}\label{cor:T_estimates}
    Let $\alpha, \beta<1$ satisfy $\alpha+\beta-2(1+\kappa)/3>0$ and let $\gamma\in \R$ satisfy $\alpha+\beta+\gamma>1+\kappa$.
    For $0<\epsilon\leq 1$, $f\in \CC_\epsilon^\alpha$, $g\in \CC_\epsilon^\beta$ and $h \in \CC^\gamma$, we have
    \begin{align*}
        \|T(f,g,h)\|_{\CC^{(\alpha+\beta-3+\frac{2}{3}(2-\kappa))\wedge(\gamma-2+ \frac{1}{3}(2-\kappa))}}
        \lesssim \epsilon^{-1+\frac{\kappa}{4}}\|f\|_{\CC_\epsilon^\alpha}\|g\|_{\CC_\epsilon^\beta}\|h\|_{\CC_\epsilon^\gamma}.
    \end{align*}
\end{cor}
\begin{proof}
    By the assumption, we have
    \begin{align*}
        &\|T(f,g,h)\|_{\CC^{(\alpha+\beta-3+\frac{2}{3}(2-\kappa))\wedge(\gamma-2+ \frac{1}{3}(2-\kappa))}}\\
        &\lesssim \|B(f,g)\nabla h\|_{\CC^{(\alpha+\beta-2+\frac{2}{3}(2-\kappa))\wedge(\gamma-1+ \frac{1}{3}(2-\kappa))}}\\
        &\lesssim \|B(f,g)\|_{\CC^{\alpha+\beta-2+\frac{2}{3}(2-\kappa)}}\|\nabla h\|_{\CC^{\gamma-1+ \frac{1}{3}(2-\kappa)}}.
    \end{align*}
    From Corollary \ref{cor:B_estimates} (with $\delta=2/3$) and Lemma \ref{lem:interpolation_estimates_epsilon} (with $\delta = 1/3$), we obtain
    \begin{align*}
        \|B(f,g)\|_{\CC^{\alpha+\beta-2+\frac{2}{3}(2-\kappa)}}\|\nabla h\|_{\CC^{\gamma-1+ \frac{1}{3}(2-\kappa)}}
        \lesssim \epsilon^{-1+\frac{\kappa}{4}}\|f\|_{\CC_\epsilon^\alpha}\|g\|_{\CC_\epsilon^\beta}\|h\|_{\CC_\epsilon^\gamma}.
    \end{align*}
\end{proof}

\begin{cor}\label{cor:tilde_B_estimates}
    Let $\alpha, \beta<1$ satisfy $\alpha+\beta>1+\kappa$.
    For $0<\epsilon\leq 1$, $f\in \CC_\epsilon^\alpha$ and $g\in \CC_\epsilon^\beta$, we have
    \begin{align*}
        \|\widetilde{B}(f,g)\|_{\CC^{\alpha+\beta-2-\kappa}} \lesssim \epsilon^{-1+\frac{\kappa}{4}}\|f\|_{\CC_\epsilon^{\alpha}}\|g\|_{\CC_\epsilon^{\beta}}.
    \end{align*}
\end{cor}
\begin{proof}
    From Lemma \ref{lem:product_estimates}, we have
    \begin{align*}
        \|\nabla\cdot(\Delta f \nabla g)\|_{\CC^{\alpha+\beta-2-\kappa}}
        &\lesssim \|\Delta f \nabla g\|_{\CC^{\alpha+\beta-1-\kappa}}\\
        &\lesssim \epsilon^{-1+\frac{\kappa}{4}}\|\Delta f\|_{\CC_\epsilon^{\alpha-2}} \|\nabla g\|_{\CC_\epsilon^{\beta-1}} \\
        &\lesssim \epsilon^{-1+\frac{\kappa}{4}}\|f\|_{\CC_\epsilon^{\alpha}} \|g\|_{\CC_\epsilon^{\beta}}
    \end{align*}
    if $\alpha<2$, $\beta<1$ and $\alpha+\beta>1+\kappa$.
    Similarly, we have
    \begin{align*}
        \|\nabla\cdot (\nabla f \Delta g)\|_{\CC^{\alpha+\beta-2-\kappa}}
        \lesssim \epsilon^{-1+\frac{\kappa}{4}}\|f\|_{\CC_\epsilon^{\alpha}} \|g\|_{\CC_\epsilon^{\beta}}
    \end{align*}
    if $\alpha<1$, $\beta<2$ and $\alpha+\beta>1+\kappa$.
    From Corollary \ref{cor:B_estimates}, we have
    \begin{align*}
        \|\Delta B(f,g)\|_{\CC^{\alpha+\beta-2-\kappa}}
        &\lesssim \|B(f,g)\|_{\CC^{\alpha+\beta-\kappa}}
        \lesssim \epsilon^{-1+\frac{\kappa}{4}} \|f\|_{\CC_\epsilon^\alpha}\|g\|_{\CC_\epsilon^\beta}
    \end{align*}
    if $\alpha, \beta<1$ and $\alpha+\beta>\kappa$.
    From these, we obtain the estimates of $\widetilde{B}$.
\end{proof}

\subsubsection{Proof of Proposition \ref{prop:Right-hand_side_estimates}}
First, we estimate $Z^{(i)}_\epsilon$ for $i=0,1,2$ which is the coefficients of $\Phi_\epsilon$.
Recall the definitions (\ref{eq:defi_of_Z_2}) to (\ref{eq:defi_of_Z_0}) of $Z^{(0)}_\epsilon$, $Z^{(1)}_\epsilon$ and $Z^{(2)}_\epsilon$ and the definitions (\ref{eq:defi_of_tilde_Z_2}) to (\ref{eq:defi_of_tilde_Z_0}) of $\widetilde{Z}^{(0)}_\epsilon$, $\widetilde{Z}^{(1)}_\epsilon$ and $\widetilde{Z}^{(2)}_\epsilon$.
\begin{lem}
    \label{lem:Z_estimates}
    For any $0<t\leq T$ and $0\leq \epsilon \leq 1$, we have
    \begin{empheq}[left={\|Z_\epsilon^{(j)}(t)\|_{\CC^{-1/2-\kappa}}\leq \empheqlbrace}]{alignat*=2}
        &C &&\quad (j=1,2),\\
        &C t^{-2\kappa}(1+\epsilon^{\frac{\kappa}{4}}) &&\quad(j=0),
    \end{empheq}
    and 
    \begin{empheq}[left={\|(Z^{(j)}_\epsilon(\EX_\epsilon^{(1)})-Z^{(j)}_\epsilon(\EX_\epsilon^{(2)}))(t)\|_{\CC^{-1/2-\kappa}}\leq \empheqlbrace}]{alignat*=2}
        &C' \|\EX_\epsilon^{(1)}-\EX_\epsilon^{(2)}\|_{\CX_T^\kappa} &&\quad (j=1,2) \\
        &C't^{-2\kappa}\|\EX_\epsilon^{(1)}-\EX_\epsilon^{(2)}\|_{\CX_T^\kappa} &&\quad (j=0) . 
    \end{empheq}
    Here, $C$ is a positive constant depending only on $\kappa$ and $\|\EX_\epsilon\|_{\CX_T^\kappa}$ and 
    $C'$ is a positive constant depending only on $\kappa$ and $\|\EX_\epsilon^{(i)}\|_{\CX_T^\kappa}$.

    Moreover, we have
    \begin{empheq}[left={\|(Z^{(j)}_\epsilon(\EX_\epsilon)-Z^{(j)}_0(\EX))(t)\|_{\CC^{-1/2-\kappa}}\leq \empheqlbrace}]{alignat*=2}
        & C'' \|\EX_\epsilon-\EX\|_{\CX_T^\kappa} &&(j=2), \\
        &C''(\epsilon^{\frac{\kappa}{4}} + \|\EX_\epsilon-\EX\|_{\CX_T^\kappa}) &&(j=1),\\
        &C''t^{-2\kappa}(\epsilon^{\frac{\kappa}{4}} + \|\EX_\epsilon-\EX\|_{\CX_T^\kappa}) && (j=0). 
    \end{empheq}
    Here, $C''$ is a positive constant depending only on $\kappa$, $\|\EX\|_{\CX_T^\kappa}$ and $\|\EX_\epsilon\|_{\CX_T^\kappa}$.
\end{lem}
Only the estimates of $\overline{F}_\epsilon$ and $\overline{G}_\epsilon$ in $Z^{(i)}_\epsilon$ are non-trivial.
Therefore, by estimating them, we can prove Lemma \ref{lem:Z_estimates}.

\begin{lem}
    \label{lem:bar_F_estimates}
    For any $0<t\leq T$ and $0\leq \epsilon\leq 1$, we have
    \begin{align}
        \label{eq:lem_bar_F}
        \|\overline{F}_\epsilon(t)\|_{\CC^{-1/3-4\kappa/3}} &\leq C,\\
        \|\overline{F}_\epsilon(\ICC_\epsilon^{(1)})(t)-\overline{F}_\epsilon(\ICC_\epsilon^{(2)})(t)\|_{\CC^{-1/3-4\kappa/3}} &\leq C'(\|\EX_\epsilon^{(1)}-\EX_\epsilon^{(2)}\|_{\CX_T^\kappa}).
    \end{align}
    Here, $C$ is a positive constant depending only on $\kappa$ and $\|\EX_\epsilon\|_{\CX_T^\kappa}$ and 
    $C'$ is a positive constant depending only on $\kappa$ and $\|\EX_\epsilon^{(i)}\|_{\CX_T^\kappa}$.

    Moreover, we have
    \begin{equation}
        \|\overline{F}_\epsilon(\ICC_\epsilon)(t)-\overline{F}_0(\ICC)(t)\|_{\CC^{-1/3-4\kappa/3}} \leq C''(\epsilon^{\frac{\kappa}{4}}+\|\EX-\EX_\epsilon\|_{\CX_T^\kappa}).
    \end{equation}
    Here, $C''$ is a positive constant depending only on $\kappa$, $\|\EX\|_{\CX_T^\kappa}$ and $\|\EX_\epsilon\|_{\CX_T^\kappa}$.
\end{lem}
\begin{proof}
    Recall the definition (\ref{eq:defi_of_bar_F}) of $\overline{F}_\epsilon$.
    By the definition of $\EX_\epsilon$, it holds
    \begin{align*}
        &\|(\nabla \ICC_\epsilon \pl \nabla \ICC_\epsilon)(t)\|_{\CC^{-2\kappa}}
        \lesssim \|\nabla \ICC_\epsilon(t)\|_{\CC^{-\kappa}}^2
        \lesssim_{\EX_\epsilon} 1 ,
        &&\|\CD_\epsilon (t)\|_{\CC^{-\kappa}} \lesssim_{\EX_\epsilon} 1
    \end{align*}
    and from Lemma \ref{lem:interpolation_estimates_epsilon} (with $\delta = 1/2$), we have
    \begin{align*}
        &\|\epsilon (\Delta \ICC_\epsilon \pl \Delta \ICC_\epsilon)(t)\|_{\CC^{-3\kappa}}
        \lesssim \epsilon \|\Delta \ICC_\epsilon(t)\|_{\CC^{-3\kappa/2}}^2
        \lesssim_{\EX_\epsilon} \epsilon^{\frac{\kappa}{4}}.
    \end{align*}
    From Corollaries \ref{cor:B_estimates} to \ref{cor:tilde_B_estimates}, we have
    \begin{align*}
        &\|\epsilon B(\ICC_\epsilon,\ICC_\epsilon)(t)^2\|_{\CC^{1/2-5\kappa/4}} \lesssim \epsilon (\epsilon^{-\frac{1}{2}+\frac{\kappa}{8}}\|\ICC_\epsilon(t)\|_{\CC_\epsilon^{1-\kappa}}^2)^2\lesssim_{\EX_\epsilon} \epsilon^{\frac{\kappa}{4}}\\
        &\|\epsilon T(\ICC_\epsilon, \ICC_\epsilon, \ICC_\epsilon)(t)\|_{\CC^{-1/3-4\kappa/3}} \lesssim \epsilon^{\frac{\kappa}{4}}\|\ICC_\epsilon(t)\|_{\CC_\epsilon^{1-\kappa}}^3\lesssim_{\EX_\epsilon} \epsilon^{\frac{\kappa}{4}},\\
        &\|\epsilon\widetilde{B}(\ICC_\epsilon,\ICC_\epsilon)(t)\|_{\CC^{-3\kappa}} \lesssim \epsilon^{\frac{\kappa}{4}}\|\ICC_\epsilon(t)\|_{\CC_\epsilon^{1-\kappa}}^2 \lesssim_{\EX_\epsilon} \epsilon^{\frac{\kappa}{4}}.
    \end{align*}
    From these, we obtain (\ref{eq:lem_bar_F}).
    The remaining estimates can be obtained in the same way.
\end{proof}

\begin{lem}\label{lem:widehat_G_epsilon_u_estimates}
    For any $0<t\leq T$ and $0\leq \epsilon\leq 1$,  we have
    \begin{align}
        \|\overline{G}_\epsilon (t)\|_{\CC^{-1/3-4\kappa/3}} &\leq C t^{-2\kappa},
        \label{eq:lem_bar_G}\\
        \|\overline{G}_\epsilon(\EX_\epsilon^{(1)}) (t) - \overline{G}_\epsilon(\EX_\epsilon^{(2)}) (t)\|_{\CC^{-1/3-4\kappa/3}}
        &\leq C' t^{-2\kappa}\|\EX_\epsilon^{(1)}-\EX_\epsilon^{(2)}\|_{\CX_T^\kappa}.
    \end{align}
    Here, $C$ is a positive constant depending only on $\kappa$ and $\|\EX_\epsilon\|_{\CX_T^\kappa}$ and 
    $C'$ is a positive constant depending only on $\kappa$ and $\|\EX_\epsilon^{(i)}\|_{\CX_T^\kappa}$.

    Moreover, we have
    \begin{align}
        \|\overline{G}_\epsilon(\EX_\epsilon) (t) - \overline{G}_0(\EX) (t)\|_{\CC^{-1/3-4\kappa/3}}
        \leq C'' t^{-2\kappa}(\epsilon^{\frac{\kappa}{4}}+\|\EX-\EX_\epsilon\|_{\CX_T^\kappa}).
    \end{align}
    Here, $C''$ is a positive constant depending only on $\kappa$, $\|\EX\|_{\CX_T^\kappa}$ and $\|\EX_\epsilon\|_{\CX_T^\kappa}$.
\end{lem}
\begin{proof}
    We only show (\ref{eq:lem_bar_G}) by estimating for each $\overline{G}_\epsilon^{(i)}$. 
    The remaining estimates can be obtained in the same way.
    \begin{enumerate}[(1)]
        \item 
        Recall the definition (\ref{eq:defi_of_bar_G_1}) of $\overline{G}_\epsilon^{(1)}$.
        From Corollaries \ref{cor:B_estimates} to \ref{cor:tilde_B_estimates}, we can prove
        \[
            \|\overline{G}_\epsilon^{(1)}(t)\|_{\CC^{-1/3-4\kappa/3}} \lesssim_{\EX_\epsilon} \epsilon^{\frac{\kappa}{4}}
        \]
        in the same way as in Lemma \ref{lem:bar_F_estimates}.
    \item 
        Recall the definition (\ref{eq:defi_of_bar_G_2}) of $\overline{G}_\epsilon^{(2)}$.
        From Lemma \ref{lem:Ly} and Lemma \ref{lem:interpolation_estimates_epsilon} (with $\delta = 1/2$), we have
        \begin{align*}
            \|\nabla \ICC_\epsilon(t) (\pl+\pr) \nabla y_\epsilon(t)\|_{\CC^{-2\kappa}}
            &\lesssim \|\nabla \ICC_\epsilon(t) \|_{\CC^{-\kappa}} \|\nabla y_\epsilon(t)\|_{\CC^{-\kappa}} 
            \lesssim_{\EX_\epsilon} 1,\\
            \|\epsilon \Delta \ICC_\epsilon(t) (\pl+\pr) \Delta y_\epsilon(t) \|_{\CC^{-3\kappa}}
            &\lesssim \epsilon \|\Delta \ICC_\epsilon (t)\|_{\CC^{-3\kappa/2}}\|\Delta y_\epsilon (t)\|_{\CC^{-3\kappa/2}}
            \lesssim_{\EX_\epsilon} \epsilon^{\frac{\kappa}{4}}.
        \end{align*}
        Therefore, it holds
        \[
            \|\overline{G}_\epsilon^{(2)}(t)\|_{\CC^{-3\kappa}} \lesssim_{\EX_\epsilon} 1 + \epsilon^{\frac{\kappa}{4}}.
        \]
    \item 
        Recall the definition (\ref{eq:defi_of_bar_G_3}) of $\overline{G}_\epsilon^{(3)}$.
        From the inequality (\ref{eq:Y_epsilon_estimates}) and Lemma \ref{lem:interpolation_estimates_epsilon} (with $\delta = 1$), we have
        \begin{align*}
            \|\nabla \ICC_\epsilon(t) \rs \nabla Y_\epsilon(t)\|_{\CC^{1/2-4\kappa}}
            &\lesssim \|\nabla \ICC_\epsilon(t)\|_{\CC^{-\kappa}} \|\nabla Y_\epsilon (t)\|_{\CC^{1/2-3\kappa}} \lesssim_{\EX_\epsilon} 1,\\
            \|\epsilon \Delta \ICC_\epsilon(t) \rs \Delta Y_\epsilon (t)\|_{\CC^{1/2-5\kappa}}
            &\lesssim \epsilon \|\Delta \ICC_\epsilon(t)\|_{\CC^{1-2\kappa}} \|\Delta Y_\epsilon(t)\|_{\CC^{-1/2-3\kappa}}
            \lesssim_{\EX_\epsilon} \epsilon^{\frac{\kappa}{4}}.
        \end{align*}
        Therefore, it holds
        \[
            \|\overline{G}_\epsilon^{(3)}(t)\|_{\CC^{1/2-5\kappa}} \lesssim_{\EX_\epsilon} 1 + \epsilon^{\frac{\kappa}{4}}.
        \]
    \item
        Recall the definition (\ref{eq:defi_of_bar_G_4}) of $\overline{G}_\epsilon^{(4)}$.
        From Proposition \ref{prop:Effects_of_heat_semigroup} for $e^{t\Delta}$ and Lemma \ref{lem:interpolation_estimates_epsilon} (with $\delta = 1$), we have
        \begin{align*}
            &\|\nabla \ICC_\epsilon(t) \rs \nabla \{e^{3\ICC_\epsilon(t)}\ICCC_\epsilon(t)\pl P_t^\epsilon \ICC_\epsilon(0)\} \|_{\CC^{2\kappa}}\\
            &\lesssim \|\nabla \ICC_\epsilon(t)\|_{\CC^{-\kappa}} \|\nabla \{e^{3\ICC_\epsilon(t)}\ICCC_\epsilon(t)\pl P_t^\epsilon \ICC_\epsilon(0)\}\|_{\CC^{3\kappa}}\\
            &\lesssim_{\EX_\epsilon} \|P_t^\epsilon \ICC_\epsilon(0)\|_{\CC^{1+3\kappa}}
            \lesssim_{\EX_\epsilon} t^{-2\kappa},
        \end{align*}
        \begin{align*}
            &\| \epsilon \Delta \ICC_\epsilon(t) \rs \Delta \{e^{3\ICC_\epsilon(t)}\ICCC_\epsilon(t)\pl P_t^\epsilon \ICC_\epsilon(0)\} \|_{\CC^{\kappa}} \\
            &\lesssim \epsilon \|\Delta \ICC_\epsilon(t)\|_{\CC^{1-2\kappa}} \| \Delta \{e^{3\ICC_\epsilon(t)}\ICCC_\epsilon(t)\pl P_t^\epsilon \ICC_\epsilon(0)\} \|_{\CC^{-1+3\kappa}} \\
            &\lesssim_{\EX_\epsilon} \epsilon^{\frac{\kappa}{4}} \|P_t^\epsilon \ICC_\epsilon(0)\|_{\CC^{1+3\kappa}} 
            \lesssim_{\EX_\epsilon} \epsilon^{\frac{\kappa}{4}}t^{-2\kappa}.
        \end{align*}
        Therefore, it holds
        \[
            \|\overline{G}_\epsilon^{(4)}(t)\|_{\CC^{\kappa}} \lesssim_{\EX_\epsilon} t^{-2\kappa}(1 + \epsilon^{\frac{\kappa}{4}}).
        \]
    \item 
        Recall the definition (\ref{eq:defi_of_bar_G_5}) of $\overline{G}_\epsilon^{(5)}$.
        From Lemma \ref{lem:interpolation_estimates_epsilon} (with $\delta = 1$), we have
        \begin{align*}
            &\|\nabla \ICC_\epsilon(t) \rs  (\nabla(e^{3\ICC_\epsilon(t)}\ICCC_\epsilon(t))\pl \ICC_\epsilon(t))\|_{\CC^{1/2-3\kappa}}\\
            &\lesssim \|\nabla \ICC_\epsilon(t)\|_{\CC^{-\kappa}} \|\nabla(e^{3\ICC_\epsilon(t)}\ICCC_\epsilon(t))\pl \ICC_\epsilon(t)\|_{\CC^{1/2-2\kappa}} \\
            &\lesssim \|\ICC_\epsilon(t)\|_{\CC^{1-\kappa}} \|\nabla(e^{3\ICC_\epsilon(t)}\ICCC_\epsilon(t))\|_{\CC^{-1/2-\kappa}} \|\ICC_\epsilon(t)\|_{\CC^{1-\kappa}}
            \lesssim_{\EX_\epsilon} 1,
        \end{align*}
        \begin{align*}
            & \|\epsilon \Delta \ICC_\epsilon(t) \rs  (\Delta(e^{3\ICC_\epsilon(t)}\ICCC_\epsilon(t))\pl \ICC_\epsilon(t))\|_{\CC^{1/2-4\kappa}}\\
            & \lesssim \epsilon \|\Delta \ICC_\epsilon(t)\|_{\CC^{1-2\kappa}} \|\Delta(e^{3\ICC_\epsilon(t)}\ICCC_\epsilon(t))\pl \ICC_\epsilon(t)\|_{\CC^{-1/2-2\kappa}} \\
            & \lesssim \epsilon\|\ICC_\epsilon(t)\|_{\CC^{3-2\kappa}} \|\Delta(e^{3\ICC_\epsilon(t)}\ICCC_\epsilon(t))\|_{\CC^{-3/2-\kappa}} \|\ICC_\epsilon(t)\|_{\CC^{1-\kappa}}
            \lesssim_{\EX_\epsilon} \epsilon^{\frac{\kappa}{4}},
        \end{align*}
        \begin{align*}
            &\|\epsilon \Delta \ICC_\epsilon(t) \rs  (\nabla(e^{3\ICC_\epsilon(t)}\ICCC_\epsilon(t))\pl \nabla \ICC_\epsilon(t))\|_{\CC^{1/2-4\kappa}} \\
            &\lesssim \epsilon \|\Delta \ICC_\epsilon(t)\|_{\CC^{1-2\kappa}} \|\nabla(e^{3\ICC_\epsilon(t)}\ICCC_\epsilon(t))\pl \nabla\ICC_\epsilon(t)\|_{\CC^{-1/2-2\kappa}} \\
            & \lesssim \epsilon\|\ICC_\epsilon(t)\|_{\CC^{3-2\kappa}} \|\nabla(e^{3\ICC_\epsilon(t)}\ICCC_\epsilon(t))\|_{\CC^{-1/2-\kappa}} \|\nabla\ICC_\epsilon(t)\|_{\CC^{-\kappa}}
            \lesssim_{\EX_\epsilon} \epsilon^{\frac{\kappa}{4}}.
        \end{align*}
        Therefore, it holds
        \[
            \|\overline{G}_\epsilon^{(5)}(t)\|_{\CC^{1/2-4\kappa}} \lesssim_{\EX_\epsilon}1 + \epsilon^{\frac{\kappa}{4}}.
        \]
    \item 
        Recall the definition (\ref{eq:defi_of_bar_G_6}) of $\overline{G}_\epsilon^{(6)}$.
        From commutator estimates (Proposition \ref{prop:commutator_estimates}) and Lemma \ref{lem:interpolation_estimates_epsilon} (with $\delta = 1/2$), we have
        \begin{align*}
            &\|C(e^{3\ICC_\epsilon}\ICCC_\epsilon, \nabla\ICC_\epsilon,\nabla \ICC_\epsilon )(t)\|_{\CC^{1/2-3\kappa}}\\
            &\lesssim \|e^{3\ICC_\epsilon(t)}\ICCC_\epsilon(t)\|_{\CC^{1/2-\kappa}}\|\nabla\ICC_\epsilon(t)\|_{\CC^{-\kappa}}^2
            \lesssim_{\EX_\epsilon} 1,
        \end{align*}
        \begin{align*}        
            &\|\epsilon C(e^{3\ICC_\epsilon}\ICCC_\epsilon, \Delta\ICC_\epsilon,\Delta \ICC_\epsilon )(t)\|_{\CC^{1/2-4\kappa}}\\
            &\lesssim \epsilon \|e^{3\ICC_\epsilon(t)}\ICCC_\epsilon(t)\|_{\CC^{1/2-\kappa}} \|\Delta\ICC_\epsilon(t)\|_{\CC^{-3\kappa/2}}^2
            \lesssim_{\EX_\epsilon} \epsilon^{\frac{\kappa}{4}},
        \end{align*}
        \begin{align*}        
            \|e^{3\ICC_\epsilon(t)} \ICCC_\epsilon(t)\CD_\epsilon(t)\|_{\CC^{-\kappa}}
            &\lesssim_{\EX_\epsilon} 1. 
        \end{align*}
    \end{enumerate}
    Therefore, it holds
        \[
            \|\overline{G}_\epsilon^{(6)}(t)\|_{\CC^{-\kappa}} \lesssim_{\EX_\epsilon}1 + \epsilon^{\frac{\kappa}{4}}.
        \]
\end{proof}
We have obtained Lemma \ref{lem:Z_estimates}.
Next, we estimate $\Phi_\epsilon$.
\begin{lem}\label{lem:Phi_estimates}
    For any  $0<t\leq T$ and $0\leq\epsilon\leq 1$, we have the following:
    \begin{enumerate}[(1)]
        \item 
            For $u_\epsilon \in \CE_T^{-1/2-\kappa}\CC^{3/2-2\kappa}$, we have
            \begin{align}
                \|\Phi_\epsilon(\EX_\epsilon, u_\epsilon)(t)\|_{\CC^{-1/2-\kappa}}
                \leq C t^{-1+\frac{\kappa}{2}}(1+\|u_\epsilon\|_{\CE_T^{-1/2-\kappa}\CC^{3/2-2\kappa}}^3).\label{eq:Phi_estimates}
            \end{align}
            Here, $C$ is a positive constant depending only on $\kappa$, $\kappa'$ and $\|\EX_\epsilon\|_{\CX_T^\kappa}$.
        \item 
            For any $u_\epsilon^{(1)}, u_\epsilon^{(2)} \in \CE_T^{-1/2-\kappa}\CC^{3/2-2\kappa}$, we have
            \begin{align}
                &\|\Phi_\epsilon(\EX^{(1)}, u_\epsilon^{(1)})(t)-\Phi_\epsilon(\EX^{(2)},u_\epsilon^{(2)})(t)\|_{\CC^{-1/2-\kappa}} \notag\\
                &\leq  t^{-1+\frac{\kappa}{2}} (C'_1\|\EX_\epsilon^{(1)}-\EX_\epsilon^{(2)}\|_{\CX_T^\kappa} + C'_2\|u_\epsilon^{(1)}-u_\epsilon^{(2)}\|_{\CE_T^{-1/2-\kappa}\CC^{3/2-2\kappa}}).
            \end{align}
            Here, $C'_1$ and $C'_2$ are positive constants depending only on $\kappa$, $\kappa'$, $\|\EX_\epsilon^{(i)}\|_{\CX_T^\kappa}$ and $\|u_\epsilon^{(i)}\|_{\CE_T^{-1/2-\kappa}\CC^{3/2-2\kappa}}$.
            In particular, $C'_1$ is given by at most third-order polynomial and $C'_2$ is given by at most second-order polynomial in $\|u_\epsilon^{(i)}\|_{\CE_T^{-1/2-\kappa}\CC^{3/2-2\kappa}}$.
    \end{enumerate}
    Moreover, for any $u, u_\epsilon\in \CE_T^{-1/2-\kappa}\CC^{3/2-2\kappa}$, we have
    \begin{align}
        &\|\Phi_\epsilon(\EX_\epsilon, u_\epsilon)(t)-\Phi_0(\EX,u)(t)\|_{\CC^{-1/2-\kappa}} \notag\\
        &\leq  t^{-1+\frac{\kappa}{2}} \{C''_1(\epsilon^{\frac{\kappa}{4}} + \|\EX-\EX_\epsilon\|_{\CX_T^\kappa}) + C''_2\|u-u_\epsilon\|_{\CE_T^{-1/2-\kappa}\CC^{3/2-2\kappa}}\}.
    \end{align}
    Here, $C''_1$ and $C''_2$ are positive constants depending only on $\kappa$, $\kappa'$, $\|\EX\|_{\CX_T^\kappa}$, $\|\EX_\epsilon\|_{\CX_T^\kappa}$, $\|u\|_{\CE_T^{-1/2-\kappa}\CC^{3/2-2\kappa}}$ and $\|u_\epsilon\|_{\CE_T^{-1/2-\kappa}\CC^{3/2-2\kappa}}$.
    In particular, $C''_1$ is given by at most third-order polynomial and $C''_2$ is given by at most second-order polynomial in $\|u\|_{\CE_T^{-1/2-\kappa}\CC^{3/2-2\kappa}}$ and $\|u_\epsilon\|_{\CE_T^{-1/2-\kappa}\CC^{3/2-2\kappa}}$.
\end{lem}

\begin{proof}
    Recall the definition (\ref{eq:defi_of_Phi}) of $\Phi_\epsilon$.
    First we consider the cubic term.
    By using Lemma \ref{lem:interpolation_estimates}, we have
    \begin{align*}
        \|-e^{-6\ICC_\epsilon(t)}u_\epsilon(t)^3\|_{L^\infty}
        &\lesssim \|e^{-6\ICC_\epsilon}\|_{C_T\CC_\epsilon^{1-\kappa}} \|u_\epsilon(t)\|_{L^\infty}^3
        \lesssim_{\EX_\epsilon} t^{\frac{3\beta}{2}}\|u\|_{\CE_T^\beta\CC^\alpha}^3.
    \end{align*}
    for every $\beta<0$ and $\alpha > 0$.

    Next, we consider the square term.
    By using the interpolation estimate (Lemma \ref{lem:interpolation_estimates}), we have
    \begin{align*}
        \|u_\epsilon(t)\|_{\CC^{1/2+\kappa'}}^2
        &\lesssim \|u_\epsilon(t)\|_{L^\infty}\|u_\epsilon(t)\|_{\CC^{1/2+\kappa'}}\\
        &\lesssim t^{\frac{\beta}{2}}t^{-\frac{1/2+\kappa'-\beta}{2}}\|u_\epsilon\|_{\CE_T^\beta\CC^\alpha}^2\\
        &= t^{-\frac{1+2\kappa'}{4}+\beta}\|u_\epsilon\|_{\CE_T^\beta\CC^\alpha}^2
    \end{align*}
    for every $\kappa' > \kappa$, $\beta\leq 0$ and $\alpha\geq 1/2+\kappa'$.
    Hence, we have
    \begin{align*}
        \|Z_\epsilon^{(2)}(t)u_\epsilon^2(t)\|_{\CC^{-1/2-\kappa}}
        &\lesssim \|Z_\epsilon^{(2)}\|_{C_T\CC^{-1/2-\kappa}} \|u_\epsilon(t)\|_{\CC^{1/2+\kappa'}}^2 \\
        &\lesssim_{\EX_\epsilon} t^{-\frac{1+2\kappa'}{4}+\beta}\|u_\epsilon\|_{\CE_T^\beta\CC^\alpha}^2\\
        &\lesssim t^{-\frac{1+2\kappa'}{4}+\beta}(1+\|u_\epsilon\|_{\CE_T^\beta\CC^\alpha}^3).
    \end{align*}

    At last, we consider the linear terms.
    From Lemma \ref{lem:interpolation_estimates}, we have
    \begin{align*}
        &\|\nabla \ICC_\epsilon(t) \cdot \nabla u_\epsilon(t)+Z_\epsilon^{(1)}(t)u_\epsilon(t) + Z_\epsilon^{(0)}(t)\|_{\CC^{-1/2-\kappa}}\\
        &\lesssim \|\nabla \ICC_\epsilon(t)\|_{\CC^{-\kappa}}\|\nabla u_\epsilon(t)\|_{\CC^{\alpha-1}} + \|Z_\epsilon^{(1)}\|_{\CC^{-1/2-\kappa}}\|u_\epsilon(t)\|_{\CC^{\alpha}} + \|Z_\epsilon^{(0)}(t)\|_{\CC^{-1/2-\kappa}} \\
        &\lesssim_{\EX_\epsilon} \|u(t)\|_{\CC^{\alpha}} + t^{-2\kappa} \\
        &\lesssim t^{-\frac{\alpha-\beta}{2}}\|u\|_{\CE_T^\beta\CC^\alpha}+ t^{-2\kappa} \\
        &\lesssim t^{-\frac{\alpha-\beta}{2}}(1+\|u_\epsilon\|_{\CE_T^\beta\CC^\alpha}^3).
    \end{align*}
    for every $\alpha > 1+\kappa$ and $\beta<\alpha-4\kappa$.

    Taking $\alpha=3/2-2\kappa$ and $\beta=-1/2-\kappa$, we obtain (\ref{eq:Phi_estimates}).
    The remaining estimates can be obtained in the same way.
\end{proof}

At last, we estimate $\widetilde{G}_\epsilon(u_\epsilon)$ by using $\epsilon$-dependent estimates of $u_\epsilon$.
\begin{lem}\label{lem:G_epsilon_u_estimates}
    For any $0<t\leq T$ and $0\leq\epsilon\leq 1$,  we have the following:
    \begin{enumerate}[(1)]
        \item 
            For $u_\epsilon \in \CE_T^{-1/2-\kappa}\CC_\epsilon^{3/2-2\kappa}$, we have
            \begin{align}
                \label{eq:lem_G_epsilon_u}
                \|\widetilde{G}_\epsilon(u_\epsilon)(t)\|_{\CC^{-1/3-4\kappa/3}}
                \leq C\epsilon^{\frac{\kappa}{4}}t^{-1+\frac{\kappa}{2}}\|u_\epsilon\|_{\CE_T^{-1/2-\kappa}\CC_\epsilon^{3/2-2\kappa}}.
            \end{align}
            Here, $C$ is a positive constant depending only on $\kappa$ and $\|\EX_\epsilon\|_{\CX_T^\kappa}$.
        \item 
            For any $u^{(1)}_\epsilon, u^{(2)}_\epsilon \in \CE_T^{-1/2-\kappa}\CC_\epsilon^{3/2-2\kappa}$, we have
            \begin{multline}
                \|\widetilde{G}_\epsilon(\EX_\epsilon^{(1)},u_\epsilon^{(1)})(t)-\widetilde{G}_\epsilon(\EX_\epsilon^{(2)},u_\epsilon^{(2)})(t)\|_{\CC^{-1/3-4\kappa/3}}\\
                \leq C' t^{-1+\frac{\kappa}{2}} (\|\EX_\epsilon^{(1)} -\EX_\epsilon^{(2)}\|_{\CX_T^\kappa} +\|u_\epsilon^{(1)}-u_\epsilon^{(2)} \|_{\CE_T^{-1/2-\kappa}\CC_\epsilon^{3/2-2\kappa}}).
            \end{multline}
            Here, $C'$ is a positive constant depending only on $\kappa$, $\|\EX_\epsilon^{(i)}\|_{\CX_T^\kappa}$ and $\|u_\epsilon^{(i)}\|_{\CE_T^{-1/2-\kappa}\CC_\epsilon^{3/2-2\kappa}}$.
            In particular, it is given by a first-order polynomial in $\|u_\epsilon^{(i)}\|_{\CE_T^{-1/2-\kappa}\CC_\epsilon^{3/2-2\kappa}}$.
    \end{enumerate}
\end{lem}
\begin{proof}
    Recall the definition (\ref{eq:defi_of_G}) of $\widetilde{G}_\epsilon$.
    From Lemma \ref{lem:interpolation_estimates_epsilon} (with $\delta = 1/2$), we have
    \begin{align*}
        \|\epsilon \Delta \ICC_\epsilon(t) \Delta u_\epsilon(t)\|_{\CC^{-3\kappa/2}}
        &\lesssim \epsilon\|\Delta \ICC_\epsilon(t)\|_{\CC^{-3\kappa/2}}\|\Delta u_\epsilon(t)\|_{\CC^{\alpha-1-\kappa/2}}\\
        &\lesssim \epsilon\|\ICC_\epsilon(t)\|_{\CC^{2-3\kappa/2}}\|u_\epsilon(t)\|_{\CC^{\alpha+1-\kappa/2}}\\
        &\lesssim_{\EX_\epsilon} \epsilon^{\frac{\kappa}{4}} \|u_\epsilon(t)\|_{\CC_\epsilon^\alpha}
    \end{align*}
    if $\alpha>1+2\kappa$.
    Moreover, from Corollaries \ref{cor:B_estimates} to \ref{cor:tilde_B_estimates}, we have
    \begin{align*}
        \|\epsilon B(\ICC_\epsilon,\ICC_\epsilon)(t)B(\ICC_\epsilon,u_\epsilon)(t)\|_{\CC^{1/2-5\kappa/4}}
        &\lesssim_{\EX_\epsilon} \epsilon^{\frac{\kappa}{4}}\|u_\epsilon(t)\|_{\CC_\epsilon^{1-\kappa}}\lesssim \epsilon^{\frac{\kappa}{4}}\|u_\epsilon(t)\|_{\CC_\epsilon^{\alpha}},\\
        \|\epsilon T(u_\epsilon,\ICC_\epsilon,\ICC_\epsilon)(t)\|_{\CC^{-1/3-4\kappa/3}}
        &\lesssim_{\EX_\epsilon} \epsilon^{\frac{\kappa}{4}}\|u_\epsilon(t)\|_{\CC_\epsilon^{1-\kappa}}\lesssim \epsilon^{\frac{\kappa}{4}}\|u_\epsilon(t)\|_{\CC_\epsilon^{\alpha}},  \\
        \|\epsilon T(\ICC_\epsilon,\ICC_\epsilon,u_\epsilon)(t)\|_{\CC^{-1/3-4\kappa/3}}
        &\lesssim_{\EX_\epsilon} \epsilon^{\frac{\kappa}{4}}\|u_\epsilon(t)\|_{\CC_\epsilon^{1-\kappa}}\lesssim \epsilon^{\frac{\kappa}{4}}\|u_\epsilon(t)\|_{\CC_\epsilon^{\alpha}}, \\
        \|\epsilon\widetilde{B}(\ICC_\epsilon,u_\epsilon)(t)\|_{\CC^{-3\kappa}}
        &\lesssim_{\EX_\epsilon}\epsilon^{\frac{\kappa}{4}}\|u_\epsilon(t)\|_{\CC_\epsilon^{1-\kappa}}\lesssim\epsilon^{\frac{\kappa}{4}}\|u_\epsilon(t)\|_{\CC_\epsilon^{\alpha}}
    \end{align*}
    if $\alpha\geq 1-\kappa$.
    Here, from the definition of $\|\cdot\|_{\CE_T^\beta\CC_\epsilon^\alpha}$, we have
    \[
        \|u_\epsilon(t)\|_{\CC_\epsilon^\alpha}
        \lesssim t^{-\frac{\alpha-\beta}{2}} \|u\|_{\CE_T^\beta\CC_\epsilon^\alpha}
    \]
    for every $\beta \leq \alpha$.
    By taking $\alpha=3/2-2\kappa$ and $\beta=-1/2-\kappa$, we obtain (\ref{eq:lem_G_epsilon_u}).
    The remaining estimates can be obtained in the same way.
\end{proof}

\subsection{Local well-posedness of $u_\epsilon$}
\label{sec:LWP}
We show that the approximating dynamical $\Phi^4_3$ equation is locally well-posed.
\begin{thm}
    \label{thm:LWP_of_u_epsilon}
    Let $0 < \kappa < 1/8$ and $0\leq \epsilon\leq 1$.
    There exists a continuous function $\widetilde{T}_*:\CC^{-1/2-\kappa} \times \CX^\kappa_1 \to (0,1]$ such that the following (1) and (2) hold:
    \begin{enumerate}[(1)]
        \item For every $u(0) \in \CC^{-1/2-\kappa}$ and $\EX_\epsilon \in \CX_1^\kappa$, we set $T_*:=\widetilde{T}_*(u(0),\EX_\epsilon)$.
        Then, the equation (\ref{eq:L_epsilon_u_epsilon}) with the initial condition $u(0)$ admits a unique solution $u_\epsilon \in \CE_{T_*}^{-1/2-\kappa}\CC_\epsilon^{3/2-2\kappa}$ and there is a positive constant $C$ depending only on $\kappa$, $T_*$ and $\|\EX_\epsilon\|_{\CX_1^\kappa}$ such that
        \[
             \|u_\epsilon\|_{\CE_{T_*}^{-1/2-\kappa}\CC_\epsilon^{3/2-2\kappa}}
             \leq C (1\vee\|u(0)\|_{\CC_\epsilon^{-1/2-\kappa}}).
        \]
        \item In addition to (1), let $\{u^{(n)}(0)\}_{n=1}^\infty \subset \CC^{-1/2-\kappa}$ and $\{\EX_\epsilon^{(n)}\}_{n=1}^{\infty} \subset \CX_1^\kappa$ converge to $u(0)$ in $\CC^{-1/2-\kappa}$ and $\EX_\epsilon$ in $\CX_1^{\kappa}$, respectively.
        Set $T_*^{(n)}:=\widetilde{T}_*(u^{(n)}(0),\EX_\epsilon^{(n)})$ and let $u_\epsilon^{(n)}$ be a unique solution on $[0,T_*^{(n)}]$ to the equation (\ref{eq:L_epsilon_u_epsilon}) with the initial condition $u^{(n)}(0)$ driven by $\EX_\epsilon^{(n)}$.
        Then, for every $0<t<T_*$ we have
        \[
            \lim_{n\to \infty} \|u_\epsilon^{(n)}-u_\epsilon\|_{\CE_t^{-1/2-\kappa}\CC_\epsilon^{3/2-2\kappa}}=0.
        \]
    \end{enumerate}
\end{thm}

\begin{proof}
    Fix $M>0$ and $T>0$.
    We set 
    \[
        \CM_\epsilon u_\epsilon(t)
        := P^\epsilon_t u(0) + \CL_\epsilon^{-1}[\CR_\epsilon(\EX_\epsilon,u_\epsilon)](t).
    \]
    First, we show that $\CM_\epsilon$ is contraction map on some ball in $\CE_T^{-1/2-\kappa}\CC_\epsilon^{3/2-\kappa}$.
    To do it, we will estimate $\CM_\epsilon u_\epsilon$.
    By using Proposition \ref{prop:Effects_of_heat_semigroup} for $P^\epsilon_t u(0)$, we have
    \begin{align*}
        \|P^\epsilon_t u(0)\|_{\CC^{-1/2-\kappa}}&\lesssim \|u(0)\|_{\CC^{-1/2-\kappa}},\\
        t^{1-\frac{\kappa}{2}}\|P^\epsilon_t u(0)\|_{\CC^{3/2-2\kappa}}&\lesssim \|u(0)\|_{\CC^{-1/2-\kappa}},\\
        \epsilon^{1-\frac{\kappa}{4}} t^{1-\frac{\kappa}{2}} \|P^\epsilon_t u(0)\|_{\CC^{7/2-3\kappa}}
        &\lesssim\epsilon^{1-\frac{\kappa}{4}} t^{1-\frac{\kappa}{2}} (\epsilon t)^{-\frac{4-2\kappa}{4}} \|u(0)\|_{\CC^{-1/2-\kappa}}
        \leq \|u(0)\|_{\CC^{-1/2-\kappa}}.
    \end{align*}
    Therefore, it follows
    \[
        \|P^\epsilon_\cdot u(0)\|_{\CE_T^{-1/2-\kappa}\CC_\epsilon^{3/2-2\kappa}} \lesssim \|u(0)\|_{\CC^{-1/2-\kappa}}.
    \]
    By using Propositions \ref{prop:Schauder_estimates_2} and \ref{prop:Right-hand_side_estimates} for the second term of $\CM_\epsilon u_\epsilon$, we have
    \begin{align*}
        &\sup_{0<t\leq T} \|\CL_\epsilon^{-1}[\CR_\epsilon(\EX_\epsilon,u_\epsilon)](t)\|_{\CC^{-1/2-\kappa}}+ \sup_{0<t\leq T} t^{1-\frac{\kappa}{2}} \|\CL_\epsilon^{-1}[\CR_\epsilon(\EX_\epsilon,u_\epsilon)](t)\|_{\CC_\epsilon^{3/2-2\kappa}}\\
        & \lesssim T^{\frac{\kappa}{2}} \left(\sup_{0<t\leq T} t^{1-\frac{\kappa}{2}} \|\CR_\epsilon(\EX_\epsilon,u_\epsilon) (t)\|_{\CC^{-1/2-\kappa}} \right) \\
        &\lesssim T^{\frac{\kappa}{2}}(1+\|u_\epsilon\|_{\CE_T^{-1/2-\kappa}\CC_\epsilon^{3/2-2\kappa}}^3).
    \end{align*}
    Therefore, we have
    \[
        \|\CM_\epsilon u_\epsilon\|_{\CE_T^{-1/2-\kappa} \CC_\epsilon^{3/2-2\kappa}}
        \leq C_1 \|u(0)\|_{\CC^{-1/2-\kappa}} + C_2T^{\frac{\kappa}{2}} (1+\|u_\epsilon\|_{\CE_T^{-1/2-\kappa}\CC_\epsilon^{3/2-2\kappa}}^3).
    \]
    Here, $C_1$ and $C_2$ are positive constants and $C_2$ depends only on $\kappa$, $\kappa'$ and $\|\EX_\epsilon\|_{\CX_1^\kappa}$.
    Similarly, we have
    \begin{align*}
        &\|\CM_\epsilon u_\epsilon^{(1)} - \CM_\epsilon u_\epsilon^{(2)}\|_{\CE_T^{-1/2-\kappa}\CC_\epsilon^{3/2-2\kappa}}\\
        &=\|\CL_\epsilon^{-1}[\CR_\epsilon(\EX_\epsilon,u_\epsilon^{(1)})] - \CL_\epsilon^{-1}[\CR_\epsilon(\EX_\epsilon,u_\epsilon^{(2)})]\|_{\CE_T^{-1/2-\kappa}\CC_\epsilon^{3/2-2\kappa}}\\
        &\leq C_3T^{\frac{\kappa}{2}}\|u_\epsilon^{(1)}-u_\epsilon^{(2)} \|_{\CE_T^{-1/2-\kappa}\CC_\epsilon^{3/2-2\kappa}}
    \end{align*}
    from Propositions \ref{prop:Right-hand_side_estimates} and \ref{prop:Schauder_estimates_2}.
    Here, $C_3$ is a positive constant depending only on $\kappa$, $\kappa'$, $\|\EX_\epsilon^{(i)}\|_{\CX_1^\kappa}$ and $\|u_\epsilon^{(i)}\|_{\CE_T^{-1/2-\kappa}\CC_\epsilon^{3/2-2\kappa}}$.
    In particular $C_3$ is given by at most second-order polynomial in $\|u_\epsilon^{(i)}\|_{\CE_T^{-1/2-\kappa}\CC_\epsilon^{3/2-2\kappa}}$.
    Then, 
    \[
        C_3 \leq C_3'(1 +  \|u_\epsilon^{(1)}\|_{\CE_T^{-1/2-\kappa}\CC_\epsilon^{3/2-2\kappa}}^2 + \|u_\epsilon^{(2)}\|_{\CE_T^{-1/2-\kappa}\CC_\epsilon^{3/2-2\kappa}}^2)
    \]
    for some constant $C_3'>0$.
    Let
    \[
        M_*:=C_1\|u(0)\|_{\CC^{{-1/2-\kappa}}}\vee 1
    \]
    and 
    \[
        T_*:=\widetilde{T}_*(u(0),\EX_\epsilon):=\left( \frac{1}{2(C_2\vee C_3')(1+8M_*^2)} \right)^{2/\kappa}\wedge 1.
    \]
    Let
    \[
        \CB_{T_*,M_*}:=\{u_\epsilon\in \CE_{T_*}^{-1/2-\kappa}\CC_\epsilon^{3/2-2\kappa};\|u_\epsilon\|_{\CE_{T_*}^{-1/2-\kappa}\CC_\epsilon^{3/2-2\kappa}}\leq 2M_*\}.
    \]
    If $u, u^{(1)},u_\epsilon^{(2)} \in \CB_{T_*,M_*}$, then we have
    \begin{align*}
        \|\CM_\epsilon u_\epsilon\|_{\CE_{T_*}^{-1/2-\kappa}\CC_\epsilon^{3/2-2\kappa}} &\leq 2M_*\\
        \|\CM_\epsilon u_\epsilon^{(1)} - \CM_\epsilon u_\epsilon^{(2)}\|_{\CE_{T_*}^{-1/2-\kappa}\CC_\epsilon^{3/2-2\kappa}} &\leq \frac{1}{2}\|u_\epsilon^{(1)}-u_\epsilon^{(2)}\|_{\CE_{T_*}^{-1/2-\kappa}\CC_\epsilon^{3/2-2\kappa}}.
    \end{align*}
    Thus, $\CM_\epsilon$ is a contraction map on $\CB_{T_*,M_*}$.
    Hence, there exists a unique fixed point $u_\epsilon$ in $\CB_{T_*,M_*}$ by Banach's fixed point theorem.

    Next, we show that the solution on $[0,T_*]$ is unique.
    Let $u^{(1)}_\epsilon, u^{(2)}_\epsilon$ be solutions with a common initial condition $u(0)$ and we show that $u^{(1)}_\epsilon=u^{(2)}_\epsilon$.
    Taking $M>0$ such that
    \[
        \|u^{(i)}_\epsilon\|_{\CE_{T_*}^{-1/2-\kappa}\CC_\epsilon^{3/2-2\kappa}}\leq 2M,
    \]
    we can show that $\CM_\epsilon$ is a contraction on $\CB_{T_{**},M}$ by the similar arguments as above, where $T_{**}(\leq T_*)$ depends on $M$.
    Hence $u^{(1)}_\epsilon$ and $u^{(2)}_\epsilon$ coincide on $[0,T_{**}]$.
    Let $\tilde{u}^{(i)}_\epsilon(t):=u^{(i)}_\epsilon(t+T_{**})$, then it satisfies
    \[
        \|\tilde{u}^{(i)}_\epsilon\|_{\CE_{T_{*}-T_{**}}^{{-1/2-\kappa}}\CC_\epsilon^{3/2-2\kappa}}\leq 2M.
    \]
    Since $\tilde{u}_\epsilon^{(i)}$ is a solution with the initial condition $u^{(1)}_\epsilon(T_{**})=u^{(2)}_\epsilon(T_{**})$, $\tilde{u}^{(1)}_\epsilon$ and $\tilde{u}^{(2)}_\epsilon$ coincide on $[0,T_{**}]$ in similar way.
    Therefore, we can iterate the above arguments on $[kT_{**},(k+1)T_{**}\wedge T_*]$ for $k=1,2,\dots$ and thus $u^{(1)}_\epsilon$ and $u^{(2)}_\epsilon$ coincide on $[0,T_{*}]$.

    At last, we show (2).
    Since $T_*$ continuously depends on the initial condition $u(0)$ and the driving vector $\EX_\epsilon$, we have $T_*^{(n)} \to T_*$.
    Therefore, for every $\eta>0$, for sufficiently large $n$, we have $T^{(n)}_* > (1-\eta)T_*$.
    Thus, there exists a unique solution $u_\epsilon^{(n)}$ on $[0,(1-\eta)T_*]$ with the initial condition $u^{(n)}(0)$ and the driving vector $\EX_\epsilon^{(n)}$ from above procedure.
    It is sufficient to show that $u_\epsilon^{(n)}$ converges to $u_\epsilon$ on $[0,(1-\eta)T_*]$.
    From Proposition \ref{prop:Right-hand_side_estimates}, for every $0<\tau <(1-\eta)T_*$, we have
    \begin{align*}
        &\|u_\epsilon-u^{(n)}_\epsilon\|_{\CE_{\tau}^{-1/2-\kappa}\CC_\epsilon^{3/2-2\kappa}}\\
        &=\|\CM_\epsilon u_\epsilon - \CM_\epsilon u_\epsilon^{(n)}\|_{\CE_{\tau}^{-1/2-\kappa}\CC_\epsilon^{3/2-2\kappa}}\\
        &\leq C_4\|u(0)-u^{(n)}(0)\|_{\CC^{-1/2-\kappa}} \\
        &\quad + C_5 \tau^{\frac{\kappa}{2}} (\|\EX_\epsilon-\EX_\epsilon^{(n)}\|_{\CX_{(1-\eta)T_*}^\kappa}+\|u_\epsilon-u_\epsilon^{(n)}\|_{\CE_{\tau}^{-1/2-\kappa}\CC_\epsilon^{3/2-2\kappa}}).
    \end{align*}
    Setting $\tau=(2C_5)^{-2/\kappa}$, we have the convergence $u_\epsilon^{(n)}$ to $u_\epsilon$ on $[0,\tau]$.
    Hence, we can iterate the above arguments on $[k \tau, (k+1)\tau \wedge (1-\eta)T_*]$ for $k=1,2,\dots$ and have the convergence on $[0, (1-\eta)T_*]$.
\end{proof}

\subsection{Convergence as $\epsilon$ tends to 0}
At last, we show that the solution of the approximating equation converges the solution of the dynamical $\Phi^4_3$ equation.
\label{sec:Convergence}
\begin{thm}
    \label{thm:Convergence}
    Let $0<\kappa< 1/8$, $u(0) \in \CC^{-1/2-\kappa}$, $\EX \in \CX_1^\kappa$ and $\{\EX_\epsilon\}_{0< \epsilon \leq 1} \subset \CX_1^\kappa$.
    We set
    \[
        T_*^\epsilon:=\widetilde{T}_*(u(0), \EX_\epsilon),\quad T_* := \widetilde{T}_*(u(0), \EX).
    \]
    Let $u_\epsilon \in \CE_{T_*^\epsilon}^{-1/2-\kappa}\CC_\epsilon^{3/2-2\kappa}$ be a solution of the equation (\ref{eq:L_epsilon_u_epsilon}) with the initial condition $u(0)$ driven by $\EX_\epsilon$ and let $u \in \CE_{T_*}^{-1/2-\kappa}\CC^{3/2-2\kappa}$ be a solution of the equation (\ref{eq:L_epsilon_u_epsilon}) with $\epsilon=0$ and the initial condition $u(0)$ driven by $\EX$.
    If $\{\EX_\epsilon\}_{0 <\epsilon \leq 1}$ converges to $\EX$ in $\CX_1^{\kappa}$, then we have
    \[
        \lim_{\epsilon\downarrow 0}\|u_\epsilon - u\|_{\CE_t^{-1/2-\kappa}\CC^{3/2-2\kappa}} = 0.
    \]
    for every $0<t<T_*$.
\end{thm}

\begin{proof}
    From Propositions \ref{prop:Right-hand_side_estimates} and \ref{prop:Schauder_estimates_2}, we have
    \begin{align*}
        &\|u- u_\epsilon\|_{\CE_\tau^{-1/2-\kappa}\CC^{3/2-2\kappa}}\\
        &\leq \tau^{\frac{\kappa}{2}} \{C'_1(\epsilon^{\frac{\kappa}{4}}+\|\EX-\EX_\epsilon\|_{\CX_\tau^\kappa}) + C_2' \|u-u_\epsilon\|_{\CE_\tau^{-1/2-\kappa}\CC^{3/2-2\kappa}}\}
    \end{align*}
    for any $0<\tau<T_*$.
    Taking $\tau = (2C_2')^{-2/\kappa}$, we have
    \[
        \lim_{\epsilon \downarrow 0} \|u- u_\epsilon\|_{\CE_\tau^{-1/2-\kappa}\CC^{3/2-2\kappa}} = 0.
    \]
    We can iterate the above arguments on $[k\tau,(k+1)\tau\wedge T_*]$ for $k=1,2,\dots$ and have the convergence as $\epsilon$ tends to 0.
\end{proof}


\section{Driving vectors}
\label{sec:Driver}
This section is devoted to the proof of Theorem \ref{thm:Convergence_of_DV}.
First, in Section \ref{sec:Convergence_criteria}, we recall the basic of multiple It\^{o}-Wiener integral.
Next, in Section \ref{sec:Ito_Wiener}, we derive that each element of the stochastic driving vector in Definition \ref{defi:EX_epsilon} is represented by multiple It\^{o}-Wiener integral.
At last, in Section \ref{sec:Convergence_of_DV}, we only show that differential term $\CD_\epsilon$ converges to $\CD$.
The convergences for other elements are shown in \cite{HIN}.

\subsection*{Notation}
We write $\mathbbm{e}_k(x):=e^{2 \pi i k \cdot x}$ for $k \in \Z^3, x \in \R^3$.
We set $E:=\R \times \Z^3$.
Throughout this section, we use the following notation about $E$:
\begin{itemize}
    \item We use $m=(s,k)$, $n=(t,l)$, $\mu=(\sigma, k)$ and $\nu=(\tau,l)$ to denote a generic element in $E$.
    \item We sometimes consider more than one elements $m_1,m_2,\dots$ indexed by positive integers.
    \item For $k_j \in \Z^3$, we write $k_{-j}:=-k_j$. 
    For $m_j=(s_j,k_j)$, we write $m_{-j}:=(s_j, -k_j)$.
    For $\mu_j=(\sigma_j,k_j)$, we write $\mu_{-j}:=(-\sigma_j,-k_j)$.
    \item For $p_1,\dots,p_n \in \Z\setminus\{0\}$, we use $k_{p_1,\dots, p_n}:=(k_{p_1},\dots, k_{p_n})$ and $k_{[p_1,\dots, p_n]}:=k_{p_1}+\dots+k_{p_n}$ for simplicity. 
    The same notations are used for $s$, $t$, $l$, $m$, $n$, $\sigma$, $\tau$, $\mu$ and $\nu$.
    \item We define $|k|_*:=1+|k| =1+\sqrt{k_1^2+k_2^2+k_3^2}$ for $k=(k_1,k_2,k_3)\in \Z^3$ and $|m|_*:= 1+ |s|^{1/2} + |k| $ for $m\in E$.
    \item For simplicity, we write
    \[
        \int_{E^p} f(m_{1,\dots,p}) d m_{1,\dots, p}
        := \sum_{k_{1,\dots,p}\in \Z^3} \int_{\R^{p}} f(m_{1,\dots,p}) ds_{1,\dots,p}
    \]
    for a integrable function $f$.
    \item We use the following function space.
    \begin{align*}
        L_p^n&:= \left\{ f:E^p\to \C ; \int_{E^p} |f(m_{1,\dots,p})|^n dm_{1,\dots,p} < \infty \right\},\quad (n\geq 1) \\
        L_p^\infty&:= \left\{ f:E^p\to \C ; \sup_{m_{1,\dots,p}\in E^p} |f(m_{1,\dots,p}) | < \infty \right\}.
    \end{align*}
\end{itemize}

Let $f:E^p \to \C$ satisfy
\[
    \int_{\R^p} |f(s_1,k_1,\dots,s_p,k_p)|^2 ds_1\dots ds_p < \infty
\]
for every $k_1, \dots, k_p \in \Z^3$.
For such $f$, we can define the Fourier transform $\CF_{\mathrm{time}}f$ with respect to time parameters.
In particular, if $f$ is integrable, then $\CF_{\mathrm{time}}f$ is given by
\begin{multline*}
    [\CF_{\mathrm{time}}f](\sigma_1,k_1,\dots,\sigma_p,k_p)\\
    = \int_{\R^p} e^{-2\pi i (\sigma_1 s_1 + \dots + \sigma_p s_p)} f(s_1,k_1,\dots,s_p,k_p) ds_1\cdots ds_p.
\end{multline*}

\subsection{Convergence criteria}
\label{sec:Convergence_criteria}
For $f \in L_p^2$, we denote the $p$-th multiple It\^{o}-Wiener integral by
\[
    \CI_p(f):=\int_{E^p} f(m_{1,\dots,p}) :d\widehat{W}(m_1)\cdots d\widehat{W}(m_p): ,
\]
where $W$ is a cylindrical Brownian motion and we set
\[
    \widehat{W}(s,k):=\left<W_s, \be_{-k}\right>.
\]
We consider a random field of the form
\[
    X(t,x)=\CI_p(f_{(t,x)})
\]
for a kernel $f_{(t,\cdot)} \in C(\TT^3,L^\infty_p)$.
Even if $f_{(t,\cdot)}$ is not square integrable, we can understand $X(t)$ as an $\CS'$-valued random variable such that
\[
    \left< X(t), \phi \right> = \CI_p(\left< f_{(t,\cdot)}, \phi \right>)
\]
if $\left< f_{(t,\cdot)}, \phi\right>:= \int_{\TT^3} f_{(t,x)}\phi(x) \in L^2_p$ for every $\phi\in \CS$.

We are interested in the case that $f$ satisfies the following good conditions.

\begin{defi}[\cite{HIN}, Definition 5.2]
    We say that a family $\{f_{(t,x)}\}_{t\geq 0, x\in \TT^3}$ is \textit{good} if it has the form
    \[
        f_{(t,x)}(m_{1,\dots,p})=\be_{k_{[1,\dots,p]}}(x)H_t(m_{1,\dots,p})
    \]
    for some $H_t \in L^\infty_p$ which is in $L^2$ with respect to $(s_1,\dots,s_p)$ for each fixed $(k_1,\dots,k_p)$ and $Q_t = \CF_{\mathrm{time}}H_t$ satisfies
    \[
        Q_t(\mu_{1,\dots,p}) = e^{-2\pi i \sigma_{[1,\dots,p]}t}Q_0(\mu_{1,\dots,p}).
    \]
\end{defi}

What we want to know is the estimates of $X$ by Besov norm.
For a function $f:E^p \to \C$ and $\mu=(\sigma, k)$, we write
\[
    \int_{\mu_{[1,\dots,p]}=\mu} f(\mu_{1,\dots,p})
\]
for the integration over the ``hyperplane'' $\{\mu_{[1,\dots,p]}=\mu;\mu_{1,\dots,p}\in E^p\}$.
From the following proposition, it is enough to estimate $Q_0$ in order to estimate the Besov norm of $X$.
\begin{prop}[\cite{HIN}, Proposition 5.4]\label{prop:criteria}
    Let $\{f_{(t,x)}\}_{t\geq 0,x\in \TT^3}$ be a good kernel.
    Assume that there exist $\gamma > 1$, $\delta \geq 0$ and $C>0$ such that
    \begin{equation*}
        \int_{\mu_{[1,\dots,p]} = \mu} |Q_0(\mu_{1,\dots, p})|^2 \leq C |\mu|_*^{-2\gamma} |k|_*^{-2\delta}.
    \end{equation*}
    Then we have
    \[
        \EE[\|X\|_{C_T^\kappa\CC^{\alpha-2\kappa}}^{2p}] \lesssim C^{2p}
    \]
    for every $1 < p < \infty$, $\alpha < -\frac{5}{2} + \gamma + \delta$ and $0 \leq \kappa < \frac{\gamma - 1}{2}\wedge 1$.
\end{prop}

\subsection{It\^{o}-Wiener integral expressions of driving vectors}
\label{sec:Ito_Wiener}
We write $P_t^\epsilon:=e^{t(\Delta-\epsilon\Delta^2-1)}$.
Fix $0< \epsilon \leq 1$.
Since $\IX_\epsilon$ is a stationary solution to the equation
\[
    \CL_\epsilon \IX_\epsilon = \xi,
\]
we have the expression
\[
    \IX_\epsilon = \int_{-\infty}^{t} P_{t-s}^\epsilon \xi_\epsilon ds
    = \sum_{k\in \Z^d} \be_{k}(x) \int_\R \mathbbm{1}_{s \leq t} e^{-(t-s)\alpha_\epsilon(k)} d\widehat{W}(s,k).
\]
Hence, we have
\[
   \IX_\epsilon(t,x) = \CI_1(f^{\epsilon}_{(t,x)}),
\]
where 
\[
    f^{\epsilon}_{(t,x)}(s,k):=\be_k(x) H^{\epsilon}_t(s,k),\quad    H_t^\epsilon(s,k):= \mathbbm{1}_{s\leq t} e^{-(t-s)\alpha_\epsilon(k)},
\]
\[
    \alpha_\epsilon(k):=4\pi^2|k|^2 + \epsilon 16 \pi^4 |k|^4 +1 .
\]
Then, we can show that $f^{\epsilon}_{(t,x)}$ is a good kernel with functions $H^{\epsilon}_t$ and
\[
    Q_0^\epsilon(\sigma,k):=\frac{1}{-2\pi i \sigma + \alpha_\epsilon(k)}.
\]

We show that the driving vector
\[
    \EX_\epsilon = (\IX_\epsilon, \XX_\epsilon, \ICC_\epsilon(0), \ICCC_\epsilon, \CICCC_\epsilon, \CCICC_\epsilon, \CCICCC_\epsilon, \CD_\epsilon)
\]
in Definition \ref{defi:EX_epsilon} can be expressed in the form of It\^{o}-Wiener integral.
First, we have the following proposition about lower terms.
We define $p(\tau)$ as the number of leaves in Table \ref{tab:p_tau}.
\begin{table}[h]
    \centering
    \caption{the number of leaves in $\tau$}
    \label{tab:p_tau}
     \begin{tabular}{|c|c|c|c|c|c|c|c|c|c|}
      \hline
      $\tau$ & $\IX$& $\XX$ & $\XXX$ & $\ICC$ & $\ICCC$ & $\CICCC$ & $\CCICC$ & $\CCICCC$ & $\CD$ \\
      \hline
      $p(\tau)$ & $1$ & $2$ & $3$ & $2$ & $3$ & $4$ & $4$ & $5$ & $4$\\
      \hline
     \end{tabular}
   \end{table}
\begin{prop}[\cite{HIN} Proposition 5.5]
    Let $\tau = \IX, \XX, \XXX, \ICC, \ICCC, \nabla \ICC, \Delta \ICC$ and $p=p(\tau)$.
    Then 
    \[
        \tau_\epsilon(t,x) = \CI_p(f^{\epsilon, \tau}_{(t,x)}),
    \]
    where $f^{\epsilon, \tau}_{(t,x)}=f^{\epsilon, \tau}_{(t,x)}(m_{1,\dots,p})$ is a good kernel with functions $H^{\epsilon,\tau}_t$ and $Q^{\epsilon, \tau}_0$ defined as follows.
    \begin{enumerate}[(1)]
        \item $\{H^{\epsilon,\tau}_t\}_{t\geq 0} \subset L^2_p$ is given as follows.
        \begin{itemize}
            \item $H^{\epsilon,\IX}_t:=H^{\epsilon}_t$
            \item For $\tau = \XX, \XXX$,
                \[
                    H^{\epsilon,\tau}_t(m_{1,\dots,p}):=\prod_{j=1}^p H^\epsilon_t (m_j).
                \]
            \item For $\tau = \ICC, \ICCC$,
                \[
                    H^{\epsilon,\tau}_t(m_{1,\dots,p}):= \int_{\R} H^{\epsilon}_t(u, k_{[1,\dots,p]}) H^{\epsilon, \tau_0}_{u}(m_{1,\dots,p}) du,
                \]
                where $\tau_0=\XX,\XXX$, respectively.
            \item For $\tau = \nabla \ICC$,
                \[
                    H^{\epsilon,\nabla \ICC}_t (m_{1,2})=2\pi i k_{[1,2]}H^{\epsilon,\ICC}_t(m_{1,2}).
                \]
            \item For $\tau = \Delta \ICC$,
                \[
                    H^{\epsilon,\Delta \ICC}_t (m_{1,2})=(-4\pi^2) |k_{[1,2]}|^2 H^{\epsilon,\ICC}_t(m_{1,2}).
                \]
        \end{itemize}
        \item $Q^{\epsilon,\tau}_0 \in L^2_p$ is given as follows.
        \begin{itemize}
            \item $Q^{\epsilon, \IX}_0:=Q^{\epsilon}_0$.
            \item For $\tau = \XX, \XXX$,
                \[
                    Q^{\epsilon,\tau}_0(\mu_{1,\dots,p}):=\prod_{j=1}^{p}Q_0^{\epsilon}(\mu_j).
                \]
            \item For $\tau = \ICC, \ICCC$, 
                \[
                    Q^{\epsilon,\tau}_0(\mu_{1,\dots,p}) = Q^{\epsilon}_0(\mu_{[1,\dots,p]}) Q^{\epsilon, \tau_0}_0(\mu_{1,\dots,p}),
                \]
                where $\tau_0=\XX, \XXX$, respectively.
            \item For $\tau = \nabla \ICC$, 
                \[
                    Q^{\epsilon,\nabla \ICC}_0(\mu_{1,2})=2\pi i k_{[1,2]}Q^{\epsilon,\ICC}_0(\mu_{1,2}). 
                \]
            \item For $\tau = \Delta \ICC$, 
                \[
                    Q^{\epsilon,\Delta \ICC}_0(\mu_{1,2})=(-4\pi^2) |k_{[1,2]}|^2Q^{\epsilon,\ICC}_0(\mu_{1,2}). 
                \]
        \end{itemize}
    \end{enumerate}
\end{prop}

From this proposition, we can guess the limit $\tau$ of $\{\tau_\epsilon\}_{0<\epsilon\leq 1}$ as follows:
\begin{defi}
    Let $\tau \in \{\IX, \XX, \XXX, \ICC, \ICCC \}$ and $p=p(\tau)$.
    We define
    \[
        \tau(t,x) := \CI_p(f_{(t,x)}^{\tau}),
    \]
    where
    \[
        f_{(t,x)}^{\tau} := \lim_{\epsilon \downarrow 0}f_{(t,x)}^{\epsilon, \tau}.
    \]
    The same notations are used for $H_t^{\tau}$ and $Q_0^{\tau}$.
\end{defi}

\begin{prop}
    Let $\tau \in \{\IX, \XX, \XXX, \ICC, \ICCC \}$.
    Then, $f_{(t,x)}^{\tau}$ is a good kernel with functions $H_t^{\tau}$ and $Q_0^{\tau}$.
\end{prop}

Next, we consider the resonant products of two It\^{o}-Wiener integrals for the expression of $\tau= \CICCC, \CCICC, \CCICCC, \CD$.
We define the function $\psi_\circ : \Z^3 \times \Z^3 \to \R$ by
\begin{equation*}
    \psi_{\circ}(k,l):= \sum_{|i-j|\leq 1} \rho_i(k) \rho_j(l).
\end{equation*}

\begin{prop}[\cite{HIN}, Proposition 5.7]
    Let $f_{(t,x)}^{(1)}=f_{(t,x)}^{(1)}(m_{1,\dots,p_1})$ and $f_{(t,x)}^{(2)}=f_{(t,x)}^{(2)}(m_{1,\dots,p_2})$ be symmetric and good kernels with functions $H_t^{(1)}$, $Q_0^{(1)}$ and $H_t^{(2)}$, $Q_0^{(2)}$, respectively.
    Let $\tau_1(t,x)=\CI_{p_1}(f_{(t,x)}^{(1)})$ and $\tau_2(t,x)=\CI_{p_2}(f_{(t,x)}^{(2)})$.

    Then, we have
    \[
        \tau_1(t,x) \rs \tau_2(t,x) = \sum_{r=0}^{p_1\wedge p_2}r! \binom{p_1}{r}\binom{p_2}{r} \CI_{p_1+p_2-2r}(f_{(t,x)}^{(\tau_1,\tau_2,r)}),
    \]
    where $f_{(t,x)}^{(\tau_1,\tau_2,r)}$ are good kernels with functions $H_t^{(\tau_1,\tau_2,r)}$ and $Q_0^{(\tau_1,\tau_2,r)}$ defined as follows.
    \begin{enumerate}[(1)]
        \item For $r=0$, 
        \begin{multline*}
            H_t^{(\tau_1,\tau_2,0)}(m_{1,\dots,p_1+p_2})\\= \psi_\circ(k_{[1,\dots,p_1]},k_{[p_1+1,\dots, p_1+p_2]})H^{(1)}_t(m_{1,\dots,p_1})H^{(2)}_t(m_{p_1+1,\dots,p_1+p_2}).
        \end{multline*}
        For $r>0$,
        \begin{multline*}
            H_t^{(\tau_1,\tau_2,r)}(m_{1,\dots,p_1-r},m_{p_1+1,\dots, p_1+p_2-r})\\
            = \int_{E^{r}} H_t^{(\tau_1,\tau_2,0)}(m_{1,\dots,p_1-r},n_{1,\dots,r},m_{p_1,\dots,p_1+p_2-r},n_{-1,\dots,-r}) dn_{1,\dots,r}.
        \end{multline*}
        \item For $r=0$, 
        \begin{multline*}
            Q_0^{(\tau_1,\tau_2,0)}(\mu_{1,\dots,p_1+p_2})\\
            = \psi_{\circ}(k_{[1,\dots,p_1]},k_{[p_1+1,\dots,p_1+p_2]})Q_0^{(1)}(\mu_{1,\dots,p_1})Q_0^{(2)}(\mu_{p_1+1,\dots,p_1+p_2}).
        \end{multline*}
        For $r>0$,
        \begin{multline*}
            Q_0^{(\tau_1,\tau_2,r)}(\mu_{1,\dots,p_1-r},\mu_{p_1+1,\dots,p_1+p_2-r})\\
            = \int_{E^{r}} Q_0^{(\tau_1,\tau_2,0)}(\mu_{1,\dots,p_1-r},\nu_{1,\dots,r},\mu_{p_1,\dots,p_1+p_2-r},-\nu_{1,\dots,r}) d\nu_{1,\dots,r}.
        \end{multline*}
    \end{enumerate}
\end{prop}
Let $\tau = \CICCC, \CCICC, \CCICCC, \CD$.
From this proposition, $\tau_\epsilon$ can be expressed as a sum of It\^{o}-Wiener integrals.
Recall the renormalization (v) in Definition \ref{defi:EX_epsilon}.
Here, renormalization constant $b_\epsilon$ can be expressed $b_\epsilon = 2f^{(\ICC_\epsilon,\XX_\epsilon,2)}$.
\begin{itemize}
    \item For $\CICCC_\epsilon$, we have
    \begin{align}
        \CICCC_\epsilon(t,x)=\CI_4(f_{(t,x)}^{\epsilon,\CICCC}) + 3\CI_2(f_{(t,x)}^{\epsilon,\cICCc}),
    \end{align}
    where
    \[
        f^{\epsilon,\CICCC}:=f^{(\ICCC_\epsilon,\IX_\epsilon,0)},\quad f^{\epsilon,\cICCc}:=f^{(\ICCC_\epsilon,\IX_\epsilon,1)}.
    \]
    \item For $\CCICC_\epsilon$, we have
    \begin{align}
        \CCICC_\epsilon(t,x)=\CI_4(f_{(t,x)}^{\epsilon,\CCICC}) + 4\CI_2(f_{(t,x)}^{\epsilon,\CcICc}),
    \end{align}
    where
    \[
        f^{\epsilon,\CCICC}:=f^{(\ICC_\epsilon,\XX_\epsilon,0)},\quad f^{\epsilon,\CcICc}:=f^{(\ICC_\epsilon,\XX_\epsilon,1)}.
    \]
    \item For $\CCICCC_\epsilon$, we have
    \begin{align}
        \CCICCC_\epsilon(t,x)=\CI_5(f_{(t,x)}^{\epsilon,\CCICCC}) + 6\CI_3(f_{(t,x)}^{\epsilon,\CcICCc}) + 6\CI_1(\CR f_{(t,x)}^{\epsilon,\ccICcc}),
    \end{align}
    where
    \[
        f^{\epsilon,\CCICCC}:=f^{(\ICCC_\epsilon,\XX_\epsilon,0)},\,\, f^{\epsilon,\CcICCc}:=f^{(\ICCC_\epsilon,\XX_\epsilon,1)}, \,\,
        \CR f^{\epsilon,\ccICcc} := f^{(\ICCC_\epsilon,\XX_\epsilon,2)} - b_\epsilon f^{\epsilon,\IX},
    \]
    \item For $\CD_\epsilon$, we have
    \begin{align}
        \CD_\epsilon(t,x)=\CI_4(f_{(t,x)}^{\epsilon,\CD}) + 4\CI_2(f_{(t,x)}^{\epsilon,\CD,1}) + c_\epsilon, \label{eq:defi_of_CD_epsilon}
    \end{align}
    where
    \[
        f^{\epsilon,\CD}:=f^{(\nabla \ICC_\epsilon,\nabla \ICC_\epsilon,0)}+\epsilon f^{(\Delta \ICC_\epsilon,\Delta \ICC_\epsilon,0)},\,\, f^{\epsilon,\CD,1}:=f^{(\nabla \ICC_\epsilon,\nabla \ICC_\epsilon,1)}+\epsilon f^{(\Delta \ICC_\epsilon,\Delta \ICC_\epsilon,1)},
    \]
    \[
        c_\epsilon := 2f^{(\nabla \ICC_\epsilon,\nabla \ICC_\epsilon,2)}+2\epsilon f^{(\Delta \ICC_\epsilon,\Delta \ICC_\epsilon,2)} - b_\epsilon.
    \]
\end{itemize}

From these, we want to guess the limit $\tau$ of $\{\tau_\epsilon\}_{0<\epsilon\leq 1}$.
By using similar way in \cite{HIN} Proposition 5.18, we can show that there exists a kernel $\CR f^{\ccICcc}$ such that
\[
    \CR f^{\epsilon,\ccICcc} \to \CR f^{\ccICcc}.
\]
And, in Section \ref{sec:Differential_terms}, we will show that there exists a real number $c$ such that
\[
    c_\epsilon \to c.
\]
\begin{defi}
    $\tau = \CICCC,\CCICC,\CCICCC,\CD$ is defined as follows.
    \begin{align}
        \CICCC(t,x)&:=\CI_4(f_{(t,x)}^{\CICCC}) + 3\CI_2(f_{(t,x)}^{\cICCc}), \\
        \CCICC(t,x) &:= \CI_4(f_{(t,x)}^{\CCICC}) + 4\CI_2(f_{(t,x)}^{\CcICc}), \\
        \CCICCC(t,x)&:=\CI_5(f_{(t,x)}^{\CCICCC}) + 6\CI_3(f_{(t,x)}^{\CcICCc}) + 6\CI_1(\CR f_{(t,x)}^{\ccICcc}), \\
        \CD(t,x)&:=\CI_4(f_{(t,x)}^{\CD}) + 4\CI_2(f_{(t,x)}^{\CD,1}) + c. \label{eq:defi_of_CD}
    \end{align}
    Here, we set
    \[
        f_{(t,x)}^{\tau} := \lim_{\epsilon\downarrow 0} f_{(t,x)}^{\epsilon, \tau}
    \]
    for $\tau = \CICCC, \cICCc, \CCICC, \CcICc, \CCICCC, \CcICCc$ and 
    \[
        f_{(t,x)}^{\CD}:=\lim_{\epsilon\downarrow 0}f_{(t,x)}^{(\nabla \ICC_\epsilon,\nabla \ICC_\epsilon,0)}, \quad f_{(t,x)}^{\CD,1}:=\lim_{\epsilon\downarrow 0}f_{(t,x)}^{(\nabla \ICC_\epsilon,\nabla \ICC_\epsilon,1)}.
    \]
    The same notations are used for $H_t^{\tau}$ and $Q_0^{\tau}$.
\end{defi}
\begin{prop}
    Let $\tau \in \{ \CICCC,\CCICC,\CCICCC,\CD, \cICCc, \CcICc, \CcICCc\}$.
    Then, $f_{(t,x)}^{\tau}$ is a good kernel with functions $H_t^{\tau}$ and $Q_0^{\tau}$.
    $\CR f_{(t,x)}^{\ccICcc}$ and $f_{(t,x)}^{\CD,1}$ are also  good kernels with functions $\CR H_t^{\ccICcc}$, $\CR Q_0^{\ccICcc}$ and $H_t^{\CD,1}$, $Q_0^{\CD,1}$, respectively.
\end{prop}

We set
\[
   \EX:=(\IX, \XX, \ICC, \ICCC, \CICCC, \CCICC, \CCICCC, \CD). 
\]
The following is main theorem in this section:
\begin{thm}
    \label{thm:Convergence_of_DV}
    Let $\kappa > 0$ and $T>0$.
    For any $p\geq 1$, we have
    \[
        \lim_{\epsilon\downarrow 0}\EE[\|\EX-\EX_\epsilon\|_{\CX_T^\kappa}^p]=0.
    \]
\end{thm}
We only show that $\CD_\epsilon$ converges to $\CD$ in next section.
For the convergence about other elements, see \cite{HIN} Section 5.3.2 and 5.3.3.

\subsection{Proof of convergence of driving vectors}
\label{sec:Convergence_of_DV}
The following two estimates are often used.
\begin{lem}[\cite{HIN} Lemma 5.11]\label{lem:useful_estimates_1}
    If $0 < \alpha,\beta <5$ and $\alpha + \beta > 5$, we have
    \[
        \int_E |\mu|_*^{-\alpha} |\nu - \mu|_*^{-\beta} d \mu \lesssim |\nu|_*^{-\alpha-\beta + 5}.
    \]
\end{lem}

\begin{lem}[\cite{HIN} Lemma 5.12] \label{lem:useful_estimates_2}
    $\psi_{\circ}$ is bounded and supported in the set $\{ (k,l) \in \Z^3 \times \Z^3 ; C^{-1} |l|_* \leq |k|_* \leq C|l|_* \}$ for some $C>0$.
    Moreover, we have
    \[
        |\psi_{\circ}(k,l)| \lesssim |k+l|_*^{-\theta}|l|_*^{\theta}
    \]
    for every $\theta>0$.
\end{lem}

\label{sec:Differential_terms}
In this section, we will show that $\CD_\epsilon $ converges to $\CD$ as $\epsilon$ tends to 0.
\begin{thm}
    \label{thm:Convergence_of_D}
    For every small $\kappa > 0$ and $p\geq 1$, we have
    \[
        \lim_{\epsilon\downarrow 0}\EE[\|\CD_\epsilon - \CD\|_{C_T\CC^{-\kappa}}^p] = 0.
    \]
\end{thm}
We prove Theorem \ref{thm:Convergence_of_D} as a consequence of Propositions \ref{prop:Convergence_of_D_1} to \ref{prop:Convergence_of_D_3}.
Recall the It\^{o}-Wiener integral expression (\ref{eq:defi_of_CD_epsilon}) and (\ref{eq:defi_of_CD}) of $\CD_\epsilon$ and $\CD$, respectively.
We show the convergence of the fourth-order It\^{o}-Wiener integral term in Proposition \ref{prop:Convergence_of_D_1}, the second-order It\^{o}-Wiener integral term in Proposition \ref{prop:Convergence_of_D_2} and the constant term in Proposition \ref{prop:Convergence_of_D_3}, respectively.
Before the proof, we introduce the equality and inequalities used frequently.
\begin{align}
    (Q_0 - Q_0^\epsilon)(\mu) &= \epsilon 16 \pi^4|k|^4 Q_0(\mu)Q_0^\epsilon(\mu), \label{eq:Q0_difference} \\
    |\mu|^2_*+\epsilon |k|^4 &\lesssim |Q_0^\epsilon(\mu) |^{-1},\label{eq:Q0_epsilon}\\
    |Q_0^\epsilon(\mu)|&\lesssim|Q_0(\mu)| \lesssim |\mu|_*^{-2}. \label{eq:Q0}
\end{align}
We use the equality (\ref{eq:Q0_difference}) in order to estimate $|(Q_0 - Q_0^\epsilon)(\mu)|$ with $\epsilon$.
We use the inequality (\ref{eq:Q0_epsilon}) in order to estimate $\epsilon |k|^4$ uniformly on $0<\epsilon \leq 1$.
We use the inequality (\ref{eq:Q0}) in order to estimate $|Q_0^\epsilon(\mu)|$ uniformly on $0<\epsilon\leq 1$.

\begin{prop}
    \label{prop:Convergence_of_D_1}
    For every small $\lambda, \kappa >0$, we have
    \begin{align*}
        \int_{\mu_{[1,\dots,4]}=\mu} |Q_0^{\CD}(\mu_{1,\dots,4}) |^2
        &\lesssim |\mu|_{*}^{-5+\kappa} |k|_{*}^{-\kappa},\\
        \int_{\mu_{[1,\dots,4]}=\mu} |Q_0^{\epsilon, \CD} (\mu_{1,\dots,4}) |^2
        &\lesssim |\mu|_{*}^{-5+\kappa} |k|_{*}^{-\kappa},\\
        \int_{\mu_{[1,\dots,4]}=\mu} |(Q_0^{\CD} -Q_0^{\epsilon, \CD}) (\mu_{1,\dots,4}) |^2
        &\lesssim \epsilon^\lambda|\mu|_{*}^{-5+12\lambda}.
    \end{align*}
\end{prop}
\begin{proof}
    Recall the definitions of $Q_0^{\CD}$ and $Q_0^{\epsilon,\CD}$
    \begin{align}
        &Q_0^{\CD}(\mu_{1,\dots,4})\notag\\
        &= \psi_{\circ}(k_{[1,2]},k_{[3,4]})(-4\pi^2k_{[1,2]}\cdot k_{[3,4]}) Q_0^{\ICC}(\mu_{1,2})Q_0^{\ICC}(\mu_{3,4}), \label{eq:Q0_D} \\
        &Q_0^{\epsilon,\CD}(\mu_{1,\dots,4}) \notag\\
        &= \psi_{\circ}(k_{[1,2]},k_{[3,4]})  \notag \\
        & \quad \times\{ -4\pi^2k_{[1,2]}\cdot k_{[3,4]} + \epsilon 16 \pi^4 |k_{[1,2]}|^2|k_{[3,4]}|^2\} Q_0^{\epsilon,\ICC}(\mu_{1,2})Q_0^{\epsilon,\ICC}(\mu_{3,4}), \label{eq:Q0_epsilon_D}
    \end{align}
    where
    \[
        Q_0^{\ICC}(\mu_{1,2}) = Q_0(\mu_{[1,2]})Q_0(\mu_1)Q_0(\mu_2), \quad    
        Q_0^{\epsilon, \ICC}(\mu_{1,2}) = Q_0^{\epsilon}(\mu_{[1,2]})Q_0^{\epsilon}(\mu_1)Q_0^{\epsilon}(\mu_2).
    \]

    For $Q_0^{\epsilon,\CD}$, by using (\ref{eq:Q0_epsilon}) and (\ref{eq:Q0}), we have
    \begin{align}
        &|Q_0^{\epsilon,\CD}(\mu_{1,\dots,4})| \notag \\
        &\lesssim |\psi_{\circ}(k_{[1,2]},k_{[3,4]}) |\notag\\
        & \quad \times( |k_{[1,2]}|+ \epsilon|k_{[1,2]}|^2)|( |k_{[3,4]}|+ \epsilon|k_{[3,4]}|^2) |Q_0^{\epsilon,\ICC}(\mu_{1,2})||Q_0^{\epsilon,\ICC}(\mu_{3,4})| \notag\\
        &\lesssim |\psi_{\circ}(k_{[1,2]},k_{[3,4]}) ||Q_0^\epsilon(\mu_{[1,2]})|^{1/2}|Q_0^\epsilon(\mu_{[3,4]})|^{1/2} \prod_{j=1}^4|\mu_j|_*^{-2} \notag\\
        &\lesssim |\psi_{\circ}(k_{[1,2]},k_{[3,4]})| |\mu_{[1,2]}|_*^{-1}|\mu_{[3,4]}|_*^{-1}  \prod_{j=1}^4|\mu_j|_*^{-2}. \label{eq:Q0_epsilon_D_ineq}
    \end{align}
    By using Lemmas \ref{lem:useful_estimates_1} and \ref{lem:useful_estimates_2}, we have
    \begin{align*}
        &\int_{\mu_{[1,\dots,4]=\mu}} |Q_0^{\epsilon,\CD}(\mu_{1,\dots,4})|^2 \\
        & \lesssim \int_{\mu'_{[1,2]}=\mu} |\psi_{\circ}(k'_1,k'_2)|^2 |\mu'_1|_*^{-2}|\mu'_2|_*^{-2} \\
        &\qquad\qquad \times \left( \int_{\mu_{[1,2]}=\mu'_1} \prod_{j=1}^2|\mu_j|_*^{-4} \right)\left( \int_{\mu_{[3,4]}=\mu'_2}\prod_{j=3}^4|\mu_j|_*^{-4} \right) \\
        & \lesssim\int_{\mu'_{[1,2]}=\mu} |\psi_{\circ}(k'_1,k'_2)|^2 |\mu'_1|_*^{-5}|\mu'_2|_*^{-5}  \\
        &\lesssim |k|_*^{-\kappa} \int_{\mu'_{[1,2]}=\mu}  |\mu'_1|_*^{-5+\kappa/2}|\mu'_2|_*^{-5+\kappa/2} \\
        &\lesssim |\mu|_*^{-5+\kappa} |k|_*^{-\kappa}.
    \end{align*}
    For $Q_0^{\epsilon,\CD}$, from (\ref{eq:Q0}), we have
    \begin{align*}
        |Q_0^\CD(\mu_{1,\dots,4})|
        &\lesssim |\psi_{\circ}(k_{[1,2]},k_{[3,4]}) | |k_{[1,2]}||k_{[3,4]}| |Q_0^{\ICC}(\mu_{1,2})||Q_0^{\ICC}(\mu_{3,4})| \\
        & \lesssim |\psi_{\circ}(k_{[1,2]},k_{[3,4]})| |\mu_{[1,2]}|_*^{-1}|\mu_{[3,4]}|_*^{-1}  \prod_{j=1}^4|\mu_j|_*^{-2}.
    \end{align*}
    Thus, we obtain the estimates of $\int_{\mu_{[1,\dots,4]}=\mu} |Q_0^{\CD} (\mu_{1,\dots,4}) |^2$ in the same way.

    For the difference between $Q_0^\CD$ with $Q_0^{\epsilon,\CD}$, we have
    \begin{align*}
        &|(Q_0^\CD - Q_0^{\epsilon,\CD})(\mu_{1,\dots,4}) | \\
        &= | \psi_{\circ}(k_{[1,2],}k_{[3,4]})| \\
        &\quad \times \Big|(-4\pi^2k_{[1,2]}\cdot k_{[3,4]}) Q_0^{\ICC}(\mu_{1,2})Q_0^{\ICC}(\mu_{3,4}) \\
        &\quad \quad - (-4\pi^2k_{[1,2]}\cdot k_{[3,4]} + \epsilon 16\pi^4|k_{[1,2]}|^2|k_{[3,4]}|^2) Q_0^{\epsilon,\ICC}(\mu_{1,2})Q_0^{\epsilon,\ICC}(\mu_{3,4})\Big| \\
        &\lesssim  | \psi_{\circ}(k_{[1,2]},k_{[3,4]})| \\
        &\quad \times\Big\{|k_{[1,2]}| |k_{[3,4]}| \Big| Q_0^{\ICC}(\mu_{1,2})Q_0^{\ICC}(\mu_{3,4}) - Q_0^{\epsilon, \ICC}(\mu_{1,2})Q_0^{\epsilon,\ICC}(\mu_{3,4}) \Big| \\
        &\quad\qquad + \epsilon |k_{[1,2]}|^2|k_{[3,4]}|^2 |Q_0^{\epsilon,\ICC}(\mu_{1,2})Q_0^{\epsilon,\ICC}(\mu_{3,4})| \Big\} \\
        &=: (A) + (B).
    \end{align*}
    First, we consider $(A)$.
    From (\ref{eq:Q0_difference}) and (\ref{eq:Q0}), we have
    \begin{align*}
        &|(Q_0-Q_0^\epsilon)(\mu)| \leq 2|Q_0(\mu)| \lesssim |\mu|_*^{-2},\\
        &|(Q_0 - Q_0^\epsilon)(\mu)| \lesssim \epsilon |k|^4|Q_0||Q_0^\epsilon| \lesssim \epsilon |k|^4|Q_0|^2 \lesssim \epsilon.
    \end{align*}
    By combining these inequalities, we have
    \[
        |(Q_0 - Q_0^\epsilon)(\mu)| = |(Q_0 - Q_0^\epsilon)(\mu)|^{\lambda/2} |(Q_0 - Q_0^\epsilon)(\mu)|^{1-\lambda/2} \lesssim \epsilon^{\lambda/2} |\mu|_*^{-2+\lambda}
    \]
    for every small $\lambda > 0$.
    From this, it follows
    \begin{multline}
        | Q_0^{\ICC}(\mu_{1,2})Q_0^{\ICC}(\mu_{3,4}) - Q_0^{\epsilon, \ICC}(\mu_{1,2})Q_0^{\epsilon,\ICC}(\mu_{3,4}) | \\
        \lesssim \epsilon^{\lambda/2} |\mu_{[1,2]}|_*^{-2+\lambda}|\mu_{[3,4]}|_*^{-2+\lambda} \prod_{j=1}^4 |\mu_j|_*^{-2+\lambda}. \label{eq:Q0_D_1}
    \end{multline}
    Hence, we have
    \begin{align*}
        (A)
        &\lesssim \epsilon^{\lambda/2} |\mu_{[1,2]}|_{*}^{-1+\lambda}|\mu_{[3,4]}|_{*}^{-1+\lambda} \prod_{j=1}^4 |\mu_j|_{*}^{-2+\lambda}.
    \end{align*}
    Next, we consider $(B)$.
    From (\ref{eq:Q0_epsilon}) and (\ref{eq:Q0}), we have
    \begin{align*}
        (B)
        &\lesssim  \epsilon|k_{[1,2]}|^{2}|k_{[3,4]}|^{2}|Q_0^{\epsilon}(\mu_{[1,2]})||Q_0^{\epsilon}(\mu_{[3,4]})| \prod_{j=1}^{4} |Q_0^{\epsilon}(\mu_j)|\\
        &\lesssim \epsilon^{\lambda/2}|k_{[1,2]}|^{\lambda}|k_{[3,4]}|^{\lambda}(\epsilon |k_{[1,2]}|^4)^{1/2-\lambda/4} (\epsilon |k_{[3,4]}|^4)^{1/2-\lambda/4} \\
        & \quad \times |Q_0^{\epsilon}(\mu_{[1,2]})||Q_0^{\epsilon}(\mu_{[3,4]})| \prod_{j=1}^{4} |\mu_j|_*^{-2} \\
        &\lesssim  \epsilon^{\lambda/2}|k_{[1,2]}|^{\lambda}|k_{[3,4]}|^{\lambda}  |Q_0^{\epsilon}(\mu_{[1,2]})|^{1/2+\lambda/4}|Q_0^{\epsilon}(\mu_{[3,4]})|^{1/2+\lambda/4} \prod_{j=1}^{4} |\mu_j|_*^{-2} \\
        &\lesssim  \epsilon^{\lambda/2} |\mu_{[1,2]}|_*^{-1+\lambda/2}|\mu_{[3,4]}|_*^{-1+\lambda/2} \prod_{j=1}^{4} |\mu_j|_*^{-2} \\
        &\lesssim \epsilon^{\lambda/2} |\mu_{[1,2]}|_{*}^{-1+\lambda}|\mu_{[3,4]}|_{*}^{-1+\lambda} \prod_{j=1}^4 |\mu_j|_{*}^{-2+\lambda}.
    \end{align*}
    By using Lemma \ref{lem:useful_estimates_1} and \ref{lem:useful_estimates_2}, it holds
    \begin{align*}
        &\int_{\mu_{[1,\dots,4]}=\mu}|(Q_0^\CD - Q_0^{\epsilon,\CD})(\mu_{1,\dots,4}) |^2 \\
        &\lesssim \epsilon^{\lambda}  \int_{\mu'_{[1,2]}=\mu}|\mu'_{1}|_*^{-2+2\lambda}|\mu'_{2}|_*^{-2+2\lambda}\\
        &\quad \quad \times \left(\int_{\mu_{[1,2]}=\mu'_1} |\mu_1|_*^{-4+2\lambda}|\mu_2|_*^{-4+2\lambda}\right) \left(\int_{\mu_{[3,4]}=\mu'_2} |\mu_3|_*^{-4+2\lambda}|\mu_4|_*^{-4+2\lambda} \right) \\
        &\lesssim \epsilon^{\lambda} \int_{\mu'_{[1,2]}=\mu}|\mu'_{1}|_*^{-5+6\lambda}|\mu'_{2}|_*^{-5+6\lambda}\\
        &\lesssim  \epsilon^{\lambda}|\mu|_*^{-5+12\lambda}.
    \end{align*}
\end{proof}

\begin{prop}
    \label{prop:Convergence_of_D_2}
    For every small $\lambda, \kappa >0$, we have
    \begin{align*}
        \int_{\mu_{[1,3]}=\mu}|Q_0^{\CD,1}(\mu_{1,3})|^2
        &\lesssim |\mu|_*^{-4+\kappa}|k|_*^{-1-\kappa},\\
        \int_{\mu_{[1,3]}=\mu}|Q_0^{\epsilon,\CD,1}(\mu_{1,3})|^2
        &\lesssim |\mu|_*^{-4+\kappa}|k|_*^{-1-\kappa},\\
        \int_{\mu_{[1,3]}=\mu} |(Q_0^{\CD,1} - Q_0^{\epsilon, \CD,1}) (\mu_{1,3}) |^2
        &\lesssim \epsilon^\lambda |\mu|_{*}^{-4+14\lambda+\kappa} |k|_{*}^{-1-\kappa}.
    \end{align*}
\end{prop}
\begin{proof}
    Recall the definitions of $Q_0^{\CD,1}$ and $Q_0^{\epsilon,\CD,1}$
    \begin{align*}
        Q_0^{\CD,1}(\mu_{1,3})
        &= \int_E Q_0^{\CD}(\mu_{1,2,3,-2}) d\mu_2, \\
        Q_0^{\epsilon,\CD,1}(\mu_{1,3})
        &= \int_E Q_0^{\epsilon,\CD}(\mu_{1,2,3,-2}) d\mu_2,
    \end{align*}
    where the definitions of $Q_0^{\CD}$ and $Q_0^{\epsilon,\CD}$ are (\ref{eq:Q0_D}) and (\ref{eq:Q0_epsilon_D}), respectively.

    For $Q_0^{\epsilon,\CD,1}$, we have
    \begin{align*}
        &|Q_0^{\epsilon,\CD,1}(\mu_{1,3})|\\
        &\leq \int_{E} |Q_0^{\epsilon,\CD}(\mu_{1,2,3,-2})| d\mu_{2} \\
        &\lesssim \int_{E} |\psi_{\circ}(k_{[1,2]},k_{[3,-2]})| |\mu_{[1,2]}|_*^{-1}|\mu_{[3,-2]}|_*^{-1}  |\mu_1|_*^{-2}|\mu_2|_*^{-4}|\mu_3|_*^{-2} d\mu_{2} \\
        &\lesssim |\mu_1|_*^{-2}|\mu_3|_*^{-2}\int_{E}  |\psi_{\circ}(k_{[1,2]},k_{[3,-2]})|( |\mu_{[1,2]}|_*^{-2}+|\mu_{[3,-2]}|_*^{-2}) |\mu_2|_*^{-4} d\mu_{2} \\
        & \lesssim |\mu_1|_*^{-2}|\mu_3|_*^{-2} |k_{[1,3]}|_*^{-1/2-\kappa/2}  \int_{E}( |\mu_{[1,2]}|_*^{-3/2+\kappa/2}+|\mu_{[3,-2]}|_*^{-3/2+\kappa/2})|\mu_2|_*^{-4} d\mu_{2} \\
        &\lesssim |\mu_1|_*^{-2}|\mu_3|_*^{-2} |k_{[1,3]}|_*^{-1/2-\kappa/2} (|\mu_1|_*^{-1/2+\kappa/2}+|\mu_3|_*^{-1/2+\kappa/2}),
    \end{align*}
    where we use the inequality (\ref{eq:Q0_epsilon_D_ineq}) in the second inequality, Young inequality in the third inequality, Lemma \ref{lem:useful_estimates_2} in fourth inequality and Lemma \ref{lem:useful_estimates_1} in last inequality.
    Hence, from Lemma \ref{lem:useful_estimates_1}, we have
    \begin{align*}
        \int_{\mu_{[1,3]}=\mu}|Q_0^{\epsilon,\CD,1}(\mu_{1,3})|^2
        \lesssim |\mu|_*^{-4+\kappa}|k|_*^{-1-\kappa}.
    \end{align*}
    We can show the estimates of $Q_0^{\CD,1}$ in same way.

    For the difference between $Q_0^{\CD,1}$ with $Q_0^{\epsilon,\CD,1}$, we have
    \begin{align*}
        &|(Q_0^{\CD,1} - Q_0^{\epsilon,\CD,1})(\mu_{1,3}) | \\
        &\lesssim \int_E| \psi_{\circ}(k_{[1,2]},k_{[3,-2]})| \\
        &\qquad\times\Big\{|k_{[1,2]}| |k_{[3,-2]}| \Big| Q_0^{\ICC}(\mu_{1,2})Q_0^{\ICC}(\mu_{3,-2}) - Q_0^{\epsilon, \ICC}(\mu_{1,2})Q_0^{\epsilon,\ICC}(\mu_{3,-2}) \Big| \\
        &\qquad \qquad + \epsilon |k_{[1,2]}|^2|k_{[3,-2]}|^2 |Q_0^{\epsilon,\ICC}(\mu_{1,2})Q_0^{\epsilon,\ICC}(\mu_{3,-2})| \Big\}d\mu_2\\
        &=: (A)+(B).
    \end{align*}

    First, we consider $(A)$.
    From Young's inequality and (\ref{eq:Q0_D_1}), we have
    \begin{align*}
        (A)
        &\lesssim \int_E |\psi_{\circ}(k_{[1,2]},k_{[3,-2]}) | |k_{[1,2]}||k_{[3,-2]}|\\
        &\qquad \times |Q^{\ICC}_0(\mu_{1,2})Q^{\ICC}_0(\mu_{3,-2}) - Q^{\epsilon, \ICC}_0(\mu_{1,2})Q^{\epsilon,\ICC}_0(\mu_{3,-2}) | d\mu_2\\
        &\lesssim \epsilon^{\lambda/2} \int_E |\psi_{\circ}(k_{[1,2]},k_{[3,-2]}) | (|k_{[1,2]}|^2+|k_{[3,-2]} |^2)\\
        &\qquad\qquad \times |\mu_{[1,2]} |_{*}^{-2+\lambda}|\mu_{[3,-2]} |_{*}^{-2+\lambda} |\mu_1 |_{*}^{-2+\lambda}|\mu_2 |_{*}^{-4+2\lambda}|\mu_3 |_{*}^{-2+\lambda}d\mu_2 \\
        &=: (A_1) + (A_2).
    \end{align*}
    From Lemmas \ref{lem:useful_estimates_1} and \ref{lem:useful_estimates_2}, we have
    \begin{align*}
        (A_1)
        &\lesssim \epsilon^{\lambda/2} \int_E |\psi_{\circ}(k_{[1,2]},k_{[3,-2]}) | \\
        &\qquad\qquad \times |\mu_{[1,2]} |_{*}^{\lambda}|\mu_{[3,-2]} |_{*}^{-2+\lambda} |\mu_1 |_{*}^{-2+\lambda}|\mu_2 |_{*}^{-4+2\lambda}|\mu_3 |_{*}^{-2+\lambda}d\mu_2 \\
        & \lesssim \epsilon^{\lambda/2} \int_E |\psi_{\circ}(k_{[1,2]},k_{[3,-2]}) | \\
        &\qquad\qquad \times|\mu_{[3,-2]} |_{*}^{-2+\lambda} |\mu_1 |_{*}^{-2+2\lambda}|\mu_2 |_{*}^{-4+3\lambda}|\mu_3 |_{*}^{-2+\lambda}d\mu_2 \\
        &\lesssim \epsilon^{\lambda/2}|\mu_1 |_{*}^{-2+2\lambda} |\mu_3 |_{*}^{-2+\lambda}  |k_{[1,3]}|_*^{-1/2-\kappa/2}\\
        &\quad \times \int_E|\mu_{[3,-2]} |_{*}^{-3/2+\lambda+\kappa/2} |\mu_2 |_{*}^{-4+3\lambda}d\mu_2 \\
        &\lesssim \epsilon^{\lambda/2}|\mu_1 |_{*}^{-2+2\lambda} |\mu_3 |_{*}^{-5/2+5\lambda+\kappa/2}  |k_{[1,3]}|_*^{-1/2-\kappa/2},
    \end{align*}
    where we use $|\mu_{[1,2]}|_* \lesssim |\mu_1|_*+|\mu_2|_*$ in the second inequality.
    Similarly, we have
    \begin{align*}
        (A_2)
        \lesssim \epsilon^{\lambda/2} |\mu_1 |_{*}^{-5/2+5\lambda+\kappa/2} |\mu_3 |_{*}^{-2+2\lambda}  |k_{[1,3]}|_*^{-1/2-\kappa/2}.
    \end{align*}

    Next, we consider $(B)$.
    From Young inequality, it holds
    \begin{align*}
        (B)
        &\lesssim \epsilon \int_E |\psi_{\circ}(k_{[1,2]},k_{[3,-2]})| (|k_{[1,2]}|^4 + |k_{[3,-2]}|^4)\\
        &\qquad \qquad \times |Q_0^\epsilon(\mu_{[1,2]})| |Q_0^\epsilon(\mu_{[3,-2]})| |\mu_1|_*^{-2}|\mu_2|_*^{-4}|\mu_3|_*^{-2} d\mu_2 \\
        &=: (B_1) + (B_2).
    \end{align*}
    From (\ref{eq:Q0_epsilon}), Lemma \ref{lem:useful_estimates_1} and Lemma \ref{lem:useful_estimates_2}, we have
    \begin{align*}
        (B_1)
        &= \epsilon^{\lambda/2} \int_E |\psi_{\circ}(k_{[1,2]},k_{[3,-2]})||k_{[1,2]}|^{2\lambda} (\epsilon |k_{[1,2]}|^4)^{1-\lambda/2} \\
        &\qquad \qquad \times |Q_0^\epsilon(\mu_{[1,2]})| |Q_0^\epsilon(\mu_{[3,-2]})| |\mu_1|_*^{-2}|\mu_2|_*^{-4}|\mu_3|_*^{-2} d\mu_2\\
        &\lesssim \epsilon^{\lambda/2} \int_E |k_{[1,3]}|_*^{-1/2-\kappa/2} |k_{[1,2]}|^{2\lambda}\\
        &\qquad \qquad \times |Q_0^\epsilon(\mu_{[1,2]})|^{\lambda/2} |\mu_{[3,-2]}|_*^{-3/2+\kappa/2} |\mu_1|_*^{-2}|\mu_2|_*^{-4}|\mu_3|_*^{-2} d\mu_2\\
        &\lesssim\epsilon^{\lambda/2} |\mu_1|_*^{-2} |\mu_3|_*^{-2}|k_{[1,3]}|_*^{-1/2-\kappa/2}  \int_E  |\mu_{[1,2]}|_*^{\lambda} |\mu_{[3,-2]}|_*^{-3/2+\kappa/2}|\mu_2|_*^{-4}d\mu_2 \\
        &\lesssim\epsilon^{\lambda/2} |\mu_1|_*^{-2} |\mu_3|_*^{-2}|k_{[1,3]}|_*^{-1/2-\kappa/2} \\
        &\quad \times \int_E (|\mu_{1}|_*^{\lambda}+|\mu_{2}|_*^{\lambda}) |\mu_{[3,-2]}|_*^{-3/2+\kappa/2}|\mu_2|_*^{-4}d\mu_2 \\
        &\lesssim\epsilon^{\lambda/2} |\mu_1|_*^{-2+\lambda} |\mu_3|_*^{-2}|k_{[1,3]}|_*^{-1/2-\kappa/2}  \int_E  |\mu_{[3,-2]}|_*^{-3/2+\kappa/2}|\mu_2|_*^{-4+\lambda}d\mu_2 \\
        &\lesssim\epsilon^{\lambda/2}  |\mu_1|_*^{-2+\lambda} |\mu_3|_*^{-5/2+\lambda+\kappa/2}|k_{[1,3]}|_*^{-1/2-\kappa/2}\\
        &\lesssim \epsilon^{\lambda/2}  |\mu_1 |_{*}^{-2+2\lambda} |\mu_3 |_{*}^{-5/2+5\lambda+\kappa/2} |k_{[1,3]}|_*^{-1/2-\kappa/2}.
    \end{align*}
    In the same way, we have
    \begin{align*}
        (B_2)
        &\lesssim\epsilon^{\lambda/2} |\mu_1|_*^{-5/2+5\lambda+\kappa/2} |\mu_3|_*^{-2+2\lambda} |k_{[1,3]}|_*^{-1/2-\kappa/2}.
    \end{align*}
    Thus, we have
    \begin{align*}
        |(Q_0^{\CD,1} - Q_0^{\epsilon,\CD,1})(\mu_{1,3}) |
        &\lesssim \epsilon^{\lambda/2}  |\mu_1 |_{*}^{-2+2\lambda} |\mu_3 |_{*}^{-5/2+5\lambda+\kappa/2} |k_{[1,3]}|_*^{-1/2-\kappa/2} \\
        &\quad  + \epsilon^{\lambda/2} |\mu_1|_*^{-5/2+5\lambda+\kappa/2} |\mu_3|_*^{-2+2\lambda} |k_{[1,3]}|_*^{-1/2-\kappa/2}.
    \end{align*}
    Then, it follows
    \begin{align*}
        \int_{\mu_{[1,3]}=\mu} |(Q_0^{\CD,1} - Q_0^{\epsilon,\CD,1})(\mu_{1,3}) |^2
        &\lesssim  \epsilon^\lambda |\mu|_*^{-4+14\lambda+\kappa} |k|_*^{-1-\kappa}
    \end{align*}
    from Lemma \ref{lem:useful_estimates_1}
\end{proof}

\begin{prop}
    \label{prop:Convergence_of_D_3}
    We have
    \[
        c_\epsilon = -2\times \sum_{k_1,k_2 \in \Z^3}\frac{1}{4\alpha_\epsilon(k_1)\alpha_\epsilon(k_2)\alpha_\epsilon(k_{[1,2]})(\alpha_\epsilon(k_1)+\alpha_\epsilon(k_2)+\alpha_\epsilon(k_{[1,2]}))}.
    \]
    Also, there exists a limit of $c_\epsilon$ as $\epsilon$ tends to 0.
\end{prop}
\begin{proof}
    We have
    \begin{align*}
        &2f^{(\nabla \ICC_\epsilon,\nabla \ICC_\epsilon,2)}+2\epsilon f^{(\Delta \ICC_\epsilon,\Delta \ICC_\epsilon,2)}\\
        &=  2 \int_{E^2} (4 \pi^2|k_{[1,2]}|^2 + \epsilon 16 \pi^4 |k_{[1,2]}|^4) (H^{\epsilon,\ICC}_t (m_{1,2}))^2 dm_{1,2} \\
        &= 2 \sum_{k_1,k_2\in \Z^3} (4 \pi^2|k_{[1,2]}|^2 + \epsilon 16 \pi^4 |k_{[1,2]}|^4) \int_{\R^2}  (H^{\epsilon,\ICC}_t (m_{1,2}))^2 ds_{1,2}.
    \end{align*}
    From Fubini theorem, we have
    \begin{align*}
        &\int_{\R^2}  (H^{\epsilon,\ICC}_t (m_{1,2}))^2 ds_{1,2}\\
        &= \int_{\R^2} \left( \int_{\R} H^{\epsilon}_t(u, k_{[1,2]}) H^{\epsilon, \XX}_{u}(m_{1,2}) d u \right)^2ds_{1,2} \\
        &= \int_{\R^2} \prod_{j=1}^2 \left( \int_{\R} H^{\epsilon}_{u_1}(m_j) H^{\epsilon}_{u_2}(m_j) ds_j \right) H^{\epsilon}_t(u_1, k_{[1,2]})H^{\epsilon}_t(u_2, k_{[1,2]}) du_1du_2.
    \end{align*}
    Since
    \begin{align*}
        \int_{\R} H^{\epsilon}_{u_1}(m_j) H^{\epsilon}_{u_2}(m_j) ds_j
        &= \int_{-\infty}^{u_1\wedge u_2} e^{-(u_1+u_2-2s_j)\alpha_\epsilon(k_j)} ds_j
        = \frac{e^{-|u_1-u_2|\alpha_\epsilon(k_j)}}{2\alpha_\epsilon(k_j)},
    \end{align*}
    we have
    \begin{align*}
        &\int_{\R^2}  (H^{\epsilon,\ICC}_t (m_{1,2}))^2 ds_{1,2}\\
        &= \frac{1}{4\alpha_\epsilon(k_1)\alpha_\epsilon(k_2)} \\
        &\qquad \times \int_{-\infty}^{t}\int_{-\infty}^{t} e^{-|u_1-u_2|(\alpha_\epsilon(k_1) + \alpha_\epsilon(k_2))} e^{-(t-u_1)\alpha_\epsilon(k_{[1,2]})}e^{-(t-u_2)\alpha_\epsilon(k_{[1,2]})} du_1du_2\\
        &= \frac{1}{4\alpha_\epsilon(k_1)\alpha_\epsilon(k_2)\alpha_\epsilon(k_{[1,2]})(\alpha_\epsilon(k_1)+\alpha_\epsilon(k_2)+\alpha_\epsilon(k_{[1,2]}))}.
    \end{align*}
    Then, it follows
    \begin{align*}
        &2f^{(\nabla \ICC_\epsilon,\nabla \ICC_\epsilon,2)}+2\epsilon f^{(\Delta \ICC_\epsilon,\Delta \ICC_\epsilon,2)}\\
        &= 2 \sum_{k_1,k_2\in \Z^3} \frac{4 \pi^2|k_{[1,2]}|^2 + \epsilon 16 \pi^4 |k_{[1,2]}|^4}{4\alpha_\epsilon(k_1)\alpha_\epsilon(k_2)\alpha_\epsilon(k_{[1,2]})(\alpha_\epsilon(k_1)+\alpha_\epsilon(k_2)+\alpha_\epsilon(k_{[1,2]}))}.
    \end{align*}
    On the other hands, $b_\epsilon$ can be expressed as follows:
    \begin{align*}
        b_\epsilon
        &= 2 \sum_{k_1,k_2\in \Z^3} \frac{1}{4\alpha_\epsilon(k_1)\alpha_\epsilon(k_2)(\alpha_\epsilon(k_1)+\alpha_\epsilon(k_2)+\alpha_\epsilon(k_{[1,2]}))}\\
        &= 2 \sum_{k_1,k_2\in \Z^3} \frac{4\pi^2|k_{[1,2]} |^2 + \epsilon 16\pi^4|k_{[1,2]} |^4 + 1}{4\alpha_\epsilon(k_1)\alpha_\epsilon(k_2)\alpha_\epsilon(k_{[1,2]})(\alpha_\epsilon(k_1)+\alpha_\epsilon(k_2)+\alpha_\epsilon(k_{[1,2]}))}.
    \end{align*}
    Then, we have
    \begin{align*}
        c_\epsilon
        = -2 \sum_{k_1,k_2\in \Z^3} \frac{1}{4\alpha_\epsilon(k_1)\alpha_\epsilon(k_2)\alpha_\epsilon(k_{[1,2]})(\alpha_\epsilon(k_1)+\alpha_\epsilon(k_2)+\alpha_\epsilon(k_{[1,2]}))}.
    \end{align*}
    By monotone convergence theorem, we have
    \[
        c:=\lim_{\epsilon \downarrow 0} c_\epsilon 
        = -2 \sum_{k_1,k_2\in \Z^3} \frac{1}{4\alpha_0(k_1)\alpha_0(k_2)\alpha_0(k_{[1,2]})(\alpha_0(k_1)+\alpha_0(k_2)+\alpha_0(k_{[1,2]}))}.
    \]
    Moreover, it holds
    \begin{align*}
        c
        & \lesssim \sum_{k_1,k_2\in \Z^3} \frac{1}{|k_1|_*^2|k_2|_*^2|k_{[1,2]}|_*^4}
        \lesssim \sum_{k_{1}\in\Z^3} \frac{1}{|k_{1}|_*^5} < \infty
    \end{align*}
    by a similar convolution estimate to Lemma \ref{lem:useful_estimates_1}.
\end{proof}

\appendix
\section{Useful estimates}

The notations for $\Delta_j$, $\Delta_{<k}$, $\CC^\alpha$, $\CC_\epsilon^{\alpha}$, $\pl$, $\rs$, $e^{t\Delta}$, $e^{-\epsilon t \Delta^2}$, $P_t^\epsilon$ and $\CL_\epsilon^{-1}$ are at the begining of Section \ref{sec:2}.
\subsection{Estimates of paraproducts}
\begin{prop}[Paraproduct estimates (\cite{BCD11} Theorem 2.82 and 2.85)]
    \label{prop:paraproduct_estimates}
    We have the following:
    \begin{enumerate}[(1)]
        \item For every $\beta \in \R$, $\|u\pl v\|_{\CC^\beta} \lesssim \|u\|_{L^\infty}\|v\|_{\CC^\beta}$.
        \item For every $\alpha <0$ and $\beta \in \R$, $\|u\pl v\|_{\CC^{\alpha+\beta}}\lesssim \|u\|_{\CC^\alpha} \|v\|_{\CC^\beta}$.
        \item For every $\alpha, \beta \in \R$ such that $\alpha + \beta >0$, $\|u\rs v\|_{\CC^{\alpha+\beta}}\lesssim \|u\|_{\CC^\alpha} \|v\|_{\CC^\beta}$.
    \end{enumerate}
\end{prop}

\begin{prop}[Commutator estimates (\cite{GIP15} Lemma 2.4)]
    \label{prop:commutator_estimates}
    Assume that $0< \alpha<1$ and $\beta, \gamma \in \R$ are such that $\alpha + \beta + \gamma > 0$ and $\beta + \gamma <0$.
    Then there exists a bounded trilinear operator $C$ from $\CC^\alpha \times \CC^\beta \times \CC^\gamma$ to $\CC^{\alpha + \beta + \gamma}$ such that
    \[
        C(f,g,h)=(f\pl g) \rs h - f(g\rs h)
    \]  
    for every $f, g, h \in \CS$.
\end{prop}

\begin{prop}[Paralinearization theorem (\cite{BCD11} Theorem 2.92 or \cite{GIP15} Lemma 2.6)]
    \label{prop:paralinearization}
    For $0<\alpha<1$, $F\in C_b^2(\R)$ and $f\in \CC^\alpha$, we have
    \[
        \|F(f)-F'(f)\pl f\|_{\CC^{2\alpha}} \lesssim \|F\|_{C_b^{2}}(1+\|f\|_{\CC^\alpha}^2).
    \]
\end{prop}
Here, for $\alpha>0$, $C_b^\alpha(\R)$ is the space of $\lfloor \alpha \rfloor$ times continuously differentiable functions, bounded with bounded partial derivatives and with $(\alpha - \lfloor \alpha \rfloor)$-H\"{o}lder continuous partial derivatives of order $\lfloor \alpha \rfloor$, equipped with its usual norm $\|\cdot\|_{C_b^\alpha}$.

\begin{prop}[Bony's paramultiplication bound (\cite{Bon81} Theorem 2.3)]
    \label{prop:Bonys_paramultiplication}
    Let $\alpha >0$ be non-integer and $\beta \in \R$.
    For $f,g \in \CC^\alpha$ and $h \in \CC^\beta$, we have
    \[
        \|f\pl (g\pl h)\|_{\CC^{\alpha + \beta}} \lesssim \|f\|_{\CC^\alpha }\|g\|_{\CC^\alpha} \|h\|_{\CC^\beta}.
    \]
\end{prop}

\subsection{Basic estimates of heat semigroup}
\begin{prop}[Smoothing effects of heat semigroup]
    \label{prop:Effects_of_heat_semigroup}
    For all $T>0$, $0< t \leq T$, $\epsilon >0$, $\alpha \in \R$, $\delta  \geq 0$ and $u \in \CS'$ we have
    \[
        \|e^{t \Delta} u \|_{\CC^{\alpha + \delta}} \lesssim_T t^{-\delta/2} \|u\|_{\CC^{\alpha}}, \quad 
        \|e^{t \Delta} u \|_{\CC^{\delta}} \lesssim_T t^{-\delta/2} \|u\|_{L^\infty}
    \]
    and 
    \[
        \|e^{-\epsilon t \Delta^2} u \|_{\CC^{\alpha + \delta}} \lesssim_T (\epsilon t)^{-\delta/4} \|u\|_{\CC^{\alpha}}, \quad 
        \|e^{-\epsilon t \Delta^2} u \|_{\CC^{\delta}} \lesssim_T (\epsilon t)^{-\delta/4} \|u\|_{L^\infty}.
    \]
\end{prop}
\begin{proof}
    See \cite{GIP15} Lemma A.7 for the estimates of $e^{t\Delta}$.
    For $e^{-\epsilon t \Delta^2}$, by taking $\widetilde{\varphi}((\epsilon t)^{1/4} \mathrm{D})$ with $\widetilde{\varphi}(x) := e^{-16 \pi^4|x|^4}$, instead of $\varphi(t^{1/2} \mathrm{D})$ with $\varphi(x) := e^{-4\pi^2|x|^2}$, we can prove the estimates in same way.
\end{proof}
\begin{prop}[Continuity of heat semigroup]
    \label{prop:Effects_of_heat_semigroup_2}
    For all $t>0$, $\epsilon>0$, $0< \beta<1$ and $u \in \CS'$, we have
    \[
        \|(e^{t\Delta}-\mathrm{Id})u\|_{L^\infty} \lesssim t^{\frac{\beta}{2}} \|u\|_{\CC^\beta}
    \]  
    and
    \[
        \|(e^{-\epsilon t\Delta^2}-\mathrm{Id})u\|_{L^\infty} \lesssim (\epsilon t)^{\frac{\beta}{4}} \|u\|_{\CC^\beta}.
    \]
\end{prop}
\begin{proof}
    For the estimates of $e^{t\Delta}$, see \cite{GIP15} Lemma A.8.
    For the estimates of $e^{-\epsilon t\Delta^2}$, we can prove in the same way by taking the same $\tilde{\varphi}$ as in Proposition \ref{prop:Effects_of_heat_semigroup}.
\end{proof}
Fix $\kappa>0$.
\begin{prop}[Schauder estimates (\cite{GIP15} Lemma A.9)]
    \label{prop:Schauder_estimates}
    Let $T>0$, $0\leq \epsilon\leq 1$, $\alpha \in \R$ and $u \in C_T\CC^{\alpha}$.
    Then, we have
    \begin{equation}
        \label{eq:Sch_result}
        \|\CL_\epsilon^{-1}[u]\|_{C_T\CC_\epsilon^{\alpha+2}}\lesssim \|u\|_{C_T\CC^\alpha}.
    \end{equation}
    If $-2 < \alpha<0$, we also have
    \begin{equation}
        \label{eq:Sch_result_2}
        \|\CL_\epsilon^{-1}[u]\|_{C_T^{\frac{\alpha+2}{2}}L^\infty}\lesssim \|u\|_{C_T\CC^\alpha}.
    \end{equation}
\end{prop}
\begin{proof}
    From \cite{GIP15} Lemma A.9, we have
    \[
        \|\CL_\epsilon^{-1}[u](t)\|_{C_T\CC^{\alpha+2}}\lesssim \|u\|_{C_T\CC^\alpha}
    \]
    for every $0\leq t\leq T$.
    If $\epsilon>0$, by using Proposition \ref{prop:Effects_of_heat_semigroup}, we have
    \begin{align*}
        \|\CL_\epsilon^{-1}[u](t)\|_{\CC^{\alpha+4-\kappa}}
        &\leq \int_0^t \|P_{t-s}^\epsilon u (s)\|_{\CC^{\alpha+4-\kappa}} ds \\
        &\lesssim \int_0^t (\epsilon(t-s))^{-1+\frac{\kappa}{4}} \|u(s)\|_{\CC^{\alpha}} ds\\
        &\lesssim \epsilon^{-1+\frac{\kappa}{4}} T^{\frac{\kappa}{4}}\|u\|_{\CC^\alpha}.
    \end{align*}
    Thus, we have
    \[
        \epsilon^{1-\frac{\kappa}{4}} \|\CL_\epsilon^{-1}[u](t) \|_{\CC^{\alpha+4-\kappa}} \lesssim \|u\|_{\CC^\alpha}.
    \]
    Therefore, we have (\ref{eq:Sch_result}).
    For (\ref{eq:Sch_result_2}), see \cite{GIP15} Lemma A.9.
\end{proof}

\begin{prop}
    \label{prop:Schauder_estimates_2}
    Let $\kappa>0$, $0 \leq \epsilon\leq 1$, $T>0$ and $\alpha \in \R$.
    If $u(t)\in \CC^\alpha$ for any $0 < t \leq T$, then we have
        \begin{align}
            \sup_{0<t\leq T}\|\CL_\epsilon^{-1}[u](t)\|_{\CC^{\alpha}}
            &\lesssim T^{1-\eta}\left( \sup_{0<t\leq T} t^{\eta}\|u(t)\|_{\CC^{\alpha}} \right) \label{eq:Schauder_2_t_unif}
        \end{align}
        and
        \begin{align}
            \sup_{0<t\leq T} t^\eta \|\CL_\epsilon^{-1}[u](t)\|_{\CC_\epsilon^{\alpha+2-\kappa}}
            \lesssim T^{\frac{\kappa}{2}}\left( \sup_{0<t\leq T} t^{\eta}\|u(t)\|_{\CC^{\alpha}} \right).\label{eq:Schauder_2_t_depend}
        \end{align}
        for any $0\leq \eta < 1$.
\end{prop}
\begin{proof}
        Fix $0<t\leq T$ and $0\leq \epsilon\leq 1$.
        First, we find the $\epsilon$-uniform estimates of $\CL_\epsilon^{-1}[u]$.
        By using smoothing effects of heat semigroup (Proposition \ref{prop:Effects_of_heat_semigroup}), we have
        \begin{align}
            \|\CL_\epsilon^{-1}[u](t)\|_{\CC^{\alpha+\delta}}
            &\leq \int_0^t \|P_{t-s}^\epsilon u(s)\|_{\CC^{\alpha+\delta}} ds \notag \\
            &\lesssim \int_0^t (t-s)^{-\frac{\delta}{2}} \|u(s)\|_{\CC^\alpha} ds \notag \\
            &\leq \left( \sup_{0<t\leq T} t^{\eta} \|u(t)\|_{\CC^{\alpha}} \right) \int_0^t (t-s)^{-\frac{\delta}{2}} s^{-\eta} ds \notag \\
            &\lesssim t^{1-\frac{\delta}{2}-\eta}\left( \sup_{0<t\leq T} t^{\eta}\|u(t)\|_{\CC^{\alpha}} \right) \label{eq:pr_Sch_1}
        \end{align}
        for every $0\leq \delta < 2$ and $0\leq \eta< 1$.
        
        Next, we find the $\epsilon$-depend estimates of $\CL_\epsilon^{-1}[u]$.
        If $\epsilon>0$, by using smoothing effects of heat semigroup (Proposition \ref{prop:Effects_of_heat_semigroup}), we have
        \begin{align}
            &\|\CL_\epsilon^{-1}[u](t)\|_{\CC^{\alpha+\delta'+2-\kappa}} \notag \\
            &\leq \int_0^t \|P_{t-s}^\epsilon u(s)\|_{\CC^{\alpha+4-2\kappa}} ds \notag \\
            &\lesssim \int_0^t (\epsilon(t-s))^{-\frac{4-2\kappa}{4}} \|u(s)\|_{\CC^\alpha} ds \notag \\
            &\leq \epsilon^{-1+\frac{\kappa}{2}}\left( \sup_{0<t\leq T} t^{\eta} \|u(t)\|_{\CC^{\alpha}} \right) \int_0^t (t-s)^{-1+\frac{\kappa}{2}} s^{-\eta} ds \notag \\
            &\lesssim \epsilon^{-1+\frac{\kappa}{4}}t^{\frac{\kappa}{2}-\eta}\left( \sup_{0<t\leq T} t^{\eta}\|u(t)\|_{\CC^{\alpha}} \right)  \label{eq:pr_Sch_2}
        \end{align}
        for every $0\leq \eta< 1$.

        Taking $\delta=0$ for the equation (\ref{eq:pr_Sch_1}), we have (\ref{eq:Schauder_2_t_unif}).
        On the other hand, taking $\delta = 2-\kappa$ for the equation (\ref{eq:pr_Sch_1}) and using the equation (\ref{eq:pr_Sch_2}), we have (\ref{eq:Schauder_2_t_depend}).
\end{proof}

\begin{lem}[Commutation between heat semigroup and paraproduct]
    \label{lem:commutation_heat_semigroup_paraproduct}
    Let $\alpha < 1$, $\beta \in \R$, $\delta > 0$.
    For $t>0$ and $\epsilon > 0$, we have
    \[
        \|e^{t\Delta}(f\pl g)-f\pl(e^{t\Delta}g)\|_{\CC^{\alpha+\beta +\delta}} \lesssim t^{-\frac{\delta}{2}}\|f\|_{\CC^{\alpha}}\|g\|_{\CC^{\beta}}
    \]
    and 
    \[
        \|e^{-\epsilon t\Delta^2}(f\pl g)-f\pl(e^{-\epsilon t\Delta^2}g)\|_{\CC^{\alpha+\beta +\delta}} \lesssim  (\epsilon t)^{-\frac{\delta}{4}}\|f\|_{\CC^{\alpha}}\|g\|_{\CC^{\beta}}.
    \]
\end{lem}
\begin{proof}
    See \cite{MW17b} Proposition A.16.
    The key part of this proof for the estimates of $e^{t\Delta}$ is 
    \[
        \|h_k\|_{L^p}\lesssim \|\tilde{G}_{k,t}\|_{L^1}\|\nabla (\Delta_{<k-1} f)\|_{L^{p_1}} \|\Delta_k g\|_{L^{p_2}},
    \]
    where we set
    \begin{align*}
        h_k &:= G_{k,t}*(\Delta_{<k-1} f \Delta_k g) - (\Delta_{<k-1} f) (G_{k,t}*\Delta_k g), \\
        G_{k,t}&:= \CF^{-1}\left( \phi(2^{-k}\cdot) e^{-t|\cdot|^2} \right), \\
        \tilde{G}_{k,t}(y)&:= yG_{k,t}(y)
    \end{align*}
    and $\phi \in C_c^\infty$ is an appropriate function.

    For the estimates of $e^{-\epsilon t \Delta^2}$, take
    \begin{align*}
        G_{k,\epsilon,t} &:=\CF^{-1}\left(\phi(2^{-k}\cdot) e^{-16\pi^4 \epsilon t|\cdot|^4}\right),\\
        \tilde{G}_{k,\epsilon,t}(y)&:= yG_{k,\epsilon,t}(y)
    \end{align*}
    instead of $G_{k,t}$ and $\tilde{G}_{k,t}$, respectively.
    Since it holds
    \[
        \|\tilde{G}_{k,\epsilon,t} \|_{L^1}\lesssim 2^{-k}(1+2^{4k}\epsilon t)e^{-c\epsilon t 2^{4k}},
    \]
    we can show the estimates of $e^{-\epsilon t \Delta^2}$.
\end{proof}

\begin{prop}\label{prop:commutation_L_inverse_paraproduct}
    For $0 \leq \epsilon \leq 1$, $0< \alpha < 1$, $\beta \in \R$ and $\delta < 2$, we have
    \[
        \|\CL_\epsilon^{-1}[f\pl g] - f \pl \CL_\epsilon^{-1}[g]\|_{C_T\CC^{\alpha + \beta + \delta}}
        \lesssim (\|f\|_{C_T\CC^{\alpha}} + \|f\|_{C_T^{\alpha/2}L^\infty})\|g\|_{C_T\CC^{\beta}}.
    \]
\end{prop}
\begin{proof}
    It follows from Proposition \ref{prop:Effects_of_heat_semigroup}, Lemma \ref{lem:commutation_heat_semigroup_paraproduct} and paraproduct estimates (Proposition \ref{prop:paraproduct_estimates}).
\end{proof}

\section{The derivation of the equation (\ref{eq:L_epsilon_v})}
In this section, we use the following differential operators
\[
    L_\epsilon := \Delta -\epsilon \Delta^2, \quad \CL_\epsilon:= \pt - L_\epsilon +1
\]
for $0\leq \epsilon \leq 1$.
We will express $\CL_\epsilon v$ using $\CL_\epsilon g$ when $v=e^h g$.
\subsection{The expansion of $L_0 v$}
First, we consider the simple case ($\epsilon = 0$).
We denote
\[
    B(f,g) := \nabla f\cdot \nabla g.
\]
For $L_0$, the following relations hold.
\begin{prop}
    We have
    \begin{align}
        L_0(fg)
        & = (L_0 f) g + f (L_0 g) + 2B(f,g) \label{eq:L0(fg)},\\
        L_0(e^h) &= e^h \{L_0 h +B(h,h)\} \label{eq:L0(eh)},\\
        L_0(e^h g) &= e^h\{(L_0 h)g + L_0 g +  B(h,h)g + 2B(h,g) \} \label{eq:L0(ehg)}.\\
        B(h, e^{k}v) &=e^k\{vB(h,k)+B(h,v)\} \label{eq:B0(ehv_g)}.
    \end{align}
\end{prop}
\begin{proof}
    By Leibniz rule, we have (\ref{eq:L0(fg)}), (\ref{eq:L0(eh)}), (\ref{eq:B0(ehv_g)}), and 
    \begin{align}
        B(e^h, g)
        = e^hB(h,g). \label{eq:B0(eh_g)}
    \end{align}
    It holds (\ref{eq:L0(ehg)}) from (\ref{eq:L0(fg)}), (\ref{eq:L0(eh)}) and (\ref{eq:B0(eh_g)}).
\end{proof}
\begin{cor}
    We have
    \begin{align}
        L_0v = (L_0 h)v + e^h(L_0 g) + \{-B(h,h)v + 2B(h,v)\}. \label{eq:L0(v)}
    \end{align}
\end{cor}

\subsection{The expansion of $L_\epsilon v$}
Next, we consider the case $L_\epsilon$.
We denote
\begin{align*}
    \widetilde{B}(f,g) &:= \nabla\cdot(\Delta f \nabla g) + \nabla\cdot(\nabla f \Delta g) + \Delta B(f,g),\\
    T(f,g,h)&:= \nabla\cdot(B(f,g) \nabla h)=B(f,g)\Delta h + B(B(f,g),h).
\end{align*}
For $T$, the following relations hold:
\begin{align}
    T(fg,h,k)
    &= B(fg,h)\Delta k + B(B(fg,h),k) \notag \\
    &= \{fB(g,h)+gB(f,h)\} \Delta k + B(fB(g,h)+gB(f,h),k) \notag \\
    &= \{fB(g,h)+gB(f,h)\} \Delta k + fB(B(g,h),k)+B(f,k)B(g,h) \notag \\
    &\quad  + gB(B(f,h),k) + B(g,k)B(f,h) \notag \\
    &= f T(g,h,k) + g T(f,h,k) + B(f,k)B(g,h) + B(g,k)B(f,h) \label{eq:T(fg,h,k)}, \\
    T(f,g,hk)
    &= B(f,g)\Delta (hk) + B(B(f,g),hk) \notag \\
    &= B(f,g)\{(\Delta h)k + h(\Delta k) + 2B(h,k)\} \notag \\
    &\quad + B(B(f,g),h)k + B(B(f,g),k)h \notag \\
    &= hT(f,g,k) + kT(f,g,h) + 2B(f,g)B(h,k) \label{eq:T(f,g,hk)}.
\end{align}
Since the computation is a little long, we divide it into three propositions.
\begin{prop}
    We have
    \begin{align}
        L_\epsilon (fg)
        &= (L_\epsilon f) g + f(L_\epsilon g) + 2A^1_\epsilon(f,g) \label{eq:L_epsilon(fg)},\\
        L_\epsilon(e^h) &= e^h\{L_\epsilon h + A^2_\epsilon(h)\} \label{eq:L_epsilon(eh)}, \\
        L_\epsilon(e^h g)&= e^h\{(L_\epsilon h)g + L_\epsilon g + 2A^1_\epsilon(h,g) + A^2_\epsilon(h)g + 2A^3_\epsilon(h,g) \} \label{eq:L_epsilon(ehg)},
    \end{align}
    where
    \begin{align*}
        A^1_\epsilon(f,g) &:=  B(f,g)\\
        &\quad -\epsilon\{B(\Delta f,g) + B(f,\Delta g) + (\Delta f)(\Delta g) + \Delta(B(f,g))  \}\\
        &\,\,  = (\nabla f \cdot \nabla g + \epsilon\Delta f \Delta g) - \epsilon \widetilde{B}(f,g),\\
        A^2_\epsilon(h)
        &:= A^1_\epsilon(h,h) -\epsilon \{2T(h,h,h) + B(h,h)^2\},\\
        A^3_\epsilon(h,g)
        &:= -\epsilon \{ 2T(g,h,h) + T(h,h,g) + 2B(h,h)B(h,g) \}.
    \end{align*}
\end{prop}
\begin{remark}
    By definition, $A^2_\epsilon$ and $A^3_\epsilon$ are nonlinear with respect to $h$ and $A^3_\epsilon$ is not symmetric.
\end{remark}

\begin{proof}
    First, we show (\ref{eq:L_epsilon(fg)}).
    By using (\ref{eq:L0(fg)}) twice, we have
    \begin{align}
        \Delta^2 (fg)
        &= \Delta \{(\Delta f)g + f(\Delta g) + 2B(f,g)\} \notag \\
        &= (\Delta^2 f)g + f(\Delta^2 g) \notag \\
        &\quad + 2\{B(\Delta f, g) + B(f,\Delta g) + (\Delta f)(\Delta g) + \Delta (B(f,g))\}\label{eq:Delta2(fg)},
    \end{align}
    therefore we obtain (\ref{eq:L_epsilon(fg)}) as follows.
    \begin{align*}
        L_\epsilon (fg) 
        &= L_0(fg)-\epsilon\Delta^2(fg)\\
        &=\{(L_0 f) g + f (L_0 g) + 2B(f,g)\} - \epsilon\{ (\Delta^2 f) g + f (\Delta^2 g) \} \\
        &\quad - 2\epsilon\{B(\Delta f, g) + B(f,\Delta g) + (\Delta f)(\Delta g) + \Delta (B(f,g))\}\\
        &=(L_\epsilon f) g + f(L_\epsilon g) + 2A^1_\epsilon(f,g).
    \end{align*}

    Next, we show (\ref{eq:L_epsilon(eh)}).
    From (\ref{eq:L0(fg)}) and (\ref{eq:L0(eh)}), we have
    \begin{align*}
        \Delta^2 e^h
        &= \Delta (e^h \Delta h + e^h B(h,h)) \\
        &= (\Delta e^h) \Delta h + e^h \Delta^2 h + 2 B(e^h,\Delta h)\\
        &\quad  + (\Delta e^h) B(h,h) + e^h \Delta B(h,h) + 2B(e^h,B(h,h)) \\
        &= e^h \{ \Delta^2 h +(\Delta h)^2 + 2B(h,h)\Delta h + 2 B(\Delta h,h)\\
        &\qquad \qquad + B(h,h)^2 + \Delta B(h,h)+ 2 B(B(h,h),h)\} \\
        &= e^h\{ \Delta^2 h -\epsilon^{-1} (A_\epsilon^1(h,h)-B(h,h)) +2T(h,h,h) +B(h,h)^2 \}.
    \end{align*}
    Therefore, we obtain (\ref{eq:L_epsilon(eh)}) as follows.
    \begin{align*}
        L_\epsilon (e^h)
        &= L_0(e^h) - \epsilon \Delta^2 e^h \\
        &= e^h\{(L_0 h) + B(h,h)\}\\
        &\quad - \epsilon e^h\{ \Delta^2 h -\epsilon^{-1} (A_\epsilon^1(h,h)-B(h,h)) +2T(h,h,h) +B(h,h)^2 \}\\
        &=e^h L_\epsilon h + e^h A^2_\epsilon(h).
    \end{align*}

    At last, we show (\ref{eq:L_epsilon(ehg)}).
    From (\ref{eq:L_epsilon(fg)}) and (\ref{eq:L_epsilon(eh)}), we have
    \begin{align*}
        L_\epsilon (e^hg)
        &= (L_\epsilon e^h) g + e^h(L_\epsilon g) + 2A^1_\epsilon(e^h,g) \\
        &= e^h\{ (L_\epsilon h)g + L_\epsilon g + A^2_\epsilon(h)g \}+ 2A^1_\epsilon(e^h,g).
    \end{align*}
    It is sufficient that we consider $A^1_\epsilon(e^h,g)$.
    We derive the relations between the derivatives of $e^h$ and $g$. 
    From (\ref{eq:L0(fg)}) and (\ref{eq:L0(eh)}), we have
    \begin{align*}
        B(\Delta e^h,g)
        &= B(e^h\Delta h+e^hB(h,h),g) \\
        &= e^h\{B(\Delta h,g) + (\Delta h) B(h,g) + B(h,h)B(h,g) + B(B(h,h),g)\},\\
        B(e^h,\Delta g)
        &= e^h B(h,g), \\
        (\Delta e^h)(\Delta g)
        &= e^h\{ (\Delta h)(\Delta g) + B(h,h)\Delta g\}, \\
        \Delta(B(e^h,g))
        &= \Delta(e^h B(h,g)) \\
        &= (\Delta e^h)B(h,g) + e^h \Delta B(h,g) + 2B(e^h,B(h,g)) \\
        &= e^h\{ (\Delta h)B(h,g) + B(h,h)B(h,g) + \Delta B(h,g) + 2B(h,B(h,g))\}\\
        &= e^h\{ \Delta B(h,g) + T(g,h,h) + B(h,h)B(h,g) + B(B(h,g),h) \}.
    \end{align*}
    Hence, $A^1_\epsilon(e^h,g)$ can be expressed as follows.
    \begin{align}
        A^1_\epsilon(e^h,g)
        &= B(e^h,g) \notag \\
        &\quad -\epsilon\{B(\Delta e^h,g) + B(e^h,\Delta g) + (\Delta e^h)(\Delta g) + \Delta(B(e^h,g))\} \notag \\
        &= e^h[ A^1_\epsilon(h,g) -\epsilon \{ (\Delta h) B(h,g) + 2B(h,h)B(h,g) + B(B(h,h),g) \notag \\
        &\quad \qquad + B(h,h)\Delta g + T(g,h,h)  + B(B(h,g),h)\}] \notag \\
        &=e^h[A^1_\epsilon(h,g) -\epsilon \{ 2T(g,h,h) + T(h,h,g)+ 2B(h,h)B(h,g)\}] \notag \\
        &= e^h\{A^1_\epsilon(h,g) + A^3_\epsilon(h,g)\} \label{eq:B_epsilon(eh_g)} .
    \end{align}
\end{proof}

\begin{prop}\label{prop:B_4}
    We have
    \begin{multline}\label{eq:B_epsilon(ehv_g)}
        A^1_\epsilon(h, e^{k}v)
        = e^k[ v\{A^1_\epsilon(h,k)+A^3_\epsilon(k,h)\} + A^1_\epsilon(h,v) \\
        -\epsilon \{ 2T(h,k,v)+ 2T(k,v,h) + 2T(h,v,k)\\
        + 4B(k,h)B(k,v) + 2B(k,k)B(h,v)\}],
    \end{multline}
    \begin{multline}\label{eq:D_epsilon(ehv_g)}
        A^3_\epsilon(h,e^kv)
        = e^k[ v A^3_\epsilon(h,k) + A^3_\epsilon(h,v) \\
        - \epsilon \{v( 2 B(h,k)^2 + B(h,h)B(k,k)) \\
        + 4B(k,h)B(v,h)+ 2 B(h,h)B(k,v) \}].
    \end{multline}
\end{prop}
\begin{proof}
    First, we consider $A^1_\epsilon(h,fv)$.
    From (\ref{eq:L0(fg)}) and (\ref{eq:L0(eh)}), we have the relations between the derivatives of $h$ and $fv$.
    \begin{align*}
        B(\Delta h, fv)
        &= vB(\Delta h, f) + fB(\Delta h,v), \\
        B(h, \Delta (fv))
        &= B(h, (\Delta f)v + f(\Delta v) + 2B(f,v)) \\
        &= vB(h,\Delta f) + fB(h,\Delta v)\\
        &\quad + (\Delta f)B(h,v) + (\Delta v)B(h,v) + 2B(h,B(f,v)),\\
        (\Delta h)(\Delta (fv))
        &= v(\Delta h)(\Delta f) + f(\Delta h)(\Delta h) + 2(\Delta h)B(f,v), \\
        \Delta B(h,fv)
        &= \Delta (vB(h,f)+fB(h,v))\\
        &= v \Delta B(h,f) +(\Delta v )B(h,f)+ 2B(v,B(h,f))\\
        &\quad + f\Delta B(h,v) + (\Delta f)B(h,v)+ 2B(f,B(h,v))\\
        &= v \Delta B(h,f)+ f\Delta B(h,v) \\
        &\quad +T(h,f,v) + T(h,v,f) + B(v,B(h,f)) + B(f,B(h,v)).
    \end{align*}
    Hence, we have
    \begin{align*}
        &A^1_\epsilon(h,fv)\\
        &=v A^1_\epsilon(h,f) + fA^1_\epsilon(h,v)\\
        &\quad -\epsilon \{ (\Delta f)B(h,v) + (\Delta v)B(h,v) + 2B(h,B(f,v)) + 2(\Delta h)B(f,v)\\
        &\qquad \quad  + T(h,f,v) + T(h,v,f) + B(v,B(h,f)) + B(f,B(h,v))\}\\
        &= v A^1_\epsilon(h,f) + fA^1_\epsilon(h,v) -2\epsilon \{ T(h,f,v)+T(f,v,k)+T(h,v,f) \}.
    \end{align*}
    Take $f=e^k$.
    From the definition of $T$ and (\ref{eq:B0(eh_g)}), we have
    \begin{align*}
        T(h,e^k,v)&= e^k \{T(h,k,v) + B(k,h)B(k,v)\}, \\
        T(e^k,v,h)&= e^k \{T(k,v,h) + B(k,h)B(k,v)\}, \\
        T(h,v,e^k)&= e^k \{T(h,v,k) + B(h,v)B(k,k)\}.
    \end{align*}
    From these and (\ref{eq:B_epsilon(eh_g)}), we obtain (\ref{eq:B_epsilon(ehv_g)}).

    Next, We consider $A^3_\epsilon(h, e^kv)$.
    From (\ref{eq:T(fg,h,k)}), (\ref{eq:T(f,g,hk)}) and above relations of $T$, we have
    \begin{align*}
        T(e^kv,h,h)
        &= vT(e^k,h,h)+e^kT(v,h,h) + 2B(e^k,h)B(v,h)\\
        &= e^k\{v T(k,h,h) + T(v,h,h) + v B(k,h)^2+ 2B(k,h)B(v,h)\}, \\
        T(h,h,e^k v)
        &= vT(h,h,e^k) + e^kT(h,h,v) + 2B(h,h)B(e^k,v) \\
        &= e^k\{v T(h,h,k) + T(h,h,v)\\
        &\quad + vB(h,h)B(k,k) + 2B(h,h)B(k,v) \}.
    \end{align*}
    By using these and (\ref{eq:B0(ehv_g)}), it follows (\ref{eq:D_epsilon(ehv_g)}) .
\end{proof}

\begin{prop}\label{prop:B_5}
    When $g=e^{-h}v$, from the above relations, we have
    \begin{align}
        L_\epsilon v = (L_\epsilon h)v + e^h(L_\epsilon g) - F_\epsilon(h)v + 2G_\epsilon(h,v) \label{eq:L_epsilon(v)},
    \end{align}
    where
    \begin{align*}
        F_\epsilon(h) &:= |\nabla h|^2+\epsilon(\Delta h)^2 + \epsilon \{- \widetilde{B}(h,h) + 2T(h,h,h) - B(h,h)^2\}, \\
        G_\epsilon(h,v) &:= \nabla h \cdot \nabla v +\epsilon\Delta h\Delta v\\
        &\quad + \epsilon \{-\widetilde{B}(h,v) + 2T(v,h,h) + T(h,h,v) - 2 B(h,h)B(h,v) \}.
    \end{align*}
\end{prop}

\begin{proof}
    Let $g=e^{-h}v$.
    From Proposition \ref{prop:B_4}, we have
    \begin{align*}
        A^1_\epsilon(h, e^{-h}v)
        &=e^{-h} [v\{-A^1_\epsilon(h,h)+A^3_\epsilon(-h,h)\}+ A^1_\epsilon(h,v) \\
        &\quad\qquad  -\epsilon \{ -4T(v,h,h) - 2T(h,h,v) + 6B(h,h)B(h,v)\}], \\
        A^3_\epsilon(h,e^{-h}v)
        &= e^{-h}[ vA^3_\epsilon(h,-h) + A^3_\epsilon(h,v) \\
        &\quad\qquad - \epsilon \{3 vB(h,h)^2 -6 B(h,h)B(h,v)\}].
    \end{align*}
    Note that
    \begin{align*}
        A^3_\epsilon(-h,h)
        &= -\epsilon \{3T(h,h,h) -2B(h,h)^2\},\\
        A^3_\epsilon(h,-h)
        &= -\epsilon\{ -3T(h,h,h) - 2B(h,h)^2 \},
    \end{align*}
    therefore we have 
    \begin{align*}
        A^3_\epsilon(-h,h)+A^3_\epsilon(h,-h) = 4\epsilon B(h,h)^2.    
    \end{align*}
    Hence, we have
    \begin{align*}
        &A^1_\epsilon(h,e^{-h}v)+A^3_\epsilon(h,e^{-h}v)\\
        &= e^{-h}[ v \{ -A^1_\epsilon(h,h) + A^3_\epsilon(-h,h) + A^3_\epsilon(h,-h) \} + A^1_\epsilon(h,v) + A^3_\epsilon(h,v) \\
        &\quad\qquad -\epsilon \{ -4T(v,h,h) -2T(h,h,v)+ 3v B(h,h)^2 \} ]\\
        &= e^{-h} [v \{ -A^1_\epsilon(h,h) + \epsilon B(h,h)^2 \}\\
        &\quad\qquad +  A^1_\epsilon(h,v) + \epsilon \{2T(v,h,h)+T(h,h,v) -2B(h,h)B(h,v) \}].
    \end{align*}
    From this and (\ref{eq:L_epsilon(ehg)}), we obtain (\ref{eq:L_epsilon(v)}) as follows.
    \begin{align*}
        L_\epsilon(v)
        &=L_\epsilon(e^hg)\\
        &=e^h(L_\epsilon h)g + e^h(L_\epsilon g) + e^h\{2A^1_\epsilon(h,g) + A^2_\epsilon(h)g + 2A^3_\epsilon(h,g) \}\\
        &= (L_\epsilon h) v + e^h(L_\epsilon g) + v\{ A^1_\epsilon(h,h) -\epsilon (2T(h,h,h) + B(h,h)^2)\} \\
        &\quad+ 2v \{ -A^1_\epsilon(h,h) + \epsilon B(h,h)^2 \}\\
        &\quad + 2 [ A^1_\epsilon(h,v) + \epsilon \{2T(v,h,h)+T(h,h,v) -2B(h,h)B(h,v) \}]\\
        &=(L_\epsilon h) v + e^h(L_\epsilon g)- v\{A^1_\epsilon(h,h) + \epsilon (2T(h,h,h) - B(h,h)^2) \}\\
        &\quad +2[A^1_\epsilon(h,v) + \epsilon\{2T(v,h,h) +T(h,h,v) - 2B(h,h)B(h,v)\}]\\
        &= (L_\epsilon h)v + e^h(L_\epsilon g) - F_\epsilon(h)v + 2G_\epsilon(h,v).
    \end{align*}
\end{proof}

\subsection{The expansion of $\CL_\epsilon v$}
\begin{prop}{\label{prop:derivation_of_equation}}
    When $g=e^{-h}v$, we have
    \begin{equation}
        \CL_\epsilon v=(\CL_\epsilon h) v -hv + e^h (\CL_\epsilon g) + F_\epsilon(h) v - 2G_\epsilon(h,v).
        \label{eq:CL_epsilon_v_relation}
    \end{equation}
\end{prop}
\begin{proof}
    By using Leibniz rule, we have
    \begin{align*}
        \pt v
        &= \pt (e^h g)
        = (\pt h)e^h g  + e^h \pt g
        = (\pt h)v + e^h \pt g.
    \end{align*}
    From Proposition \ref{prop:B_5}, we have (\ref{eq:CL_epsilon_v_relation}) as follows.
    \begin{align*}
        \CL_\epsilon v
        &= (\pt - L_\epsilon +1) v \\
        &= \{ (\pt h)v + e^h \pt g \} - \{ (L_\epsilon h)v + e^h (L_\epsilon g) - F_\epsilon(h)v + 2G_\epsilon(h,v) \} + v \\
        &= (\CL_\epsilon h) v -hv + e^h (\CL_\epsilon g) + F_\epsilon(h) v - 2G_\epsilon(h,v).
    \end{align*}
\end{proof}

\end{document}